%% file: main.tex
\newif\ifjournal
\journalfalse

\ifjournal
    \documentclass[moor,dblanonrev]{informs4}  %
\else
    \documentclass[moor,nonblindrev]{informs4} %
\fi

\RequirePackage{tgtermes}
\RequirePackage{newtxtext}
\RequirePackage{newtxmath}
\RequirePackage{bm}

\OneAndAHalfSpacedXI  %

\usepackage[sort&compress]{natbib}
 \bibpunct[, ]{[}{]}{,}{n}{}{,}%
 \def\bibfont{\small}%

\EquationsNumberedThrough

\makeatletter
\let\theoremstyle\@undefined
\let\newtheoremstyle\@undefined
\let\theorembodyfont\@undefined
\let\theoremheaderfont\@undefined
\let\theorempreskipamount\@undefined
\let\theorempostskipamount\@undefined
\let\proof\@undefined
\let\endproof\@undefined
\let\openbox\@undefined
\makeatother

\MANUSCRIPTNO{} %

\RequirePackage{amsthm,amsmath,amsfonts,amssymb}
\usepackage{xspace}
\usepackage{outlines}
\usepackage[noend]{algpseudocode}
\usepackage{algorithm}
\usepackage{xcolor}
\usepackage{subcaption}
\usepackage{microtype}
\usepackage{thmtools}
\usepackage{thm-restate}
\usepackage{soul}
\allowdisplaybreaks
\usepackage{multirow}
\usepackage{array}
\usepackage{wrapfig}
\usepackage{adjustbox}
\usepackage{booktabs}
\usepackage{nicefrac}
\usepackage{float}
\usepackage{tikz}
\usetikzlibrary{shapes.geometric, arrows}
\usepackage{etoolbox}
\usepackage{refcount}
\usepackage[colorlinks,citecolor=blue,urlcolor=blue,linkcolor=blue]{hyperref}
\usepackage{cleveref}

\makeatletter
\ifjournal\else
    \def\theARTICLETOPLEFT{}%
    \def\theARTICLETOPRIGHT{}%
    \RRHSecondLine{}\LRHSecondLine{}%
    \def\setoddRH{\hbox to \textwidth{\fs.7.8.\tabcolsep0pt
      \begin{tabular*}{\textwidth}[b]{l@{\extracolsep\fill}r}
      {\theRRHFirstLine}&\raisebox{0pt}[0pt][0pt]{\fs.10.10.\thepage}\\[-4pt]
      \rlap{\VRHDW{0.5pt}{0pt}{\textwidth}}&\\
      \end{tabular*}}}%
    \def\setevenRH{\hbox to \textwidth{\fs.7.8.\tabcolsep0pt
      \begin{tabular*}{\textwidth}[b]{l@{\extracolsep\fill}r}
      \raisebox{0pt}[0pt][0pt]{\fs.10.10.\thepage}&{\theLRHFirstLine}\\[-4pt]
      \rlap{\VRHDW{0.5pt}{0pt}{\textwidth}}&\\
      \end{tabular*}}}%
    \def\theARTICLEABSTRACT{%
        \HOOKb
        \vspace*{18pt}%
        \begin{minipage}[t]{\textwidth}\parindent1em
            \ABSfont
            \noindent\theABSTRACT\endgraf
            \vskip5pt
            \theFUNDING
            \theKEYWORDS
            \theSUBJECTCLASS
            \theAREAOFREVIEW
            \theMSCCLASS
            \theORMSCLASS
            \if@BLINDREV\else\theHISTORY\fi
            \noindent\hrulefill
        \end{minipage}%
        \vspace*{0pt}%
    }%
\fi
\makeatother

\ifjournal
    \ECRUNAUTHOR{}\ECAUpunct{}
\fi

\makeatletter
\newsavebox\myboxA
\newsavebox\myboxB
\newlength\mylenA
\newcommand*\xoverline[2][0.75]{%
    \sbox{\myboxA}{$\m@th#2$}%
    \setbox\myboxB\null%
    \ht\myboxB=\ht\myboxA%
    \dp\myboxB=\dp\myboxA%
    \wd\myboxB=#1\wd\myboxA%
    \sbox\myboxB{$\m@th\overline{\copy\myboxB}$}%
    \setlength\mylenA{\the\wd\myboxA}%
    \addtolength\mylenA{-\the\wd\myboxB}%
    \ifdim\wd\myboxB<\wd\myboxA%
       \rlap{\hskip 0.5\mylenA\usebox\myboxB}{\usebox\myboxA}%
    \else
        \hskip -0.5\mylenA\rlap{\usebox\myboxA}{\hskip 0.5\mylenA\usebox\myboxB}%
    \fi}
\makeatother

\newcommand{\valAppTail}{EC.1}
\newcommand{\valAppStable}{EC.2}
\newcommand{\valAppDom}{EC.3}
\newcommand{\valAppSupp}{EC.4}
\newcommand{\valAppBAR}{EC.4.2}
\newcommand{\valAppBdd}{EC.4.3}
\newcommand{\valAppIndep}{EC.4.4}
\newcommand{\valAppEquality}{EC.4.5}
\newcommand{\smartref}[3]{%
    \ifcsvoid{r@#1}{%
        \textbf{?? (Missing Label #1)}%
    }{%
        \edef\actualval{\getrefnumber{#1}}%
        \edef\expectedval{#3}%
        \ifx\actualval\expectedval
            #2~\ref{#1}%
        \else
            \PackageError{SmartRef}{%
                Reference Mismatch! Label '#1' is actually \actualval, but you wrote #3%
            }{%
                Change the third argument of \noexpand\smartref{#1}{#2}{#3} to \actualval.%
            }%
        \fi
    }%
}

\newcommand{\ggn}{$GI/GI/n$\xspace}
\newcommand{\ggone}{$GI/GI/1$\xspace}
\newcommand{\mgn}{$M/GI/n$\xspace}

\newcommand{\R}{\mathbb{R}}
\newcommand{\Rnn}{\mathbb{R}_{\geq 0}}
\newcommand{\Z}{\mathbb{Z}}
\newcommand{\Znn}{\mathbb{Z}_{\geq 0}}
\newcommand{\Zp}{\mathbb{Z}_{\geq 1}}
\newcommand{\odv}{\mathnormal{d}}
\newcommand{\E}[1]{\mathbb{E}\left[#1\right]}
\newcommand{\Eplain}[1]{\mathbb{E}[#1]}
\newcommand{\Ebig}[1]{\mathbb{E}\big[#1\big]}

\newcommand{\Ebigg}[1]{\mathbb{E}\bigg[#1\bigg]}
\newcommand{\Prob}[1]{\mathbb{P}\left[#1\right]}
\newcommand{\Ep}[2]{\mathbb{E}_{#1}{\left[#2\right]}}
\newcommand{\Epplain}[2]{\mathbb{E}_{#1}{[#2]}}
\newcommand{\Epbig}[2]{\mathbb{E}_{#1}{\big[#2\big]}}
\newcommand{\EpBig}[2]{\mathbb{E}_{#1}{\Big[#2\Big]}}
\newcommand{\Epbigg}[2]{\mathbb{E}_{#1}{\bigg[#2\bigg]}}
\newcommand{\Probp}[2]{\mathbb{P}_{#1}\left[#2\right]}
\newcommand{\Probpbig}[2]{\mathbb{P}_{#1}\big[#2\big]}

\newcommand{\Pnb}{\mathbb{P}}
\newcommand{\Var}[1]{\textnormal{Var}\left[#1\right]}
\newcommand{\given}{\;\middle|\;}
\newcommand{\givenplain}{\;|\;}
\newcommand{\givenbig}{\;\big|\;}
\newcommand{\indibrac}[1]{\vmathbb{1}{\left\{#1\right\}}}

\newcommand{\indibracbig}[1]{\vmathbb{1}{\big\{#1\big\}}}
\newcommand{\indibracBig}[1]{\vmathbb{1}{\Big\{#1\Big\}}}
\newcommand{\abs}[1]{\left\lvert#1\right\rvert}

\newcommand{\absBig}[1]{\Big\lvert#1\Big\rvert}

\newcommand{\sumi}{\sum_{i=1}^n}
\newcommand{\sumj}{\sum_{j=1}^n}

\newcommand{\Rsmax}{R_{s}^{\max}}
\newcommand{\Ramin}{R_{a}^{\min}}
\newcommand{\Smax}{S_{\max}}
\newcommand{\sone}{{s,1}}
\newcommand{\stwo}{{s,2}}
\newcommand{\sn}{{s,n}}
\newcommand{\si}{{s,i}}
\newcommand{\sj}{{s,j}}
\newcommand{\sell}{{s,\ell}}
\newcommand{\Rsi}{R_{\si}}
\newcommand{\Rsj}{R_{\sj}}
\newcommand{\Qdrop}{\widetilde{Q}}
\newcommand{\Qdropj}{Q_{-j}}

\newcommand{\Qdropone}{Q_{-1}}
\newcommand{\Qdropn}{Q_{-n}}
\newcommand{\Xsuper}{\xoverline{X}}
\newcommand{\xsup}{\bar{x}}
\newcommand{\mutotal}{\mu_{\Sigma}}
\newcommand{\rhodropj}{\rho_{-j}}
\newcommand{\Qdropstable}{\widetilde{Q}^s}
\newcommand{\qdropstable}{\widetilde{q}^s}
\newcommand{\indexstable}{J^s}
\newcommand{\Rsmaxj}{R_{\sj}^{\max}}
\newcommand{\Rsmaxi}{R_{\si}^{\max}}

\newcommand{\spacemain}{\mathcal{X}}
\newcommand{\spacesuper}{\xoverline{\mathcal{X}}}
\newcommand{\aft}{^+}
\newcommand{\bef}{^-}

\newcommand{\sdsuper}{\pi} %
\newcommand{\opinner}{\mathcal{A}}
\newcommand{\taui}{\tau_{i}}
\newcommand{\tauj}{\tau_{j}}
\newcommand{\taua}{\tau_a}
\newcommand{\age}{\alpha}
\newcommand{\agei}{\alpha_{i}}
\newcommand{\agej}{\alpha_{j}}
\newcommand{\agea}{\alpha_{a}}
\newcommand{\Ysuper}{\overline{Y}}
\newcommand{\Na}{N_a}
\newcommand{\Ni}{N_i}
\newcommand{\Nj}{N_j}
\newcommand{\cF}{\mathcal{F}}

\newcommand{\dQ}{\Delta Q}
\newcommand{\Df}[1]{\Delta #1}
\newcommand{\df}[1]{\dot{#1}}
\newcommand{\Qproxy}{V}
\newcommand{\Rmean}{\overline{R}}

\theoremstyle{plain}%
\newtheorem{theorem}{Theorem}
\newtheorem*{theorem*}{Theorem}
\newtheorem{lemma}{Lemma}

\theoremstyle{definition}

\newtheorem{assumption}{Assumption}
\newtheorem{example}{Example}
\newtheorem{remark}{Remark}

\def\ECHowTheorems{}

\newif\ifcomment
\commenttrue

\newif\ifoutline
\outlinefalse

\begin{document}

\RUNAUTHOR{Hong}

\RUNTITLE{An interpretable universal bound for multiserver queues}

\TITLE{An interpretable universal bound for multiserver queues via a leave-one-out technique}

\ARTICLEAUTHORS{%
\AUTHOR{Yige Hong}
\AFF{Computer Science Department, Carnegie Mellon University,
\EMAIL{hongyige98@gmail.com}}
} %

\ABSTRACT{%
\input{abstract}

}%

\def\MSCCLASSname{{\it MSC2020 subject classification\/}{\kern0.7pt}:\enskip}
\MSCCLASS{Primary 60K25; secondary 68M20, 90B22}

\maketitle

\section{Introduction}
\label{sec:intro}
The \ggn queue is one of the most fundamental multiserver queueing models, with a wide range of applications \citep[see, e.g.,][]{Wor_09_ggn_applications}. %
The \ggn queue consists of $n$ servers and a central queue, where external jobs arrive over time and wait in the queue for service in the first-come, first-served order. 
When a server frees up, it picks the first job in the queue and starts serving it until the job is completed.
An illustration of the \ggn queue is provided in \Cref{fig:g-g-n}.
The name ``GI'' refers to ``general'' and ``independent'', indicating that the interarrival times and service times are two i.i.d. sequences, where elements of each sequence follow a general distribution.

In the theory of the \ggn queue, a natural goal is to find a simple formula that predicts the \emph{queue length} (the number of jobs in the queue waiting for service). %
A prominent line of work focuses on finding accurate bounds or approximations when the system \emph{scales up} in a certain sense. 
This pursuit has been remarkably successful for the single-server case, dating back to the seminal work of \citet{kin_62_bound}. Kingman derives a simple upper bound for the steady-state mean queue length of the \ggone queue, $\Ep{\pi}{Q}$: 
\begin{equation}
\label{eq:kingman}
\Ep{\pi}{Q} \leq \frac{\rho^2\Var{\mu S}+\Var{\lambda A}}{2(1-\rho)},
\end{equation}
where $S$ is the service time, $A$ is the interarrival time, $\mu\triangleq 1/\E{S}$, $\lambda\triangleq 1/\E{A}$, and $\rho \triangleq \lambda/\mu$ is the \emph{load}. 
Despite its simplicity, Kingman's bound is surprisingly accurate. In particular, it is asymptotically tight in the so-called \emph{heavy-traffic} scaling regime where $\rho\uparrow 1$ \citep{kin_62_approx}, so it precisely captures the queue's behavior in the scenario where the queueing delay becomes significant.

\begin{figure}
    \includegraphics[width=9cm]{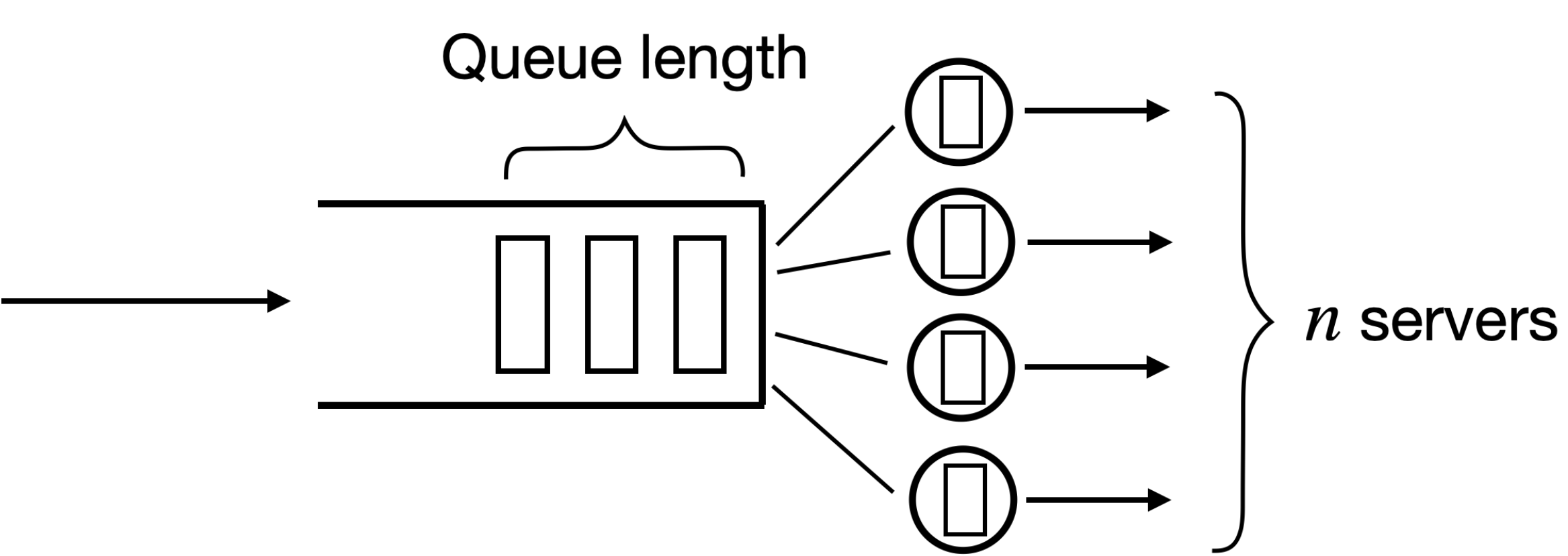}
    \centering
    \caption{An illustration of the \ggn queue, where each circle denotes a server, and each rectangle denotes a job.}
    \label{fig:g-g-n}
\end{figure}

In contrast to the single-server case, finding a similarly simple and accurate upper bound for the \ggn queue is a much greater challenge. This complexity arises from a rich set of combinations of two basic scalings: the \emph{heavy-traffic} scaling, where the load $\rho\uparrow 1$, and the \emph{many-server} scaling, where the number of servers $n\to\infty$. 
In practice, these two parameters often scale jointly to balance utilization and delay \citep{Whi_92_sqrt_staffing,BorManRei_04_sqrt_staff,GarManRei_02_sqrt_staff,ManMom_08_qed,Ata_12}. 
Different joint scaling rates give rise to dramatically different asymptotic behaviors.   
For example, in the classical heavy-traffic regime ($\rho\uparrow 1$ with fixed $n$), the limiting queue length distribution of the \ggn queue is similar to that of \ggone, depending only on the first two moments of the input distributions \citep[see, e.g.,][]{Kol_74_heavy_1}. 
In contrast, in the Halfin-Whitt (HW) regime ($n\to\infty$ and $\rho = 1 - \Theta(1/\sqrt{n})$) or the Non-Degenerate Slowdown (NDS) regime ($n\to\infty$ and $\rho = 1 - \Theta(1/n)$),
the limiting behavior is far more complex and can depend intricately on the full input distributions \citep[see, e.g.,][]{BraDai_17,Ree_09,AghRam_20_hw,AtaSol_11_NDS}. 

A long-standing challenge has been to establish a queue length bound of the order $O\big(1/(1-\rho)\big)$ that holds \emph{universally} for all $\rho < 1$ and $n$. 
Achieving this is significant because this order is known to be tight in a broad class of important regimes, including the classical heavy-traffic, NDS, and Halfin-Whitt \citep{Kol_74_heavy_1,AghRam_20_hw,AtaSol_11_NDS}. %
However, most prior results are tailored to specific scaling regimes, making them either invalid or loose outside of the specific regimes.
A universal bound of the order $O\big(1/(1-\rho)\big)$ has remained elusive until the recent breakthrough by \citet{LiGol_25}, which provides the first such result, stated in their Corollary~2 as:
\begin{equation}
    \label{eq:intro:li-gol-bound}
    \Ep{\pi}{Q} 
    \leq \left(2.1 \times 10^{21}  \E{(\mu S)^2}  \Big(\E{(\mu S)^2}^{1+\epsilon} + \E{(\mu S)^{2+\epsilon}}\Big) \big(\frac{1}{\epsilon}\big)^4 + 49 \E{(\Lambda A)^2}\right) \frac{1}{1-\rho},
\end{equation}
where $\Lambda = 1/\E{A}$ and $\epsilon\in (0,0.5)$. 
In addition to this bound on the mean queue length, \citet{LiGol_25} also establishes several high-order moment and tail bounds of the same order. 

Although \citet{LiGol_25} shows the existence of a universal bound of the correct order, $O\big(1/(1-\rho)\big)$, the numerical constant in their bound is astronomically large. This stands in stark contrast to Kingman's bound for the single-server case, suggesting room for improvement. 
Indeed, \citet{LiGol_25} themselves conjecture that Kingman's bound may hold for any \mgn queue. This leads to the central question of our paper: \emph{Can we prove a universal $O\big(1/(1-\rho)\big)$ bound for the \ggn queue with a small and interpretable constant factor?}

\paragraph*{Summary of contribution}
We answer this question affirmatively. Our main contributions are as follows.
\begin{itemize}
    \item \emph{A new technique.} We introduce a \emph{leave-one-out coupling} technique that overcomes a long-standing obstacle to analyzing the \ggn queue.
    \item \emph{A simple and transparent proof.} Our proof is significantly simpler than that of \citet{LiGol_25}: it relies on a short core argument and makes the origin of each term in the bound transparent.
    \item \emph{A small and interpretable constant.} For light-tailed service times, our leading constant is interpretable and orders of magnitude smaller than that of prior work. A comparison of the constant factors is given in \Cref{tab:intro-comparison}, with the details derived in \Cref{sec:main-result}.
    \item \emph{Heterogeneous servers.} We extend the method to the \ggn queue with \emph{fully heterogeneous} service-time distributions, a setting not addressed by prior universal bounds (\Cref{sec:hetero}).
\end{itemize}

\paragraph*{Our results}
We prove that, in the \ggn queue, the steady-state mean queue length satisfies
\begin{equation}
    \label{eq:intro:our-ggn-bound}
    \Ep{\pi}{Q} \leq \frac{2\Rsmax + \rho \Var{\Lambda A} - \rho}{2(1-\rho)} + \Big(\frac{1}{2}\E{(\Lambda A)^2} - \Ramin\Big),
\end{equation}
where $\Rsmax \triangleq \sup_{t\geq0}\E{\mu S - t \given \mu S\geq t}$ (finite for a large class of light-tailed distributions), and the second term $\E{(\Lambda A)^2}/2 - \Ramin$ is a constant determined by the unitized arrival-time distribution. 
This bound holds assuming $\E{A^2}<\infty$, $\Rsmax < \infty$ (\Cref{assump:bdd-exp-remain}), $S$ is non-lattice (\Cref{assump:non-lattice}), and a certain positive Harris recurrence condition (\Cref{assump:modified-ggn-stable}). 
When we consider the \mgn queue, the bound in \eqref{eq:main-upper-bound-simplified} can be further simplified into:
\begin{equation}
    \label{eq:intro:our-mgn-bound}
    \Ep{\pi}{Q} \leq \frac{\Rsmax}{1-\rho}. 
\end{equation}
The above two bounds are given in \Cref{sec:main-result} as direct corollaries of a more refined bound (\Cref{thm:main-upper-bound}), which has a better leading constant when $n = o(1/(1 - \rho))$.
The refined bound mirrors the structure of Kingman's classical single-server bound. Its leading term recovers Kingman's bound for the \ggone queue that has interarrival time $A$ and service time $S/n$. The remaining terms, governed by the maximal mean residual service time $\Rsmax$, quantify the multiserver boundary effect. We elaborate on this decomposition in \Cref{sec:main-result}.

Our bounds hold universally for all $\rho < 1$ and $n$ and have the order $O\big(1/(1-\rho)\big)$. 
Compared with the bounds in \citep{LiGol_25}, our bounds require stronger assumptions on the service-time distribution, but often have much smaller constants when the assumptions are satisfied. 
For example, in the \mgn queue whose service times follow the gamma distribution with the shape parameter $\alpha$ (i.e., with the squared coefficient of variation $1/\alpha$), $\Rsmax = \max(1, 1/\alpha)$, so \eqref{eq:intro:our-mgn-bound} implies $\Ep{\pi}{Q} \leq \max(1, 1/\alpha) / (1-\rho)$.
\Cref{tab:intro-comparison} previews these leading constants for several standard distribution classes, including NBUE service (whose mean residual life never exceeds its mean, a class containing all increasing-failure-rate distributions), phase-type, and bounded service times; we derive them in \Cref{sec:main-result}.

Our results also generalize to the \ggn queue with \emph{fully heterogeneous} servers. The details can be found in \Cref{sec:hetero}.

\begin{table}[t]
    \centering
    \TABLE{Leading constants for the \mgn queue.\label{tab:intro-comparison}}
    {\begin{tabular}{lcc}
    \toprule
    Service distribution & \citet{LiGol_25} & This paper \\
    \midrule
    Exponential ($M/M/n$)          & $2\times 10^{21}\,\kappa$ & $1$ \\
    NBUE                           & $2\times 10^{21}\,\kappa$ & $1$ \\
    Gamma with shape $\alpha$      & $2\times 10^{21}\,\kappa$ & $\max(1,\,1/\alpha)$ \\
    Phase-type with $d$ phases     & $2\times 10^{21}\,\kappa$ & $\mu\max_{k\in[d]}\tau_k$ \\
    Bounded by $\Smax$             & $2\times 10^{21}\,\kappa$ & $\mu \Smax$ \\
    \bottomrule
    \end{tabular}}
    {The leading constant $C$ in the bound $\Ep{\pi}{Q}\leq C/(1-\rho)$ for the \mgn queue. Our column follows from \eqref{eq:intro:our-mgn-bound} (and \Cref{sec:main-result}); for phase-type service, $\tau_k$ is the mean residual service time in phase~$k$. The \citet{LiGol_25} column is the coefficient of $1/(1-\rho)$ in \eqref{eq:intro:li-gol-bound}, where $\kappa$ is a distribution-dependent factor (differing across rows) involving the moments of $\mu S$, $\Lambda A$, and a parameter $\epsilon\in(0,0.5)$. For every distribution listed, the factor $\kappa$ alone exceeds the corresponding entry in the last column. The bounds of \citet{LiGol_25} also hold under weaker (finite-moment) assumptions; the table compares the resulting leading constants where both apply.}
\end{table}

\paragraph*{Proof techniques}
On an intuitive level, our proof proceeds by quantifying the discrepancy between the \ggn queue and a \ggone queue. This \ggone queue has an interarrival time distribution $A$ and a service time distribution $S/n$; intuitively, it serves the same sequence of jobs using a ``super-server'' that is $n$ times faster than each server in \ggn. 
Observe that the \ggn matches the ``speed'' of the corresponding \ggone when all of its $n$ servers are busy, and is slower otherwise. 
Because this loss of speed is restricted to periods when \ggn has idle servers and an empty queue, one might expect the queue length of \ggn to be comparable to that of the corresponding \ggone.

The formal proof consists of three steps:
\begin{itemize}
    \item First, we construct a modified \ggn queue whose queue length stochastically dominates that of the original system.
    \item Second, we derive an identity that decomposes the mean queue length of this modified system into two components: the mean queue length of the \ggone and a set of ``covariance'' terms that capture the consequence of the ``loss of speed'' when the queue length hits zero. 
    \item Finally, we bound these covariance terms.
\end{itemize}

The first two steps are based on known techniques. The modified \ggn is a standard construction in the literature, recently adopted by \citet{LiGol_25} to obtain the first $O(1/(1-\rho))$ bound. 
In the second step, the decomposition of the mean queue length is derived using the Basic Adjoint Relationship (BAR) \citep[see, e.g.,][]{miyazawa_decomposition_1994, miyazawa_diffusion_2015,braverman_heavy_2017,BraDaiMiy_23_bar_ss}. 
Our technical innovation lies in the third step, where we introduce a novel \emph{leave-one-out coupling} to control the covariance terms.
At a high level, this technique hypothetically removes one server and compares the resulting counterfactual queue length with the original. This comparison quantifies the correlation between the queue length and the removed server's service times.
We note that similar covariance terms have appeared in previous analyses of \ggn and have been a major obstacle to obtaining a universal $O\big(1/(1-\rho)\big)$ bound \citep{Hok_85_ggs, GroHarSch_22_wfcs, scully_gittins_2020, HonWan_24_ss}.
The name of our technique reflects an analogy with leave-one-out arguments in statistics and learning theory \citep[see, e.g.,][]{LacMik_68_loocv,ShaShaSreSri_10_alg_stability,DinChe_20_loo}.

A detailed proof sketch is provided in \Cref{sec:tech-overview}, and a roadmap to the rigorous proof can be found in \Cref{sec:tech-overview-roadmap}.

\paragraph*{Paper organization}
We first review the related work in \Cref{sec:related-work}. 
In \Cref{sec:model-homo}, we formally set up the model, introduce the modified \ggn queue, and state the assumptions about these systems. 
In \Cref{sec:main-result}, we state our main results and provide some examples.  
In Section~\ref{sec:tech-overview}, we provide an informal overview and a roadmap for the formal proof. 
In \Cref{sec:lemmas-homo}, we formally set up the analysis framework and state some supporting lemmas. In \Cref{sec:pf-main-upper-bound}, we prove the main results. 
In \Cref{sec:hetero}, we generalize our results to the \ggn queue with heterogeneous servers.
Finally, we conclude the paper in \Cref{sec:conclusion}.

We organize the proof so that its ideas and its formal development can be read independently. \Cref{sec:tech-overview} is self-contained, so readers mainly seeking these ideas may read it and skip the formal development in \Cref{sec:lemmas-homo,sec:pf-main-upper-bound}.

\paragraph*{General notation}
For any two real numbers $a, b$, we use $a\wedge b$ to denote their minimum. 
For any positive integer $k$, we let $[k] \triangleq \{1,2,\dots, k\}$. We let $\R$ be the set of real numbers, let $\Rnn$ be the set of non-negative real numbers, let $\Znn$ be the set of non-negative integers, and let $\Zp$ be the set of positive integers. 
We use standard asymptotic notation for non-negative functions $f$ and $g$: $f = O(g)$ if $\limsup f/g < \infty$; $f = o(g)$ if $f/g \to 0$; $f = \Omega(g)$ if $g = O(f)$; $f = \omega(g)$ if $f/g \to \infty$; and $f = \Theta(g)$ if $f = O(g)$ and $g = O(f)$. Unless otherwise stated, these limits are taken as $\rho \to 1$ or $n \to \infty$, with the underlying unitized distributions $S / \mathbb{E}[S]$ and $A/\mathbb{E}[A]$ fixed.
We denote probability, expectation, and variance by $\Prob{\cdot}$, $\E{\cdot}$, and $\Var{\cdot}$, respectively. The subscripts of these three symbols specify the underlying probability distributions: for example, $\Probp{\nu}{\cdot}$ and $\Ep{\nu}{\cdot}$ are taken under the distribution $\nu$. 
When the distribution is clear from context, we omit the subscript.

\section{Related work}
\label{sec:related-work}
\paragraph*{General prior work on the \ggn queue}
There has been a rich body of literature on the queue-length scaling of the \ggn queue.
Prior to \citet{LiGol_25}, most papers focus on a certain representative or useful scaling regime of $\rho$ and $n$, without providing general bounds that hold for arbitrary $\rho$ and $n$.
In the following, we briefly review these papers and refer the readers to \citep[]{LiGol_25} for a comprehensive review. 

The most well-studied regime is the classical heavy-traffic regime, where $n$ is fixed, and $\rho\uparrow 1$ \citep[e.g.,][]{Kol_74_heavy_1,Kol_74_heavy_2,Loulou_73,Hok_85_ggs,Bor_65,Kin_70_ggk,Mor_75,GroHarSch_22_wfcs}. 
In this regime, relatively simple bounds or approximations are available, and the queue length is known to be $O\big(1/(1-\rho)\big)$  under proper moment assumptions. 
Some of these papers contain bounds that hold beyond the classical heavy-traffic regime, but they grow faster than order $1/(1-\rho)$ when $n$ is large \citep[e.g.,][]{Kol_74_heavy_2,Loulou_73,Hok_85_ggs,GroHarSch_22_wfcs,Mor_75}. For example, the bound proved by \citep{GroHarSch_22_wfcs} has order $O(1/(1-\rho) + n)$.

More recent papers also study the Halfin-Whitt regime (HW), where $n\to\infty$ and $\rho = 1-\Theta(1/\sqrt{n})$
\citep[e.g.,][]{HalWhi_81,Ree_09,GamGol_13_GGn_HW,AghRam_20_hw,DaiDieGao_14_hw_validity,BraDai_16_Erlang_C,BraDaiFen_17,BraDai_17,BraDaiFan_24_high_order,Bra_17,GurHuaMan_14_Erlang_A,ManMasWilRei_98,JinPanXuXu_25_hw,GamSto_12_hw_lyapunov,PuhRei_00_hw_gphn,DaiHeTez_10_GPhnG_hw_process_level,ManMom_08_qed,JelManMom_04_GDN_hw,JanLeeZwa_08_GDN_hw,JanLeeZwa_11_hw,GamMom_08_hw_lyapunov}, or the Non-Degenerate Slowdown (NDS) regime, where $n\to\infty$ and $\rho=1-\Theta(1/n)$ \citep[e.g.,][]{Ata_12,AtaSol_11_NDS,GurHuaMan_14_Erlang_A}. 
The queue length in these regimes has been shown to be of the order $O\big(1/(1-\rho)\big)$. 
However, a key difference from the classical heavy-traffic regime is that the analysis for the HW and NDS regimes is typically asymptotic, focusing on the system's limiting behavior as $n\to\infty$. 
Although some work provides convergence rates \citep[e.g.][]{BraDai_17} or proves the tightness of properly scaled queue lengths \citep[e.g.][]{GamGol_13_GGn_HW,DaiDieGao_14_hw_validity,GamMom_08_hw_lyapunov,GamSto_12_hw_lyapunov}, the bounds are generally \emph{inexplicit} except for certain special cases (e.g., exponential service times \citep{BraDai_16_Erlang_C,BraDaiFen_17,PraDanMag_24_exp_tail,GamSto_12_hw_lyapunov}, deterministic service times \citep{JelManMom_04_GDN_hw,JanLeeZwa_08_GDN_hw}, or lattice service times \citep{GamMom_08_hw_lyapunov}). It is thus difficult to translate these results into the type of explicit, non-asymptotic performance bounds that are the focus of our work.

Some papers also consider scaling regimes with load further lighter than HW, i.e. $\rho = 1-\omega(1/\sqrt{n})$ \citep[e.g.][]{HonWan_24_ss, WanXieHar_21_pomacs,PraDanMag_24_exp_tail,BraDai_16_Erlang_C,BraDaiFen_17,GurHuaMan_14_Erlang_A}. Such regimes are referred to as \textit{sub-Halfin-Whitt regimes} (sub-HW). In the sub-HW, the mean queue length of the \ggn queue could be exponentially smaller than $1/(1-\rho)$, although the precise scaling of the mean queue length is only known for very special models, such as the $M/M/n$ queue \citep[see, e.g.][]{PraDanMag_24_exp_tail,BraDai_16_Erlang_C}. 

Finally, apart from \citep{LiGol_25}, some papers also provide results that hold across regimes, but under highly specialized distributional assumptions. 
Some of them assume exponential service times and provide approximations with universal error bounds \citep[e.g.,][]{BraDai_16_Erlang_C,BraDaiFen_17,GurHuaMan_14_Erlang_A,PraDanMag_24_exp_tail}.
Another paper \citep{HonWan_24_ss} provides an $O\big(1/(1-\rho)\big)$ bound for some regimes between HW and NDS, assuming Poisson arrivals and hyperexponential service times.

\paragraph*{Basic Adjoint Relationship and similar methods}
Basic Adjoint Relationship (BAR), also known as the \emph{stationary equation} in the earlier literature \citep{miyazawa_rate_1994}, is a set of equations that characterize the stationary distribution of Markov processes.
BAR allows us to extract information about the stationary distribution of a Markov process by taking suitable test functions. BAR and BAR-based analytical techniques have been studied in the literature under different names, such as the drift method \citep[see, e.g.,][]{ErySri_12} or the rate conservation law \citep[see, e.g.,][]{miyazawa_decomposition_1994,miyazawa_diffusion_2015}. 
In this paper, we use BAR to refer to the particular framework in \citep[][]{braverman_heavy_2017,BraDaiMiy_23_bar_ss}, which handles Markov processes with both continuous and discrete state changes, utilizing the Palm distribution. 
This BAR framework has recently been applied to various queueing systems with general interarrival-time and service-time distributions under heavy traffic \citep[e.g.,][]{miyazawa_diffusion_2015,braverman_heavy_2017,BraDaiMiy_23_bar_ss,DaiGuaXu_24_jsq,DaiGlyXu_25_multiscale_jackson,DaiHuo_24_multiscale_sbp,GuaCheDai_25_uniform_bdd_multiscale,hong_performance_2023,GroHonHarSch_24_reset,BraScu_24_gg1}.

Some prior work has derived queue length bounds for the \ggn queue based on BAR or equivalent approaches
\citep[e.g.][]{Hok_85_ggs,GroHarSch_22_wfcs,scully_gittins_2020, HonWan_24_ss}. 
These bounds typically have small constant factors in front of the $1/(1-\rho)$ term, but they also contain additive error terms that depend on $n$. These additive error terms cause these bounds to grow faster than order $1/(1-\rho)$ when $n$ is large, reflecting a common technical challenge in analyzing the \ggn queue using BAR-based approaches. We will discuss this challenge in more detail in \Cref{sec:tech-overview}.

\paragraph*{Leave-one-out techniques}
Our core technical innovation, the leave-one-out technique, is similar in spirit to the leave-one-out cross-validation in statistics \citep{LacMik_68_loocv}, as well as the related methods in learning theory \citep{ShaShaSreSri_10_alg_stability} and matrix completion \citep{DinChe_20_loo}. 
On a high level, these approaches quantify the influence of a single input on a stochastic system by analyzing the sensitivity of the outcome to that component’s removal or replacement.

\section{Problem setup}%
\label{sec:model-homo}

In this section, we define two queueing systems with homogeneous servers: the \ggn queue, whose queue length we want to upper bound, and a proxy system named \emph{modified \ggn queue}, which will be the focus of our analysis.

\subsection{\ggn queue}
\label{sec:model-homo:ggn-original}
\ifoutline
\begin{outline}
    \1 Parameters: $n$, $\mu$, $\lambda$, $\rho$, $S$, $A$
    \1 Primitive stochastic elements: $S_i^k$, $A^k$; all independent
    \1 State representation: $Q(t)$, $R_{\si}(t)$, $R_{a}(t)$
    \1 FCFS policy and system dynamics
    \1 Stability
\end{outline}
\else\fi

We consider the multiserver queue with general independent interarrival-time and service-time distributions, also known as the \ggn queue, under the first-come, first-served (FCFS) scheduling policy. 
As illustrated in \Cref{fig:g-g-n}, this queueing system consists of a central queue and $n$ identical servers, serving externally arriving jobs with independent and identically distributed (i.i.d.) interarrival times. 
Each job joins the end of the queue if all servers are busy, or starts service immediately if there are any idle servers. 
After entering the service, the job picks an idle server uniformly at random, occupies the server for a random amount of service time, and leaves the system after completing the service. The service times of the jobs are also i.i.d. 
When a server completes a job, it either immediately takes the next job from the head of the queue to serve it if the queue is non-empty, or becomes idle if the queue is empty. 
If multiple servers complete the service simultaneously, we assign the head-of-line jobs to these servers uniformly at random.

The \ggn system starts running at the time $t=0$, possibly with some initial jobs in the queue and in service. We allow arbitrary initial queue length (denoted as $Q(0)$), time until first arrival, the set of initially busy servers, and the residual service times of the busy servers. These data describe the status of the system at the time $t=0$, and we refer to them collectively as the \emph{initial state}.  

The \ggn queue is driven by two independent i.i.d.\ sequences of random variables, which represent the interarrival times and service times. 
Specifically, under the FCFS policy, we can index all jobs by the order in which they begin service. 
For the $k$-th job that begins service after time~$0$, we let $S^k$ denote its service time and let $A^k$ denote the interarrival time between the $k$-th and $(k+1)$-th jobs.  
Note that $\{A^k\}_{k=1}^{Q(0)}$ correspond to jobs already in the queue at time~$0$, so they do not play a role in the analysis and can take arbitrary values. 
By the definition of the \ggn queue, the interarrival times $\{A^{k}\}_{k= Q(0)+1}^\infty$ and the service times $\{S^{k}\}_{k=1}^\infty$ are independent, and each sequence is i.i.d.\ with its own distribution. 
We represent the distributions of $A^k$ and $S^k$ using two generic random variables $A$ and $S$, respectively, and simply refer to them as the \emph{interarrival time} and the \emph{service time}. 
The \emph{arrival rate} $\Lambda$ and \emph{service rate} $\mu$ of the system can therefore be defined as $\Lambda \triangleq 1/\E{A}$ and $\mu \triangleq 1/\E{S}$. 
We also define the \emph{scaled arrival rate} $\lambda$ to be $\lambda \triangleq \Lambda / n$, which represents the average arrival rate processed by each server. 
To characterize the tails of the unitized service time $\mu S$ and the unitized interarrival time $\Lambda A$, we consider the supremum of $\mu S$'s mean residual time $\Rsmax \triangleq \sup_{t\geq0}\E{\mu S - t \given \mu S\geq t}$, and the infimum of $\Lambda A$'s mean residual time $\Ramin = \inf_{t\geq 0}\E{\Lambda A - t \given \Lambda A \geq t}$.

Next, we state two assumptions on the distributions of $A$ and $S$. 

\begin{assumption}\label{assump:bdd-exp-remain}
    We assume that $\E{A^2} < \infty$ and $\Rsmax < \infty$. 
\end{assumption}

\begin{assumption}
    \label{assump:non-lattice}
    We assume that the distribution of $S$ is non-lattice, i.e., the support of $S$ is not a subset of $\{\delta, 2\delta, 3\delta, \dots\}$ for any $\delta > 0$. 
\end{assumption}

\Cref{assump:bdd-exp-remain} requires that the interarrival time has a finite second moment, and the service time $S$ has a bounded mean residual time. 
The assumption $\Rsmax < \infty$ has also been used in \citep{Mor_75,Dow_91_bdd_sync_overhead,GroHarSch_22_wfcs}. As shown in \citep{Dow_91_bdd_sync_overhead}, $\Rsmax < \infty$ implies that the tail CDF of $S$ decays at least exponentially and hence that all moments of $S$ are finite; for completeness, we also provide a proof in \smartref{app:Rsmax-finite-implies-exponential}{Section}{\valAppTail}. 
The mean residual time is bounded for many common classes of light-tailed distributions, including the phase-type, Increasing Failure Rate (IFR), New Better Than Used in Expectation (NBUE), gamma distributions, etc. As previewed in \Cref{tab:intro-comparison}, we compute $\Rsmax$ for several of these examples in \Cref{sec:main-result}.

\Cref{assump:non-lattice} can be roughly viewed as a continuous-time analogue of the aperiodicity condition for renewal processes with the inter-event distribution $S$. The non-lattice condition is required for some common versions of renewal theorem \citep[][Theorem V.4.3]{Asm_03}, which we will use in the proof of \Cref{lem:indep-R-Qtilde}.

The \ggn queue is a Markov process whose \emph{state} at time $t$ can be represented as
\begin{equation}
    \label{eq:x-def}
    X(t)\triangleq (R_{\sone}(t), R_{\stwo}(t), \dots R_{\sn}(t), R_a(t), Q(t)), 
\end{equation}
where $Q(t)$ denotes the \emph{queue length} (not including the jobs in service), $R_a(t)$ denotes the \emph{residual arrival time}, and $R_{\si}(t)$ denotes the
\emph{residual service time} of the $i$-th server, for $i=1,2,\dots, n$. 
We use the convention that all the state variables have right-continuous sample paths, and let $R_{\si}(t) = 0$ if the $i$-th server is empty at time $t$. 
The set of all possible realizations of $X(t)$ constitutes the \emph{state space} of the \ggn queue, denoted as $\spacemain \triangleq \Rnn^{n+1}\times\Znn$.

The load of the \ggn system is defined as $\rho\triangleq \Lambda / (n\mu) = \lambda / \mu$. We assume that $\rho < 1$, which implies that the system is positive Harris recurrent \citep[][Theorem XII.2.2]{Asm_03}.

\subsection{Modified \ggn queue}
\label{sec:modified-ggn}

Now we set up the \emph{modified \ggn} queue, which we will focus on in our analysis as a proxy of the original \ggn queue. 
On a high-level, the modified \ggn queue adds virtual jobs to the \ggn queue to keep the servers busy. 
Specifically, when the queue is empty and a server is about to idle, the server immediately begins to serve a virtual job with a service time independently sampled from the distribution of $S$. The server then works non-preemptively on this virtual job, becoming available to serve the next job only after completing the virtual job. 
We initialize the modified \ggn queue with the same set of jobs as the original system, supplemented by the virtual jobs assigned to any servers that would otherwise be idle at the time $t=0$.

The modified \ggn queue is a well-established construction in the queueing literature \citep{Bor_65, KelSta_06_modified_ggn, IglWhi_70_a, ChaThoKia_94_modified_gg1, HalWhi_81, Whi_02_book}. Under the coupling described above, the queue length distribution of the modified system stochastically dominates that of the original \ggn at all times $t \geq 0$. 
This dominance, which was rigorously established in \citep[][Proposition 1]{GamGol_13_GGn_HW}, ensures that any upper bound derived for the modified queue length applies automatically to the original system. 
Following the approach of \citet{LiGol_25}, we therefore focus our analysis exclusively on the modified \ggn. 
For completeness, we provide a self-contained proof of this stochastic dominance in \smartref{app:modified-dominance}{Section}{\valAppDom}.

A nice property of the modified \ggn queue is that its sample path is \emph{determined} by $(n+1)$ mutually independent renewal processes \citep[][Corollary~1]{GamGol_13_GGn_HW}: the service-completions of each server are a renewal process with the inter-event distribution $S$, and the arrivals are a renewal process with the inter-event distribution $A$.

The modified \ggn queue is a Markov process whose state can be represented by the same tuple $X(t)$ as the original \ggn queue \eqref{eq:x-def}, with the state space $\spacemain = \Rnn^{n+1}\times \Znn$. The dynamics of this Markov process is summarized below:
\begin{itemize}%
\item \textbf{Continuous dynamics}
At any time $t$, the residual times (i.e., forward recurrence times) of the renewal processes, $R_a(t)$ and $R_{\si}(t)$ for $i\in[n]$, decrease at the constant rate $1$. 
\item \textbf{Discrete events}
When one or more residual times reach zero at time $t$, discrete events are triggered. 
Each type of event corresponds to a subroutine that maps an input state, $X\bef=(R_{\sone}\bef, \dots R_{\sn}\bef, R_{a}\bef, Q\bef)$, to an output state, $X\aft=(R_{\sone}\aft, \dots R_{\sn}\aft, R_{a}\aft, Q\aft)$, according to the following rules: 
\begin{enumerate}%
    \item[(i) ] \textbf{Arrival event} ($R_a(t-)=0$): 
    $R_a\aft = R_a\bef + A$ for some interarrival time $A$ independent of $X\bef$; $Q\aft = Q\bef + 1$; other coordinates of the state remain unchanged. 
    \item[(ii)] \textbf{Completion event at server $i$} ($R_{\si}(t-)=0$): $R_{\si}\aft = R_{\si}\bef + S$ for some service time $S$ independent of $X\bef$; $Q\aft = Q\bef - \indibrac{Q\bef > 0}$; other coordinates of the state remain unchanged. 
\end{enumerate}
Typically, only one event happens at a time, and the post-event state $X(t)$ is generated by applying the corresponding subroutine to the pre-event state $X(t-)$. 
In rare cases where multiple events occur simultaneously, the subroutines are applied sequentially. Different orders of applying the subroutines result in different system dynamics, so we stick to a fixed order to avoid ambiguities: we first apply the arrival subroutine if there is an arrival event, and then any completion subroutines in increasing order of their server indices. 
\end{itemize}

Next, we make a stability assumption for the modified \ggn queues. Note that in addition to the main modified \ggn queue under study, we will also construct a few more auxiliary instances of modified \ggn queues (\Cref{sec:leave-one-out}); these auxiliary systems have different numbers of servers but same interarrival-time and service-time distributions as the main modified \ggn queue under study. 
Our assumption should also hold for all these systems; the load of each of these systems is defined as the ratio of its arrival rate and the total service rate of its servers.

\begin{assumption}\label{assump:modified-ggn-stable}
    Each modified \ggn queue with homogeneous servers considered in this paper is positive Harris recurrent if  its load is less than $1$. 
    In addition, let $\nu$ be its unique stationary distribution, then $\lim_{T\to\infty} \frac{1}{T} \int_0^T \E{Q(t)} \odv t = \Ep{\nu}{Q} < \infty$.
\end{assumption}

\begin{remark}
    \label{remark:stability}
    \Cref{assump:modified-ggn-stable} has two parts: the positive Harris recurrence of the modified \ggn queue, and the finiteness of its steady-state mean queue length. 
    Essentially, only the first part is restrictive, as the finiteness of the steady-state mean queue length could be implied by the existing queue length bounds for modified \ggn queues. 
    In particular, Corollary~2 in \citep{LiGol_25} holds for any modified \ggn queue with load less than $1$ that satisfies positive Harris recurrent, aperiodicity, and some standard finite moment assumptions implied by \Cref{assump:bdd-exp-remain}. 
    We will discuss \Cref{assump:modified-ggn-stable} in more detail in \smartref{app:stability}{Section}{\valAppStable}, where we argue that certain regularity conditions on the distributions of $A$ and $S$ are likely to imply \Cref{assump:modified-ggn-stable}. 
\end{remark}

\section{Main results}%
\label{sec:main-result}
In this section, we state our main results for the \ggn queue with homogeneous servers under Assumptions~\ref{assump:bdd-exp-remain},  \ref{assump:non-lattice} and \ref{assump:modified-ggn-stable}, and provide some examples.  
We will also compare our results with those of \citep{LiGol_25}. 
We will present generalizations to the heterogeneous setting in Section~\ref{sec:hetero}. 

We begin with our main bound, presented in \Cref{thm:main-upper-bound} below. This bound is a more refined version of the bounds \eqref{eq:intro:our-ggn-bound} and \eqref{eq:intro:our-mgn-bound} from the introduction. 

\begin{theorem}\label{thm:main-upper-bound}
    Consider the \ggn queue with homogeneous servers under the FCFS policy, whose service-time and interarrival-time distributions satisfy Assumptions~\ref{assump:bdd-exp-remain}, \ref{assump:non-lattice} and \ref{assump:modified-ggn-stable}. When the load $\rho < 1$, let $\pi$ be the unique stationary distribution of the system. Then the steady-state mean queue length, $\Ep{\pi}{Q}$, is bounded as
    \begin{equation}
        \label{eq:main-upper-bound}
        \begin{aligned}
            \Ep{\pi}{Q} &\leq \frac{\rho\Var{\Lambda A} + \Var{\mu S} + 1-\rho}{2(1-\rho)} + \min\Big(n, \frac{1}{1-\rho}\Big) \Big(\Rsmax - \frac{1}{2} \E{(\mu S)^2}\Big) \\
            &\mspace{23mu} + \Big(\frac{1}{2}\E{(\Lambda A)^2} - \Ramin\Big).
        \end{aligned}
    \end{equation}
\end{theorem}

We first comment on the order of the bound in \eqref{eq:main-upper-bound}. 
Observe that this bound only depends on the distributions of the \emph{unitized} service time $\mu S$ and the unitized interarrival time $\Lambda A$. Thus, once these distributions are fixed, this bound remains $O\big(1/(1-\rho)\big)$ universally for any load $\rho < 1$ and server count $n$. 
When specialized to the Halfin-Whitt regime \citep{HalWhi_81} and the Non-Degenerate Slowdown regime \citep{Ata_12}, the bound has the orders $O(\sqrt{n})$ and $O(n)$, respectively, which match the correct order in these two regimes. 
The ``$\min$'' in the second term of \eqref{eq:main-upper-bound} leads to a smaller leading constant when $n = o\big(1/(1-\rho)\big)$, which includes the classical heavy-traffic regime with $n$ fixed.

As we might notice, the bound in \eqref{eq:main-upper-bound} shares some similarities to Kingman's approximation for the \ggone queue \citep{kin_62_approx}, so it is natural to wonder what each term in \eqref{eq:main-upper-bound} represents and where it comes from. 
In the following, we provide some high-level comments and leave more detailed discussions to the technical overview (\Cref{sec:tech-overview}). 
\begin{enumerate} %
    \item The first term in \eqref{eq:main-upper-bound} matches Kingman's approximation for the \ggone queue whose server has interarrival time $A$ and service time $S/n$, up to low-order errors; 
    this \ggone queue can be interpreted as having a ``super-server'' that is $n$-times faster than each server in the \ggn queue. 
    Because the first term in \eqref{eq:main-upper-bound} is dominant as $\rho\uparrow 1$ with $n$ fixed, it reflects the shared asymptotical behavior of \ggn and \ggone in the classical heavy-traffic regime. This connection between \ggn and \ggone is well known in the literature \citep[see, e.g.,][]{Hok_85_ggs,scully_gittins_2020,GroHarSch_22_wfcs}. In our analysis, this comparison between \ggn and \ggone is performed \emph{implicitly} using the BAR technique. %
    \item The second and third terms in \eqref{eq:main-upper-bound} bound the difference between the (modified) \ggn queue and the \ggone queue. 
    In particular, the second term in \eqref{eq:main-upper-bound} is associated with the conditional distribution of the residual service time given $Q = 0$ in the modified \ggn queue, reflecting a certain ``boundary effect'' when the queue is empty. 
    Note that this term vanishes when $S$ is exponentially distributed, since $\E{(\mu S)^2} / 2$ is the expectation of $\mu S$'s excess distribution, and recall that $\Rsmax = \sup_{t\geq 0} \E{\mu S - t \given \mu S \geq t}$. 
    \item The third term in \eqref{eq:main-upper-bound} is associated with the conditional distribution of the residual interarrival time given $Q = 0$ in the modified \ggn queue, and it vanishes when $A$ is exponentially distributed. 
\end{enumerate}

\Cref{thm:main-upper-bound} implies the following simpler bound after relaxing the minimum in \eqref{eq:main-upper-bound} to $1/(1-\rho)$: 
\begin{equation}
    \label{eq:main-upper-bound-simplified}
    \Ep{\pi}{Q} \leq \frac{\rho \Var{\Lambda A} - \rho + 2\Rsmax}{2(1-\rho)} + \frac{1}{2}\E{(\Lambda A)^2} - \Ramin.
\end{equation}
When we consider the \mgn queue, the bound in \eqref{eq:main-upper-bound-simplified} can be further simplified into:
\begin{equation}
    \label{eq:main-upper-bound-MGn}
     \Ep{\pi}{Q} \leq \frac{\Rsmax}{1-\rho}.
\end{equation}

Next, we apply the bound in \eqref{eq:main-upper-bound-MGn} to some examples, deriving the leading constants previewed in \Cref{tab:intro-comparison}. We assume that the examples under consideration satisfy \Cref{assump:modified-ggn-stable}. %

\begin{example}[NBUE distribution]
    Consider the \mgn queue whose service-time distribution is non-lattice and New Better Than Used in Expectation (NBUE).
    NBUE is a large class of distributions that contains all Increasing Failure Rate distributions; a distribution is NBUE if its mean residual time conditioned on any age is no larger than its original mean. 
    Then by definition, we have $\Rsmax \leq \E{\mu S} = 1$, so \eqref{eq:main-upper-bound-MGn} implies $\Ep{\pi}{Q} \leq 1 / (1-\rho)$.
\end{example}

\begin{example}[Gamma distribution]
    Consider the \mgn queue with gamma-distributed service times. 
    Suppose $S$ follows the gamma distribution with shape parameter $\alpha$ and scale parameter $\theta$, then the unitized service time $\mu S$ is also gamma-distributed with shape $\alpha$ and scale $1/\alpha$. 
    When $\alpha > 1$, $\mu S$ has increasing failure rates, so $\Rsmax = \E{\mu S} = 1$. When $\alpha < 1$, direct calculation shows that $\Rsmax = 1 / \alpha$. 
    Because the gamma distribution is always non-lattice, \eqref{eq:main-upper-bound-MGn} implies $\Ep{\pi}{Q} \leq \max(1/\alpha, 1) / (1-\rho)$. 
\end{example}

\begin{example}[Phase-type distribution]
    Consider the \mgn queue whose service-time distribution follows the phase-type distribution with $d$ phases. For $k\in[d]$, let $\tau_k$ be the expected residual time of $S$ conditioned on it being in the phase $k$, as in \Cref{tab:intro-comparison}. Then it is not hard to see that $\Rsmax \leq \mu\max_{k\in[d]} \tau_k < \infty$. Moreover, phase-type distributions are non-lattice. 
    Therefore, \eqref{eq:main-upper-bound-MGn} implies $\Ep{\pi}{Q} \leq \mu\max_{k\in[d]} \tau_k / (1-\rho)$. 
\end{example}

\begin{example}[Bounded distribution]
    Consider the \mgn queue whose service-time distribution is non-lattice and bounded by some constant $\Smax > 0$. Then we have $\Rsmax \leq \mu \Smax$, so \eqref{eq:main-upper-bound-MGn} implies $\Ep{\pi}{Q} \leq \mu \Smax / (1-\rho)$.
\end{example}

\paragraph*{Comparing with \citet{LiGol_25}}
\citet{LiGol_25} has proved a wide range of queue length bounds of the order $O\big(1/(1-\rho)\big)$, which primarily focus on the tail of the steady-state queue length distribution (Theorems 1,2, and 3). These tail bounds imply first-moment bounds as corollaries. For example:
\begin{theorem*}[Adapted from Corollary 2 of \citep{LiGol_25}]
Consider the \ggn queue with homogeneous servers under the FCFS policy. Assume that (1) $\E{A^2} < \infty$; (2) there exists $\epsilon\in(0, 0.5)$ s.t. $\E{S^{2+\epsilon}}<\infty$; (3) $\rho < 1$; (4) the system has a unique limiting distribution $\pi$. Then 
\begin{equation}
    \label{eq:LiGol-bound-example}
    \Ep{\pi}{Q} \leq \left(2.1 \times 10^{21}  \E{(\mu S)^2}  \Big(\E{(\mu S)^2}^{1+\epsilon} + \E{(\mu S)^{2+\epsilon}}\Big) \big(\frac{1}{\epsilon}\big)^4 + 49 \E{(\Lambda A)^2}\right) \frac{1}{1-\rho},
    \end{equation}
\end{theorem*}

The bound in \eqref{eq:LiGol-bound-example} achieves the order $O\big(1/(1-\rho)\big)$, with the constant factor only depending on $\E{(\Lambda A)^2}$, $\E{(\mu S)^2}$ and $\E{(\mu S)^{2+\epsilon}}$. However, this bound involves a huge numerical constant. 
Similar large numerical constants also appear in all other $O\big(1/(1-\rho)\big)$ bounds in \citep{LiGol_25}. 
Our bounds improve upon the bounds in \citep{LiGol_25}
when $\Rsmax$ is reasonably small.

\section{An informal overview of the proof, with a roadmap}
\label{sec:tech-overview}

This section overviews the proof of \Cref{thm:main-upper-bound}. We will consider four cases with gradually increasing complexity: (1) $M/M/1$ queue, (2) \ggone queue, (3) \ggn queue with $n\leq 1/(1-\rho)$, and (4) \ggn queue with $n > 1 / (1-\rho)$. 
Each case builds upon the previous one and requires at most one additional technique. 

Cases (1)--(3) consist of \emph{known results}, and the tutorial-style proof sketch provided here contains all important arguments of the proof at an informal level. Case (4) represents the main technical novelty of this work. 
At the end of the section, we provide a roadmap for the proof of this last hard case, which we will focus on in the next two sections.
Readers already familiar with BAR may proceed directly to the \ggn case (\Cref{sec:tech-overview-ggn}).

\subsection{$M/M/1$: quadratic test function}
\emph{Takeaway:} the mean queue length of $M/M/1$ can be read off from the stationarity of a \emph{quadratic} test function.

Consider $M/M/1$ queue with the arrival rate $\lambda$ and the service rate $\mu$.
The queue length $Q$ jumps at the rate $\lambda + \mu$, and when a jump happens, the post- and pre-jump states $Q\aft$ and $Q\bef$ satisfy: 
\begin{equation}
    \label{eq:mm1-update-rule}
    Q\aft =
    \begin{cases}
        Q\bef + 1 \quad &\text{ with probability } \lambda / (\lambda + \mu) \\
        Q\bef -\indibrac{Q\bef > 0} \quad &\text{ with probability } \mu / (\lambda + \mu).
    \end{cases}
\end{equation}
In steady state, both $Q\aft$ and $Q\bef$ follow the stationary distribution $\pi$ due to PASTA; the update rule \eqref{eq:mm1-update-rule} thus establishes an invariance relation that allows us to derive equalities or inequalities about $\pi$. 
In particular, to compute the first moment $\Ep{\pi}{Q}$, we use the stationarity of the second moment, i.e., $\E{(Q\aft)^2} = \E{(Q\bef)^2}$. Denote $\dQ \triangleq Q\aft - Q\bef$, then we have
\begin{align}
    \nonumber
    0 &= \E{(Q\bef + \dQ)^2 - (Q\bef)^2} \\
    \label{eq:mm1-drift-method-key}
      &= \E{2(\dQ) Q\bef} + \E{(\dQ)^2}. 
\end{align}
The key observation is that $\E{\dQ \given Q\bef} = (\lambda - \mu \indibrac{Q\bef > 0}) / (\lambda + \mu)$ is almost a constant, making the first term in \eqref{eq:mm1-drift-method-key} proportional to $\E{Q\bef}$:
\[
     \E{2(\dQ) Q\bef} = \frac{2(\lambda - \mu)}{\lambda + \mu} \E{Q\bef}. 
\]
In addition, one can verify that $\E{(\dQ)^2} = 2\lambda / (\lambda + \mu)$. 
Substituting these two terms back into \eqref{eq:mm1-drift-method-key}, we get $\Ep{\pi}{Q} = \E{Q\bef} = \lambda / (\mu-\lambda) = \rho / (1-\rho)$.

Intuitively, the queue length of $M/M/1$ drifts towards zero at a constant mean rate, with noisy transitions that occasionally cause increases.
The steady-state mean $\Ep{\pi}{Q}$ is fully determined by the rate at which the queue length drifts towards zero, $\E{\dQ \given Q\bef}$, and the second moment of the noise, $\E{(\dQ)^2}$, via \eqref{eq:mm1-drift-method-key}. %

\subsection{\ggone: ``smoothed'' quadratic test function}
\emph{Takeaway:} a ``smoothed'' variant of the quadratic test function extends the argument to general service times.

While the quadratic test function argument may not appear necessary for $M/M/1$ whose stationary distribution is available in closed form, 
it turns out to be incredibly useful for \ggone: one can get an equation similar to \eqref{eq:mm1-drift-method-key} and manipulate it into a simple bound on $\Ep{\pi}{Q}$, without figuring out everything about the stationary distribution $\pi$. 
Below, we walk through the \ggone analysis in detail. We first present a naive adaptation of the $M/M/1$ argument, explain why it fails, and then present a small refinement to fix the argument. 
The techniques presented here serve as the foundation for our analysis of \ggn.

To derive an upper bound for the queue length of \ggone, we focus on its modified version (\Cref{sec:modified-ggn}). Let $A$ and $S$ denote generic interarrival and service times, and let $\lambda \triangleq 1/\E{A}$ and $\mu\triangleq 1/\E{S}$ be the arrival and completion rates. 
Then the queue length of the modified \ggone is modulated by two renewal processes, the arrival and completion processes, with inter-event time distributions $A$ and $S$, respectively. The jump in queue length $\dQ = Q\aft - Q\bef$ is given by:  
\begin{alignat}{2}
    & \text{At arrival:}    \quad & \Df{Q} &= 1                   \\
    & \text{At completion:} \quad & \Df{Q} &= -\indibrac{Q\bef > 0}. 
\end{alignat}

We first consider the naive quadratic test function $Q^2$. Intuitively, the long-run average rate of change in $Q^2$ caused by arrivals and completions should sum to zero. This intuition is captured by the following equation:
\begin{align}
    \label{eq:gg1-first-attempt}
    0 &= \lambda \Ep{a}{(Q\bef + \Df{Q})^2 - (Q\bef)^2} + \mu \Ep{s}{(Q\bef + \Df{Q})^2 - (Q\bef)^2} \\
    \label{eq:gg1-first-attempt-step-2}
    &= 2\lambda \Ep{a}{Q\bef} -  2\mu \Ep{s}{Q\bef} + \lambda + \mu\Probp{s}{Q\bef > 0}.
\end{align}
Here, $\Ep{a}{\cdot}$ denotes the expectation with respect to the system's state distribution observed at a random arrival event in the long run, so the first term $\lambda \Ep{a}{(Q\bef + \Df{Q})^2 - (Q\bef)^2}$ is the long-run average rate of changes in $Q^2$ witnessed by arrivals;
$\Ep{s}{\cdot}$ is defined analogously for the completions; $\Ep{a}{\cdot}$ and $\Ep{s}{\cdot}$ are known as \emph{Palm expectations} and will be formalized in \Cref{sec:steady-state-anaysis:bar}. %

Equations \eqref{eq:gg1-first-attempt}--\eqref{eq:gg1-first-attempt-step-2} are \emph{not} useful here because $\Ep{a}{Q\bef}$ and $\Ep{s}{Q\bef}$ are generally different from $\Ep{\pi}{Q}$. 
Intuitively, the queue observed right before a completion should generally be longer than average, whereas the queue observed before an arrival should generally be shorter.

To fix this, the idea is to consider a ``smoothed'' proxy for $Q$ which (1) \emph{continuously} drifts towards zero at a \emph{constant} mean rate, as in the case of $M/M/1$, and (2) has zero expected jump sizes at arrival and completion events to avoid introducing $\Ep{a}{Q\bef}$ and $\Ep{s}{Q\bef}$. 
This smoothed proxy is given by 
\begin{equation}
    \label{eq:qproxy-def-gg1-overview}
    \Qproxy \triangleq Q + \mu R_s - \lambda R_a,
\end{equation}
where $R_s$ and $R_a$ are the residual service and arrival times. 
Let $\Df{\Qproxy} \triangleq \Qproxy\aft - \Qproxy\bef$ denote $\Qproxy$'s jump size at a generic event, and let $\df{\Qproxy}$ denote $\Qproxy$'s continuous rate of change, then
\begin{alignat}{3}
    & \text{At an arrival:}    \;\; & \Df{\Qproxy} &= 1 - \lambda A                   & \quad \implies \quad & \E{\Df{\Qproxy} \given Q\bef, R_a\bef, R_s\bef} = 0 \\
    & \text{At a completion:} \;\; & \Df{\Qproxy} &= -\indibrac{Q\bef > 0} + \mu S   & \quad \implies \quad & \E{\Df{\Qproxy} \given Q\bef, R_a\bef, R_s\bef} = \indibrac{Q\bef=0} \\
    & \text{At any time:}   \;\; & \df{\Qproxy} &= -\mu + \lambda                  & &
\end{alignat}
Here, $A$ denotes the new interarrival time sampled at the arrival time and is independent of the pre-jump state $(Q\bef, R_a\bef, R_s\bef)$; $S$ is defined analogously.

Consider the quadratic test function $\Qproxy^2$. The long-run average rate of change in $\Qproxy^2$ being zero leads to the next equation; compared to \eqref{eq:gg1-first-attempt}, there is an additional term $\Ep{\pi}{2\df{\Qproxy} \Qproxy}$ corresponding to continuous changes:
\begin{align}
    \label{eq:gg1-smoothed-bar:original}
    0 &= \Ep{\pi}{2\df{\Qproxy} \Qproxy} + \lambda \Ep{a}{(\Qproxy\bef + \Df{\Qproxy})^2 - (\Qproxy\bef)^2} + \mu \Ep{s}{(\Qproxy\bef + \Df{\Qproxy})^2 - (\Qproxy\bef)^2} \\
    \label{eq:gg1-smoothed-bar:expand-quadratic}
    &= \Ep{\pi}{2\df{\Qproxy} \Qproxy} + \lambda \Ep{a}{2(\Df{\Qproxy}) \Qproxy\bef} + \mu \Ep{s}{2(\Df{\Qproxy})\Qproxy\bef}  + \lambda \Ep{a}{(\Df{\Qproxy})^2} +  \mu \Ep{s}{(\Df{\Qproxy})^2} \\
    \label{eq:gg1-smoothed-bar:substitute-drift}
    &= -2(\mu-\lambda) \Ep{\pi}{\Qproxy} + 2\mu \Ep{s}{\indibrac{{Q\bef=0}}\Qproxy\bef} + \lambda \Ep{a}{(\Df{\Qproxy})^2} +  \mu \Ep{s}{(\Df{\Qproxy})^2},
\end{align}
where we expand the squares in \eqref{eq:gg1-smoothed-bar:expand-quadratic} and substitute the values of $\df{\Qproxy}$ and $\E{\Df{\Qproxy} \given \Qproxy\bef}$ to get \eqref{eq:gg1-smoothed-bar:substitute-drift}. 
We move $-2(\mu-\lambda) \Ep{\pi}{\Qproxy}$ in \eqref{eq:gg1-smoothed-bar:substitute-drift} to the left-hand side and substitute $\Ep{a}{(\Df{\Qproxy})^2} = \Var{\lambda A}$ and $\Ep{s}{(\Df{\Qproxy})^2} = \Var{\mu S} + 1-\rho$, which yields
\[
    2(\mu-\lambda) \Ep{\pi}{\Qproxy} = 2\mu \Ep{s}{\indibrac{{Q\bef=0}}\Qproxy\bef} + \lambda \Var{\lambda A} +  \mu \big(\Var{\mu S} + 1-\rho\big). 
\]
Because $Q = \Qproxy-(\mu R_s - \lambda R_a)$, further rearranging the terms yields
\begin{equation}
    \label{eq:gg1-bar}
    \begin{aligned}
        \Ep{\pi}{Q} &=  \frac{\rho\Var{\lambda A}+ \Var{\mu S} + 1 - \rho}{2(1-\rho)}   \\
        &\mspace{23mu}  +  \frac{1}{1-\rho}\Ep{s}{\indibrac{Q\bef=0}(\mu R_{s}\bef - \lambda R_a\bef)} - \Ep{\pi}{\mu R_{s}  - \lambda R_a}. 
    \end{aligned}
\end{equation}
It is not hard to see that the second and third terms above are low-order terms: in particular, the second term is non-positive since $R_s\bef = 0$ under $\Ep{s}{\cdot}$; the third term is bounded by $O(1)$. 
We skip presenting the explicit bounds for these terms. 

Equations like \eqref{eq:gg1-first-attempt} and \eqref{eq:gg1-smoothed-bar:original} follow from the Basic Adjoint Relationship (BAR); we will formalize them in \Cref{sec:steady-state-anaysis:bar}, following the framework in \citep[][]{Bra_23_jsq}. 
The ``smoothed'' test function $\Qproxy^2$ is adapted from \citet{miyazawa_diffusion_2015}; 
this test function and its variants have been the core of many recent advances in the steady-state analysis of queueing systems with general distributions \citep[see, e.g.,][]{miyazawa_diffusion_2015,braverman_heavy_2017,BraDaiMiy_23_bar_ss,DaiGuaXu_24_jsq,DaiGlyXu_25_multiscale_jackson,DaiHuo_24_multiscale_sbp,GuaCheDai_25_uniform_bdd_multiscale,hong_performance_2023,GroHonHarSch_24_reset,BraScu_24_gg1}. 

\subsection{\ggn}
\label{sec:tech-overview-ggn}

Now we are ready to consider the $GI/GI/n$ queue. For the purpose of deriving an upper bound for $\mathbb{E}_{\pi}[Q]$, we focus on the modified $GI/GI/n$ system and refer to it simply as $GI/GI/n$ for brevity. We omit the superscript ``$-$'' on random variables within the Palm expectations $\mathbb{E}_{s}[\cdot]$ or $\mathbb{E}_{a}[\cdot]$ when the context is clear.

The \ggn queue can be viewed as ``partly'' a \ggone: when all servers are busy $(Q> 0)$, the total service rate is constant, and the $n$ servers act as a single ``super-server''. However, \ggn deviates from the super-server ideal at the boundary ($Q=0$), during which virtual jobs could be added. 
Consequently, upper-bounding the queue length of \ggn boils down to quantifying how much this boundary effect could increase the queue length in the steady state.

To quantify the difference between \ggn and \ggone in queue lengths, we generalize the test function used for \ggone and substitute it in the BAR of \ggn. 
Specifically, consider a smoothed proxy for $Q$ given by:
\begin{equation}
    \label{eq:Qproxy-def-ggn-overview}
    \Qproxy \triangleq Q + \sum_{i=1}^n \mu R_{s,i} - \Lambda R_a,
\end{equation}
where $R_{s,i}$ is the residual service time of server~$i$. Expanding the BAR for the quadratic test function $\Qproxy^2$ and following the same steps as in \eqref{eq:gg1-smoothed-bar:expand-quadratic}--\eqref{eq:gg1-smoothed-bar:substitute-drift}, we get a decomposition of $\Ep{\pi}{Q}$ that generalizes the single-server result \eqref{eq:gg1-bar} (see \Cref{sec:pf-main:apply-bar-quadratic} for the detailed calculation):
\begin{equation}
    \label{eq:ggn-bar-overview}
    \mathbb{E}_{\pi}[Q] = \frac{\rho\text{Var}(\Lambda A) + \text{Var}(\mu S) + 1-\rho}{2 (1-\rho)} + \frac{1}{1-\rho} \bigg( \sum_{j=1}^n \Gamma_{s,j} - \Gamma_a \bigg),
\end{equation}
where $\Gamma_{s,j}$ and $\Gamma_{a}$ are the boundary terms, defined as:
\begin{align}
    \label{eq:gamma-s-def-overview}
    \Gamma_{s,j} &\triangleq \Ep{s}{\indibrac{Q=0} \mu R_{s,j}} - (1-\rho) \Ep{\pi}{\mu R_{s,j}}, \\
    \label{eq:gamma-a-def-overview}
    \Gamma_{a} &\triangleq \Ep{s}{\indibrac{Q=0} \Lambda R_a} - (1-\rho) \Ep{\pi}{\Lambda R_{a}},
\end{align}
where $\Ep{s}{\cdot}$ is the expectation corresponding to the long-run averaging at the completion events of \emph{all servers}. 
Contrasting with \eqref{eq:gg1-bar}, we see that the first term on the right-hand side of  \eqref{eq:ggn-bar-overview} is approximately the mean queue length of \ggone with the super server. Therefore, $\Gamma_{s,j}$ and $\Gamma_{a}$ capture the difference between \ggn and the corresponding \ggone. 

A useful heuristic is to roughly view $\Gamma_{s,j}$ (or $\Gamma_a$) as the \emph{covariance} between the event $\{Q=0\}$ and the residual service (or arrival) time, if we assume $\Ep{s}{\cdot} \approx \Ep{\pi}{\cdot}$ and show that $\Probp{s}{Q=0} = 1-\rho$. 

Although these covariances are hard to calculate exactly, we will prove that they are small:
by their definitions, it suffices to show that $\{Q=0\}$ occurs with a low probability, or to show that the expected residual service (or arrival) times \emph{conditioned on} $Q=0$ are not significantly larger (or smaller) than the unconditional expectations. 

This heuristic view of $\Gamma_{s,j}$ and $\Gamma_a$ is also consistent with the following intuition about the mean queue length: the mean queue length of \ggn is close to that of \ggone if either (1) $Q$ rarely hits zero, so the boundary effect is naturally small, or (2) the next completion time right after $Q$ hits zero is not too far away, so $Q$ \emph{will not build up significantly before seeing the next completion}.

Motivated by the above discussion, we divide the rest of the proof into two cases, based on the relationship between $n$ and the load $\rho$, as they determine the probability of $Q$ hitting $0$:
\begin{itemize}
    \item Simpler case: $n \leq 1/(1-\rho)$. This corresponds to a heavily loaded regime ($\rho \geq 1 - 1/n$) where $\{Q=0\}$ occurs with a low probability in the steady state. This case has been studied more extensively in the literature. 
    \item Harder case: $n > 1/(1-\rho)$. This corresponds to a more lightly loaded regime  where $\{Q=0\}$ occurs with relatively higher probability in the steady state, so we need to more precisely understand the conditional expectation of the residual times, $\Ep{s}{R_{s,j} \given Q=0}$ and $\Ep{s}{R_a \given Q=0}$. 
    Addressing this case is the \emph{main technical challenge} of this work. 
\end{itemize}

\subsubsection{Simpler case: $GI/GI/n$ with $n \leq 1/(1-\rho)$}
\label{sec:tech-overview-ggn-simple} 
In this case, it suffices to show
\begin{equation}
    \label{eq:ggn-gamma-small-n-bound}
    \sum_{j=1}^n \Gamma_{s,j} \leq C (1-\rho)n,
\end{equation}
for some constant $C$ independent of $\rho$ and $n$. Because $\Gamma_a \geq 0$, combined with \eqref{eq:ggn-bar-overview}, we get $\mathbb{E}_{\pi}[Q] = O(1/(1-\rho))$. To establish \eqref{eq:ggn-gamma-small-n-bound}, we first bound $\Gamma_{s,j}$ by $O(1-\rho)$ for each $j$. Specifically, we can show that
\begin{align}
    \nonumber
    \Gamma_{\sj} &= \Epbig{s}{\indibrac{Q=0} \mu R_{\sj}} - (1-\rho) \Ep{\pi}{\mu R_{\sj}} \\
    \nonumber
    &= \Epbig{s}{\mu R_{\sj} \givenbig Q = 0} \Probpbig{s}{Q=0} - (1-\rho) \Ep{\pi}{\mu R_{\sj}}  \\
    \label{eq:gamma-s-upper-bound-overview}
    &\leq (1-\rho) \Big(\Rsmax - \E{(\mu S)^2}/2\Big),
\end{align}
where \eqref{eq:gamma-s-upper-bound-overview} follows from the facts that $\Epplain{s}{\mu R_{\sj} \givenplain Q = 0} \leq \Rsmax$, $\Probp{s}{Q=0} = 1-\rho$, and  $\Ep{\pi}{\mu R_{\sj}} = \E{(\mu S)^2}/2$. This leads to the final bound:
\begin{equation}
\label{eq:ggn-small-n-final}
     \Ep{\pi}{Q}
    \leq  \frac{\rho\Var{\Lambda A} + \Var{\mu S} + 1-\rho}{2 (1-\rho)} + n \left(\Rsmax - \frac{\E{(\mu S)^2}}{2}\right) +  \left(\frac{\E{(\Lambda A)^2}}{2} - \Ramin\right),
\end{equation}
where the $-\Ramin$ term comes from a more refined upper bound of $-\Gamma_a$, obtained using an argument similar to \eqref{eq:gamma-s-upper-bound-overview}. 
The bound in \eqref{eq:ggn-small-n-final} holds for all $n$ and $\rho$, and when $n \leq 1/(1-\rho)$, it achieves the desired order, $O(1/(1-\rho))$, and coincides with the inequality stated in \Cref{thm:main-upper-bound}.

Bounds of order $O(1/(1-\rho))$ for small $n$ have been established in the literature for \ggn and its variants \citep[e.g.,][]{Mor_75,Loulou_73,Kol_74_heavy_2,scully_gittins_2020,GroHarSch_22_wfcs,HonWan_24_ss}. The proof sketch presented here follows a two-step structure common to this line of work: the relevant performance metric is first decomposed into a $1/(1-\rho)$ term with explicit coefficients and several boundary terms, which are subsequently shown to be small. 
For instance, the results in \citep{GroHarSch_22_wfcs} imply an $O(1/(1-\rho) + n)$ bound for $M/GI/n$, with the same coefficients for the $1/(1-\rho)$ term as in \eqref{eq:ggn-small-n-final}; a key step of their analysis (Lemma~3) utilizes a similar decomposition as \eqref{eq:ggn-bar-overview}. 
In the following section, we show that this two-step framework can yield an $O(1/(1-\rho))$ bound, if we adopt a more refined analysis of the boundary terms when $n > 1/(1-\rho)$.

\subsubsection{Harder case: $GI/GI/n$ with $n > 1/(1-\rho)$}
\label{sec:tech-overview-hard}
We show for any $n > 1/(1-\rho)$,
\begin{equation}
    \label{eq:ggn-gamma-large-n-bound}
    \sumj \Gamma_{\sj} \leq C,
\end{equation}
for some constant $C$ independent of $\rho$ and $n$, 
so $\Ep{\pi}{Q} = O(1/(1-\rho))$. By symmetry, \eqref{eq:ggn-gamma-large-n-bound} is equivalent to
\begin{equation}
    \label{eq:ggn-gamma-large-n-bound-2}
    \Gamma_{\sone} \triangleq \Ep{s}{\indibrac{Q=0} \mu R_{\sone}} - (1-\rho) \Ep{\pi}{\mu R_{\sone}} \leq \frac{C}{n},
\end{equation}
where we take $j=1$ to focus on the residual service time of the first server, $R_{\sone}$.

To bound $\Gamma_{\sone}$, the naive bound in \eqref{eq:gamma-s-upper-bound-overview} is no longer helpful as it only gives $\Gamma_{\sone} \leq C(1-\rho)$, which is too loose when $n$ is much larger than $1/(1-\rho)$.

To prove a more refined bound, recall the intuition that $\Gamma_{\sone}$ is roughly the covariance between $\indibrac{Q=0}$ and $R_{\sone}$. We further observe that $\indibrac{Q=0}$ and $R_{\sone}$ are \emph{asymptotically independent} when $n$ gets large: Intuitively, the queue length $Q$ in a modified \ggn queue is the joint effect of $n$ independent completion processes, so each process should only have a $1/n$ fraction of ``responsibility'' for the realization of $\indibrac{Q = 0}$. 

To quantify the effect of the first server's completion process on the queue length, consider the following motivating question: 
\begin{center}
    How would the queue length change if a server had not existed?
\end{center}

An attempt to answer this question inspires a clean proof of \eqref{eq:ggn-gamma-large-n-bound-2}. 
We consider the counterfactual queue length, $\Qdropone$, if the first server had not existed in the modified \ggn queue from the beginning, but all arrivals and completions of other servers had remained the same (see \Cref{sec:leave-one-out} for the precise definition of $\Qdropone$). Then $\Qdropone$ is independent of $R_{\sone}$, and $\Qdropone \geq Q$. 
Consider the indicator variable $\indibrac{\Qdropone=0}$, which has two key properties: 
\begin{enumerate}
    \item[Property 1:] $\indibrac{\Qdropone=0}$ is independent of $R_{\sone}$.
    \item[Property 2:] $\indibrac{Q=0} \geq \indibrac{\Qdropone=0}$ and $\Ep{s}{\indibrac{Q=0} - \indibrac{\Qdropone=0}} = O(1/n)$.
\end{enumerate}
In other words, $\indibrac{Q=0}$ is approximated by a random variable independent of $R_{\sone}$, which captures our intuition of asymptotic independence. 

To formally show $\Gamma_{\sone} \leq C/n$, consider the decomposition:
\begin{align}
    \Gamma_{\sone}  
    \label{eq:pf-main:leave-one-out-decomposition-overview-term-1}
    &=\Ep{s}{\Big(\indibrac{Q=0} -  \indibrac{\Qdropone=0}\Big)\mu R_{\sone}} \\
    \label{eq:pf-main:leave-one-out-decomposition-overview-term-2}
    &\mspace{23mu} + \Ep{s}{\indibrac{\Qdropone=0} \mu R_{\sone}}  - (1-\rho) \Ep{\pi}{\mu R_{\sone}}.
\end{align}
The term on the right-hand side of \eqref{eq:pf-main:leave-one-out-decomposition-overview-term-1} can be bounded by $\Rsmax/n$ using Property 2, and the terms in \eqref{eq:pf-main:leave-one-out-decomposition-overview-term-2} can be explicitly calculated to be $-\E{(\mu S)^2}/(2n)$ with Property 1. Putting together, we get
\[
    \Gamma_{\sone} \leq \frac{1}{n} \Big(\Rsmax -\E{(\mu S)^2}/2\Big),
\]
which proves \eqref{eq:ggn-gamma-large-n-bound-2} and thus \Cref{thm:main-upper-bound} when $n > 1/(1-\rho)$. The rigorous argument can be found in \Cref{sec:leave-one-out-core-argument}.

\subsection{Roadmap for the rigorous proof}
\label{sec:tech-overview-roadmap}

\Cref{sec:lemmas-homo,sec:pf-main-upper-bound} provide the formal development of the arguments outlined in this overview. We focus on the $n > 1/(1-\rho)$ case as it constitutes the main technical contribution of this paper; we skip the $n \leq 1/(1-\rho)$ case as all important arguments have been provided in \Cref{sec:tech-overview-ggn-simple}. Note that our paper is still self-contained, as the proof for all regimes is provided for the heterogeneous generalization (\Cref{sec:hetero}). 

Here is a detailed breakdown of the content in \Cref{sec:lemmas-homo,sec:pf-main-upper-bound}. 
\Cref{sec:lemmas-homo} establishes preliminary definitions and supporting lemmas:
\begin{itemize}
    \item \Cref{sec:leave-one-out}: The definition of \emph{leave-one-out systems}, formalizing $\Qdropone$ used in \Cref{sec:tech-overview-hard}. 
    \item \Cref{sec:steady-state-anaysis:bar}: Basic Adjoint Relationship and relevant concepts (e.g., $\Ep{s}{\cdot}$ and $\Ep{a}{\cdot}$). 
    \item \Cref{sec:other-supporting-lemmas}: Statements of intuitive facts and supporting properties used in the proof sketch, such as the values of $\Probp{s}{Q=0}$ and $\Ep{\pi}{R_a}$, the bound $\Ep{s}{\mu R_{\sj}\given Q=0} \leq \Rsmax$, and the independence between $Q_{-1}$ and $R_{s,1}$. 
\end{itemize}
\Cref{sec:pf-main-upper-bound} contains the central proof, which consists of two parts: 
\begin{itemize}
    \item \Cref{sec:pf-main:apply-bar-quadratic}: Proving the decomposition of $\Ep{\pi}{Q}$ stated in \eqref{eq:ggn-bar-overview}. 
    \item \Cref{sec:leave-one-out-core-argument}: Formalizing the argument in \Cref{sec:tech-overview-hard} to bound $\Gamma_{\sj}$ and $\Gamma_{a}$. 
\end{itemize}

Finally, \Cref{sec:hetero} generalizes these arguments to \ggn with heterogeneous servers, following essentially the same set of ideas. 

The appendix of the paper contains the proofs of some supporting lemmas. Although these proofs may appear to be long, they prove intuitive facts, so the readers should be able to understand the main body of this paper without reading the appendices.

\section{Technical framework and supporting lemmas}
\label{sec:lemmas-homo}
In this section, we formally set up the technical framework and state the supporting lemmas for \Cref{thm:main-upper-bound}. 
We define the leave-one-out systems in \Cref{sec:leave-one-out}, set up the Basic Adjoint Relationship in \Cref{sec:steady-state-anaysis:bar}, prove some additional supporting lemmas in \Cref{sec:other-supporting-lemmas}. 

To simplify the presentation, we adopt the following assumption. Our results hold without this assumption, and we will outline the necessary modifications for the general case in \Cref{remark:simultaneous}. %

\begin{assumption}
    \label{assump:simultaneous}
    In the modified \ggn queue, no two events (i.e., an arrival and a completion, two arrivals, or two completions) happen at the same time.
\end{assumption}

In addition, unless otherwise stated, we assume $\rho < 1-1/n$. 

\subsection{Leave-one-out systems}
\label{sec:leave-one-out}
\ifoutline
\begin{outline}
    \1 Definition. Introduce notation $\Qdropj$.
    \1 Upper bound
    \1 Comment on stability
\end{outline}
\else\fi

In this section, we construct a set of auxiliary systems, referred to as \emph{leave-one-out systems}. Each leave-one-out system maintains a counterfactual queue length if a certain server were removed from the main $n$-server modified \ggn system under study (referred to as the ``main system'' or simply ``the modified \ggn queue'' from now on). 

As illustrated in \Cref{fig:leave-one-out-system}, 
for each $j\in[n]$, the $j$-th leave-one-out system is a modified $GI/GI/(n-1)$ system coupled with the main system, and it shares all arrival events and the completion events of all but the $j$-th server with the main system. 
More concretely, each job arrival in the main system also increases $\Qdropj(t)$ by $1$ for each $j\in[n]$, and each completion of the $i$-th server reduces $\Qdropj(t)$ by $\indibracbig{\Qdropj(t-) > 0}$ for $i,j\in[n]$ such that $i\neq j$. 
We initialize $\Qdropj(0) = Q(0)$ all $j\in[n]$, and denote the joint queue lengths of all leave-one-out systems as $\Qdrop(t) \triangleq (\Qdropj(t))_{j\in[n]}$.

\begin{figure}
    \includegraphics[width=12cm]{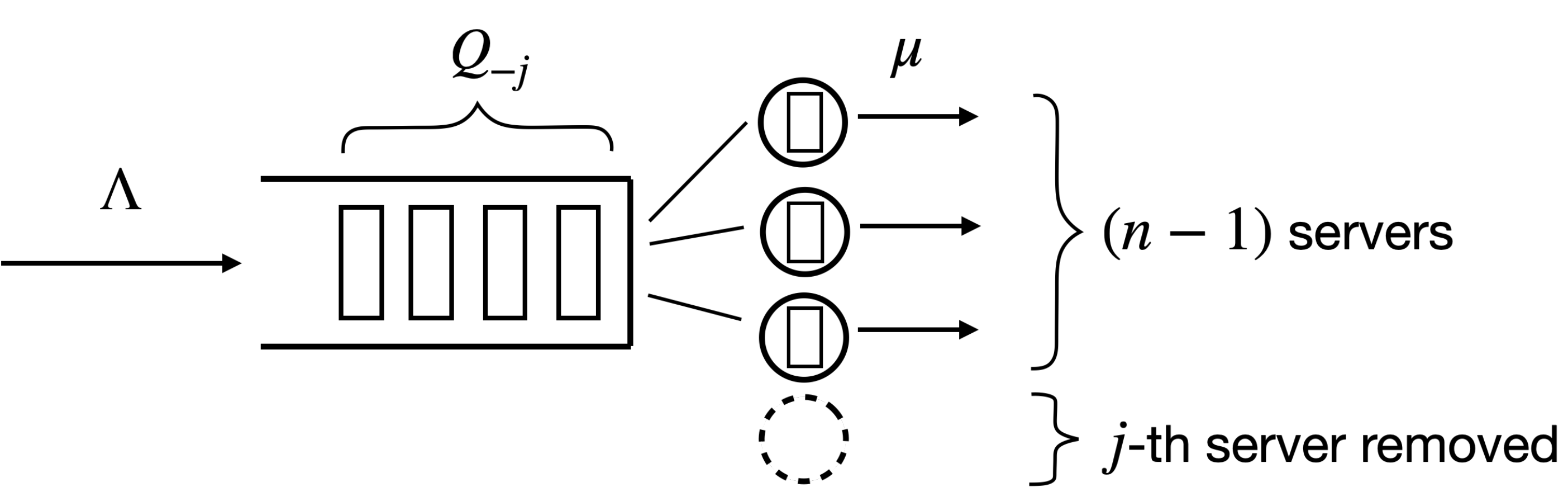}
    \centering
    \caption{An illustration of the $j$-th leave-one-out system, where the $j$-th server is removed and does not take jobs from the queue.}
    \label{fig:leave-one-out-system}
\end{figure}

As one would expect, the queue length of each leave-one-out system pathwise dominates the queue length of the main system: 

\begin{lemma}\label{lem:leave-one-out-upper-bound}
    For each $j\in[n]$,
    the queue length of the $j$-th leave-one-out system pathwise upper bounds the queue length of the main system, i.e., 
    \[
        \Qdropj(t) \geq Q(t) \quad \forall t\geq 0.
    \]
\end{lemma}

\begin{proof}
    We define the stopping time $\tau = \inf\{t \geq 0\colon \Qdropj(t) \leq Q(t) - 1\}$. It suffices to show that $\tau = \infty$ for all sample paths. First, because $\Qdropj(0) = Q(0)$, by the right-continuity of $\Qdropj(t)$ and $Q(t)$, 
    we must have $\tau  > 0$. 
    Suppose there exists a sample path such that $\tau < \infty$, then we have $\Qdropj(\tau-) = Q(\tau-)$ and $\Qdropj(\tau) \leq Q(\tau) - 1$.
    This implies that at time $\tau$, there must be at least a completion event that affects the $j$-th leave-one-out system but not the main system, or an arrival event that affects the main system but not the $j$-th leave-one-out system. As neither of these cases is possible, we get a contradiction. Consequently, for all sample paths, $\tau = \infty$, and thus $\Qdropj(t) \geq Q(t)$ for all $t\geq 0$. 
    \qedhere
\end{proof}

Finally, by \Cref{assump:modified-ggn-stable}, a leave-one-out system is positive Harris recurrent if $\rho < 1-1/n$.

\subsection{Basic Adjoint Relationship}
\label{sec:steady-state-anaysis:bar}
\ifoutline
\begin{outline}
    \1 Define the stationary distribution $\pi$ on the space of $\spacemain$ or $\spacesuper\times \Znn^n$; it exists by assumption
    \1 What BAR does and how it works
        \2 Focus on the stationary setting
        \2 We can analyze stationary distribution using stationarity of expectation with test function $f$ s.t. certain regularity conditions
        \2 Intuition: continuous change, and discrete jumps are zero in the long-run
        \2 Continuous change: differentiable function, and the operator $\opinner$
        \2 Discrete jumps: at each jump event. quantify by Palm distributions below. 
    \1 Palm distributions
        \2 Define as a measure over  $\spacesuper^2$, using $\sdsuper$
        \2 Why it sums to one (rate of the event is ...)
        \2 Why it is the same for any $h$ (stationarity)
    \1 BAR
        \2 Lemma statement, involving BAR, and the form of $\Xsuper\aft$
        \3 Comment on simultaneous events        
\end{outline}
\else\fi

As outlined in \Cref{sec:tech-overview}, we adopt the Basic Adjoint Relationship (BAR) approach to perform the steady-state analysis. In this section, we formally set up the preliminaries of BAR, following the framework in \citep[][Section~6]{BraDaiMiy_23_bar_ss}.

We will establish BAR for the state-augmented Markov process $\{\Xsuper(t)\}_{t\geq 0}$, defined as:
\begin{equation}
    \label{eq:xsuper-def}
    \Xsuper(t) \triangleq (R_{\sone}(t), R_{\stwo}(t), \dots R_{\sn}(t), R_a(t), Q(t),\Qdropone(t),\dots, \Qdropn(t)),
\end{equation}
with the state space $\spacesuper \triangleq \Rnn^{n+1}\times \Znn^{n+1}$. 
Since we assume $\rho < 1 - 1/n$ throughout \Cref{sec:lemmas-homo,sec:pf-main-upper-bound}, \Cref{assump:modified-ggn-stable} ensures that $\{\Xsuper(t)\}_{t\geq 0}$ possesses a unique stationary distribution $\pi$. 
From this point forward until the end of \Cref{sec:pf-main-upper-bound}, we slightly abuse the notation to let $\Xsuper(0)\sim\pi$, so that $\Xsuper(t)\sim\pi$ for all $t\geq 0$. 

Next, we formalize the three ways in which $\Xsuper(t)$ evolves: the continuous decrease of residual times, and the discrete jumps associated with arrivals and service completions.

To characterize the continuous changes, 
we define the differential operator $\mathcal{A}$ as
\begin{equation}
    \opinner f(\xsup) \triangleq - \sum_{\ell=1}^{n+1} \partial^+_{\ell} f(\xsup) \quad \forall \xsup\in\spacesuper,
\end{equation}
where $\partial^+_{\ell}$ denotes the right partial derivative with respect to the $\ell$-th coordinate, and $f\colon \spacesuper \to \R$ is any function that has right partial derivatives with respect to the first $(n+1)$ coordinates (the coordinates corresponding to $R_{\si}$'s and $R_a$). 
This operator captures the instantaneous rate of change of $f(\Xsuper(t))$ when $\Xsuper(t)=\xsup$.

Next, we consider changes in $\Xsuper(t)$ caused by the arrivals. Since arrivals occur at discrete points, we cannot define their rates. 
Instead, we use the so-called \emph{Palm distribution}, which is a probability measure over $\spacesuper^2$. The Palm distribution captures the steady-state probability of the pair of states before and after a random jump. 
Formally, the \emph{Palm distribution at arrivals} $\Pnb_a$ is defined as
\begin{equation}
    \label{eq:palm-a-def}
     \Probp{a}{B} = \frac{1}{\Lambda h} \E{\sum_{k=1}^\infty \indibrac{\Big(\Xsuper(\tau_a^k), \Xsuper(\tau_a^k-)\Big) \in B} \indibrac{\tau_a^k \leq  h}}, 
\end{equation}
where $B$ is any measurable subset of $\spacesuper^2$ and $\tau_a^k$ is the time of the $k$-th arrival event; $h$ is any positive number, and $\Pnb_a$ does not depend on the choice of $h$ due to the stationarity of $\Xsuper(t)$ \citep[][Theorem~VII.6.3]{Asm_03}. 
Roughly speaking, $\Pnb_a$ represents the long-run average distribution evaluated at arrival events; this can be seen by taking a large $h$, and note that $\Lambda$ is the long-term average frequency of arrival events. 
We refer the readers to \citep[][Section~VII.6]{Asm_03} or \citep[][Section~6]{BraDaiMiy_23_bar_ss} for a detailed background of Palm distributions induced by point processes.  

Given the Palm distribution $\Pnb_a$, we can define the corresponding Palm expectation $\Ep{a}{\cdot}$: 
For any measurable function $g\colon \spacesuper^2 \to \R$, we have
\begin{equation}
    \Ep{a}{g\big(\Xsuper\aft, \Xsuper\bef\big)} \triangleq \frac{1}{\Lambda h} \E{\sum_{k=1}^\infty g\Big(\Xsuper(\taua^k), \Xsuper(\taua^k-)\Big) \indibrac{\taua^k \leq h}},
\end{equation}
where $(\Xsuper\aft, \Xsuper\bef)$ on the left-hand side is a pair of dummy variables following the distribution $\Pnb_{a}$, and $h$ is any positive number. 

We similarly define $\Pnb_{\si}$ and $\Ep{\si}{\cdot}$ for the completion events of the $i$-th server for $i\in[n]$. 

Now we are ready to state the main lemma for the BAR approach, which connects the stationary distribution $\pi$ with the Palm distributions $\Pnb_a$ and $\Pnb_{\si}$ defined above, and characterize these Palm distributions. 
The proof of \Cref{lem:bar} is provided in \smartref{app:pf-BAR}{Section}{\valAppBAR}, which is adapted from \citep[][Section 6.1]{BraDaiMiy_23_bar_ss}.

\begin{restatable}[Basic Adjoint Relationship for \ggn]{lemma}{bar}
     \label{lem:bar}
     Consider the modified \ggn system satisfying Assumptions~\ref{assump:bdd-exp-remain}, \ref{assump:non-lattice} and \ref{assump:modified-ggn-stable}, with load $\rho < 1 - 1/n$. 
    For any function $f\colon \spacesuper\to \R$ that is bounded, locally Lipschitz continuous, and partially differentiable from the right with respect to its first $(n+1)$ coordinates, we have 
    \begin{equation}
        \label{eq:bar}
        \Ep{\sdsuper}{\opinner f(\Xsuper)} + \Lambda \Ep{a}{f(\Xsuper\aft) -f(\Xsuper\bef)} +   \mu \sumi \Ep{\si}{f(\Xsuper\aft) - f(\Xsuper\bef)} = 0, 
    \end{equation}
    where $\Xsuper, \Xsuper\aft, \Xsuper\bef \in\spacesuper$ are random elements following the distributions specified in the subscripts of expectations. 
    Moreover, the Palm distributions $\Pnb_a$ and $\Pnb_{\si}$ satisfy the following with probability $1$:
   \begin{enumerate}%
        \item[(i)] \label{item:Pa-jumps} If $(\Xsuper\aft, \Xsuper\bef)\sim\Pnb_a$, then $R_{a}\bef = 0$ and the coordinates of $\Xsuper\aft$ and $\Xsuper\bef$ satisfy the relations:
        \begin{align}
            \label{eq:Pnba-char-ra}
            R_a\aft &= R_a\bef + A \\
            \label{eq:Pnba-char-q}
            Q\aft &= Q\bef + 1  \\
            \label{eq:Pnba-char-q-j}
            \Qdropj\aft &= \Qdropj\bef + 1 \quad \forall j\in[n],
        \end{align}
        where $A$ is a generic interarrival time independent of $\Xsuper\bef$; other coordinates of $\Xsuper\aft$ are identical to $\Xsuper\bef$. (Here and throughout, any superscript applied to an element of $\spacesuper$, such as $\Xsuper\bef$ or $\Xsuper\aft$, propagates to all of its coordinates.)
        \item[(ii)] \label{item:Pi-jumps} For each $i\in[n]$, if $(\Xsuper\aft, \Xsuper\bef)\sim\Pnb_{\si}$, then $R_{\si}\bef=0$ and the coordinates of $\Xsuper\aft$ and $\Xsuper\bef$ satisfy the relations:
        \begin{align}
            \label{eq:Pnbsi-char-rsi}
            R_{\si}\aft &= R_{\si}\bef + S \\
            \label{eq:Pnbsi-char-q}
            Q\aft &= Q\bef - \indibrac{Q\bef > 0} \\
            \label{eq:Pnbsi-char-q-j}
            \Qdropj\aft &= \Qdropj\bef - \indibrac{\Qdropj\bef > 0} \quad \forall j\in[n]\backslash\{i\},
        \end{align}
        where $S$ is a generic service time independent of $\Xsuper\bef$; other coordinates of $\Xsuper\aft$ are identical to $\Xsuper\bef$. 
    \end{enumerate}
\end{restatable}

We now provide a brief interpretation of the statement in \Cref{lem:bar}. 

Equation~\eqref{eq:bar} states that the long-run average rate of change of $f(\Xsuper(t))$ is zero. In this expression, $\Ep{\sdsuper}{\opinner f(\Xsuper)}$ represents the average rate of change due to the continuous drift of residual times, while the remaining terms capture the discrete changes associated with arrivals and service completions. 

Furthermore, Equations~\eqref{eq:Pnba-char-ra}--\eqref{eq:Pnbsi-char-q-j} characterize the conditional Palm distribution of the post-jump state $\Xsuper\aft$ given the pre-jump state $\Xsuper\bef$. 
The specific relations between $\Xsuper\aft$ and $\Xsuper\bef$ defined by these identities are consistent with the construction of the modified \ggn in \Cref{sec:modified-ggn}.

Next, we define some convenient notational shorthand. 
We define the Palm distribution $\Probp{s}{\cdot} \triangleq (1/n)\sumi \Probp{\si}{\cdot}$ and the corresponding expectation $\Ep{s}{g(\Xsuper\aft,\Xsuper\bef)} \triangleq \frac{1}{n}\sumi \Ep{\si}{g(\Xsuper\aft,\Xsuper\bef)}$. Under this notation, \eqref{eq:bar} can be compactly rewritten as 
\begin{equation}
    \label{eq:bar-Es-form}
    \Ep{\pi}{\opinner f(\Xsuper)} + \Lambda \Ep{a}{f(\Xsuper\aft) - f(\Xsuper\bef)} + n\mu \Ep{s}{f(\Xsuper\aft) - f(\Xsuper\bef)} = 0.
\end{equation}
We also adopt the simplification that, when a Palm expectation only involves the pre-event state, $\Xsuper\bef$, we omit its superscript and simply write $\Xsuper$.

\begin{remark}[Simultaneous events]
\label{remark:simultaneous}
    We have adopted the assumption that no events occur simultaneously (\Cref{assump:simultaneous}) to simplify the presentation. Nevertheless, our results still hold without this assumption, and we outline the necessary changes as follows. 
    
    The main changes are in the definitions of the Palm distributions. When $M(t) \geq 1$ events occur simultaneously at time $t$, we let $\Xsuper(t, 0), \dots, \Xsuper(t, M(t))$ denote the intermediate states produced by sequentially applying the state-updating subroutines (see \Cref{sec:modified-ggn}); in particular, $\Xsuper(t, 0) \triangleq \Xsuper(t-)$ and $\Xsuper(t, M(t)) \triangleq \Xsuper(t)$. Consider the sequence of all pairs $(t, m)$ such that $\Xsuper(t, m)$ is an intermediate state immediately preceding an arrival update. We order these pairs lexicographically and let $(\tau_a^k, m_a^k)$ denote the $k$-th element in this sequence. The definition of $\mathbb{P}_a$ is then generalized as:    
    \begin{equation}
        \label{eq:palm-a-def-simultaneous}
         \Probp{a}{B} = \frac{1}{\Lambda h} \E{\sum_{k=1}^\infty \indibrac{\Big(\Xsuper(\tau_a^k, m_a^k+1), \Xsuper(\tau_a^k, m_a^k)\Big) \in B} \indibrac{\tau_a^k \leq  h}}. 
    \end{equation}
    The Palm distributions $\Pnb_{\si}$ for $i\in[n]$ should be generalized analogously. Intuitively, $\Pnb_a$ and $\Pnb_{\si}$ still represent the average state distributions observed by arrival events or completion events, although the states at some time points could be counted multiple times or involved in multiple Palm distributions. The same definition has also been used in \citep{BraDaiMiy_23_bar_ss} to deal with simultaneous events.
    
    With the new definitions of $\Pnb_a$ and $\Pnb_{\si}$, the statement of \Cref{lem:bar} is valid without requiring any other changes. %
    The other lemmas, including those for heterogeneous \ggn in \Cref{sec:hetero}, also hold verbatim.
    We leave the detailed verifications to readers. 

    Finally, we mention an interesting detail. As we can see in \Cref{sec:modified-ggn}, when constructing the modified \ggn, the order of updating the state in the presence of simultaneous events could affect the system dynamics. 
    However, our generalized proof should work for any order of updates. 
    Intuitively, whether we apply the arrival subroutine before or after a completion subroutine would affect the distribution of $R_{a}\bef$ under $\Pnb_{\si}$, but not those facts stated in \Cref{lem:bar}.

\end{remark}

\subsection{Other supporting lemmas}
\label{sec:other-supporting-lemmas}

In this section, we state the remaining supporting lemmas for the proof of \Cref{thm:main-upper-bound}. These lemmas formalize several intuitive properties of the system; we briefly discuss their intuitions and defer their proofs to \smartref{app:technical-lemmas}{Section}{\valAppSupp}.

The first lemma, \Cref{lem:R-bdd}, bounds the conditional mean residual times given the queue lengths, $Q$ and $\Qdrop \triangleq (\Qdropj)_{j\in[n]}$. It is proved in \smartref{app:pf-Rbdd}{Section}{\valAppBdd}.

\begin{restatable}[Bounded mean residual times]{lemma}{Rbdd}\label{lem:R-bdd}
Consider the modified \ggn queue satisfying Assumptions~\ref{assump:bdd-exp-remain} and \ref{assump:modified-ggn-stable}, with the load $\rho < 1-1/n$. For any $j\in[n]$, 
    \begin{align}
        \label{eq:R-sj-bound-stationary}
        \mu\Ep{\nu}{\Rsj \given Q, \Qdrop} &\leq \Rsmax \\ 
         \label{eq:R-a-bound-stationary}
        \Lambda\Ep{\nu}{R_a \given Q, \Qdrop} &\geq \Ramin 
    \end{align}
    where $\Ep{\nu}{\cdot}$ can denote either $\Ep{\pi}{\cdot}$ or $\Ep{\si}{\cdot}$ for $i\in [n]\backslash \{j\}$. 
\end{restatable}

The intuition behind \Cref{lem:R-bdd} is that $\Rsmax$ is defined to be an upper bound on the mean residual service time conditioned on any \emph{age} (the time elapsed since the last service completion). Because the service times are i.i.d., this age is a sufficient statistic for the mean residual service time. Consequently, $Q$ and $\Qdrop$, which are functions of past arrival and completion times, provide no additional predictive information once the age is known; thus, conditioning on them does not change the bound.  
A similar logic applies to the lower bound $\Ramin$ for the arrival process.

The next lemma establishes that for any $i,j\in[n]$ such that $i\neq j$, $R_{\sj}$ and $\indibrac{\Qdropj=0}$ are uncorrelated under the Palm distribution $\Pnb_{\si}$. It is proved in \smartref{app:pf-uncorrelation}{Section}{\valAppIndep}. 

\begin{restatable}[Uncorrelation lemma]{lemma}{indep}
\label{lem:indep-R-Qtilde}
Consider the modified \ggn queue satisfying Assumptions~\ref{assump:bdd-exp-remain}, \ref{assump:non-lattice}, and \ref{assump:modified-ggn-stable}, with the load $\rho < 1-1/n$. For any $i, j\in[n]$ such that $i\neq j$, we have
\begin{equation}
     \Ep{\si}{\indibrac{\Qdropj=0}\Rsj} = \Probp{\si}{\Qdropj=0}\Ep{\pi}{\Rsj}.
\end{equation}
\end{restatable}

Intuitively, $R_{\sj}$ depends only on the $j$-th server's completion process, whereas $\Qdropj$ and $\Pnb_{\si}$ are determined by the other $n$ renewal processes; these $n+1$ renewal processes are mutually independent. 
The formal proof of \Cref{lem:indep-R-Qtilde} also utilizes a mixing argument for the completion processes, which relies on $S$ being non-lattice (\Cref{assump:non-lattice}). Notably, this lemma is the only part of our analysis where the non-lattice assumption is required.

The next lemma shows several basic equalities about the residual times of the renewal processes and the idle probability of the queue. 
This lemma is proved in \smartref{app:pf-equalities-and-finiteness}{Section}{\valAppEquality}. 

\begin{restatable}{lemma}{applybarpreliminary}
\label{lem:apply-bar-lots-of}
    Consider the modified \ggn queue satisfying Assumptions~\ref{assump:bdd-exp-remain} and \ref{assump:modified-ggn-stable}. If $\rho < 1$, 
    for $i,j\in[n]$ such that $i\neq j$, we have
    \begin{align}
        \label{eq:E-pi-R-s}
        \Ep{\sdsuper}{\Rsi} &= \frac{1}{2}\mu\E{S^2} \\
        \label{eq:E-pi-R-a}
        \Ep{\sdsuper}{R_a} &= \frac{1}{2}\Lambda \E{A^2} \\
        \label{eq:E-s-R-s}
        \Ep{\si}{\Rsj} + \Ep{\sj}{\Rsi} &= \mu\E{S^2} \\
        \label{eq:E-sa-R-sa}
        \Ep{\si}{R_a} + \Ep{a}{\Rsi} &= \frac{1}{2}\mu\E{S^2} + \frac{1}{2}\Lambda \E{A^2} \\
        \label{eq:E-s-Q}
        \Probp{s}{Q=0} &= 1 - \rho
    \end{align}
    In addition, if $\rho < 1-1/n$, for each $j \in [n]$, the $j$-th leave-one-out system satisfies
    \begin{equation}
        \label{eq:E-s-Qtilde}
        \frac{1}{n}\sum_{i\colon i\neq j} \Probp{\si}{\Qdropj=0} = 1- \rho - \frac{1}{n}.
    \end{equation}
\end{restatable}

Intuitively, \eqref{eq:E-pi-R-s} and \eqref{eq:E-pi-R-a} characterize the mean residual times \emph{under the steady-state distribution}, while \eqref{eq:E-s-R-s} and \eqref{eq:E-sa-R-sa} characterize these means \emph{observed by independent renewal processes}. 
For instance, $\Ep{\si}{R_{\sj}}$ represents the mean residual time of the $i$-th service process observed at the completion times of the $j$-th service process, for some $j$ different from $i$; \eqref{eq:E-s-R-s} implies that $\Ep{\si}{R_{\sj}}$ matches the steady-state mean $\mu \E{S^2} / 2$ under the assumption that $\Ep{\si}{R_{\sj}} = \Ep{\sj}{R_{\si}}$, although our analysis requires neither the proof nor the assumption of such symmetry.

\Cref{eq:E-s-Q} identifies the fraction of service completions that occur when $Q=0$; these completions are ``wasted'' in the sense that a virtual job is immediately added upon the completion of the previous job. Since the ratio of the arrival rate to the total service rate is $\rho$, exactly a $(1-\rho)$ fraction of the completions must be wasted in the long run to maintain stability. A similar intuition also applies to \eqref{eq:E-s-Qtilde}.

Finally, \Cref{lem:expfinite} establishes that the queue length $Q$ has a finite mean under the Palm distributions $\mathbb{P}_a$ and $\mathbb{P}_{s,i}$. This result follows from the finiteness of $\Ep{\pi}{Q}$ (\Cref{assump:modified-ggn-stable}). 
This lemma is proved in \smartref{app:pf-equalities-and-finiteness}{Section}{\valAppEquality}.

\begin{restatable}{lemma}{expfinite}
\label{lem:expfinite}
    Consider the modified \ggn system satisfying Assumptions~\ref{assump:bdd-exp-remain} and \ref{assump:modified-ggn-stable}, with the load $\rho < 1$. Then we have $\Ep{a}{Q} < \infty$ and $\Ep{\si}{Q} < \infty$ for all $i\in[n]$. 
\end{restatable}

\section{Proof of main theorem}
\label{sec:pf-main-upper-bound}

This section is dedicated to proving \Cref{thm:main-upper-bound}, following the two-step approach outlined in \Cref{sec:tech-overview}:
\begin{enumerate}%
    \item[Step 1] (\Cref{sec:pf-main:apply-bar-quadratic}): We derive an exact expression for the mean queue length, $\Ep{\pi}{Q}$. This expression is in terms of the load $\rho$, the variances $\Var{\mu S}$ and $\Var{\Lambda A}$, and some ``covariance'' terms, $\Gamma_{\sj}$ and $\Gamma_{a}$.  %
    \item[Step 2] (\Cref{sec:leave-one-out-core-argument}): We bound $\Gamma_{\sj}$ and $\Gamma_{a}$. 
\end{enumerate}
The first step follows from the established framework of BAR, while the second step contains the novel argument based on our leave-one-out technique. Substituting the bound in Step 2 into the equality in Step 1 proves \Cref{thm:main-upper-bound}.

\subsection{Step 1: exact expression of $\E{Q}$}
\label{sec:pf-main:apply-bar-quadratic}
\begin{lemma}
\label{lem:apply-bar-main-quadratic}
    Consider the modified \ggn system satisfying Assumptions~\ref{assump:bdd-exp-remain} and \ref{assump:modified-ggn-stable}, with the load $\rho < 1$. We have
    \begin{equation}
        \label{eq:apply-bar-main-quadratic-equality}
        (1-\rho) \Ep{\pi}{Q}
        =  \frac{\rho\Var{\Lambda A} + \Var{\mu S} + 1-\rho}{2} +  \sumj \Gamma_{\sj}  - \Gamma_a \;, 
    \end{equation}
    where $\Gamma_{\sj}$ and $\Gamma_{a}$ are defined as: 
    \begin{align}
        \label{eq:gamma-s-def}
        \Gamma_{\sj} &= \Ep{s}{\indibrac{Q=0} \mu R_{\sj}} - (1-\rho) \Ep{\pi}{\mu R_{\sj}} \quad \forall j\in [n]\\
        \label{eq:gamma-a-def}
        \Gamma_{a} &= \Ep{s}{\indibrac{Q=0} \Lambda R_a} - (1-\rho) \Ep{\pi}{\Lambda R_{a}},
    \end{align}
    and $\Ep{s}{\cdot} \triangleq \frac{1}{n} \sumi \Ep{\si}{\cdot}$ is the Palm expectation with respect to all completion events.
\end{lemma}

\begin{remark}
    When $n$ is large,
    the term $\Gamma_{\sj}$ can be roughly viewed as the covariance between $\indibrac{Q=0}$ and $\mu R_{\sj}$ under the measure $\Pnb_s$, because $\Probp{s}{Q=0} = 1-\rho$ and $\Ep{s}{R_{\sj}} = \frac{1}{n} \sum_{i\colon i\neq j} \Ep{\si}{R_{\sj}} = \frac{n-1}{n}\Ep{\pi}{R_{\sj}}$. 
    Here we use the fact that $\Ep{\si}{R_{\sj}} = \Ep{\pi}{R_{\sj}}$, which is slightly stronger than \eqref{eq:E-s-R-s} in \Cref{lem:apply-bar-lots-of}, but can be shown by replacing $\indibracbig{\Qdropj=0}$ with $1$ in the proof of \Cref{lem:indep-R-Qtilde}.  
    We skip the detailed argument because it is not needed to prove our main result. 

    The term $\Gamma_a$ can be interpreted similarly if the interarrival-time distribution is non-lattice. 
\end{remark}

Next, we prove \Cref{lem:apply-bar-main-quadratic}, focusing on presenting the case where $\rho < 1-1/n$. The case where $\rho \geq 1-1/n$ follows from identical algebraic steps, with the only modification being that the BAR (\Cref{lem:bar}) must be stated for the original state, $X(t)$, rather than the augmented state, $\Xsuper(t)$. This distinction arises because the augmented state involves $\Qdropj(t)$, which is unstable when $\rho \geq 1-1/n$; nevertheless, $\Qdropj(t)$ is not required for the proof of \Cref{lem:apply-bar-main-quadratic}.

\begin{proof}[Proof of \Cref{lem:apply-bar-main-quadratic}]
    We define the random variable 
    \[
        \Qproxy \triangleq Q + \mu \sum_{j=1}^n R_{s,j} - \Lambda R_a.
    \]
    As discussed in \Cref{sec:tech-overview-ggn}, $\Qproxy$ serves as a smoothed version of the queue length $Q$. 
    Because $\Qproxy$ possesses more tractable dynamics, we focus the majority of our proof on calculating $\Ep{\pi}{\Qproxy}$, from which the original goal of calculating $\Ep{\pi}{Q}$ can be straightforwardly recovered.

    To apply the Basic Adjoint Relationship (BAR, \Cref{lem:bar}) to the quadratic term $\Qproxy^2$, we must first address the requirement that the test function should be bounded. We thus consider the truncated version of $\Qproxy$, 
    \[
        \Qproxy_{\kappa} \triangleq (Q \wedge \kappa) + \mu \sum_{j=1}^n (R_{s,j} \wedge \kappa) - \Lambda (R_a \wedge \kappa)
    \]
    for some integer $\kappa > 0$. 
    
    Let $f(\Xsuper) = \Qproxy_{\kappa}^2$.
    Since $f$ is bounded, locally Lipschitz continuous, and has right partial derivatives with respect to $R_a$ and $R_{\si}$'s, BAR implies that
     \begin{align*}
         0 &= \Ep{\pi}{2\Qproxy_{\kappa}\,\Big(-\mu \sumj \indibrac{R_{\sj} \leq \kappa} + \Lambda \indibrac{R_a\leq \kappa} \Big)} \\
         &\mspace{23mu} +\Lambda \Ep{a}{\Big(\Qproxy_{\kappa} + \indibrac{Q\leq \kappa-1} - \Lambda A\wedge\kappa\Big)^2 - \Qproxy_{\kappa}^2} \\
         &\mspace{23mu} + n\mu \Ep{s}{\Big(\Qproxy_{\kappa} - \indibrac{0<Q\leq \kappa} + \mu S\wedge\kappa\Big)^2 - \Qproxy_{\kappa}^2}.
     \end{align*}
     Expanding the squares and rearranging the terms, the above equation simplifies into
     \begin{align*}
         &2\Ep{\pi}{\Qproxy_{\kappa} \,\Big(\mu \sumj \indibrac{R_{\sj} \leq \kappa} - \Lambda \indibrac{R_a\leq \kappa} \Big)} \\
         &\qquad = \Lambda \Ep{a}{2\Qproxy_{\kappa} \Big(\indibrac{Q\leq \kappa-1} - \Lambda A\wedge\kappa\Big) + \Big(\indibrac{Q\leq \kappa-1} - \Lambda A\wedge\kappa\Big)^2} \\
         &\qquad \mspace{23mu} + n\mu \Ep{s}{2\Qproxy_{\kappa} \Big(- \indibrac{0<Q\leq \kappa} + \mu S\wedge\kappa\Big) + \Big(- \indibrac{0<Q\leq \kappa} + \mu S\wedge\kappa\Big)^2}. 
     \end{align*}
     Now we want to let $\kappa \to\infty$ in the previous equation and apply the dominated convergence theorem (DCT). To apply DCT, recall the following finiteness results: $\Ep{\pi}{Q} < \infty$ (\Cref{assump:modified-ggn-stable}), $\Ep{a}{Q}, \Ep{s}{Q} < \infty$ (\Cref{lem:expfinite}), $\Ep{\pi}{R_a}, \Ep{\si}{R_a}, \Ep{\pi}{R_{\si}}, \Ep{a}{R_{\si}}, \Ep{\si}{R_{\sj}} <\infty$ (\Cref{lem:apply-bar-lots-of}), $\Ep{a}{R_a} = \Ep{\si}{R_{\si}} = 0$, and $\E{A^2},\E{S^2}<\infty$ (\Cref{assump:bdd-exp-remain}).
     Then it is not hard to see that, in the previous displayed equation, 
     the expression inside $\Ep{\pi}{\cdot}$ is dominated by 
     $(Q + \sumj \mu R_{\sj} + \Lambda R_a)(n\mu + \Lambda)$, which is independent of $\kappa$ and has a finite mean under the distribution $\pi$. 
     For similar reasons, the expressions inside $\Ep{a}{\cdot}$ and $\Ep{s}{\cdot}$ are dominated by quantities which are independent of $\kappa$ and have finite means under $\Pnb_a$ and $\Pnb_s$, respectively. Therefore, DCT implies that
     \begin{align}
         \label{eq:quadratic-bar:cont-term}
         &2(n\mu  - \Lambda)\Ep{\pi}{\Qproxy} \\
         \label{eq:quadratic-bar:arrival-term}
         &\qquad = \Lambda \Ep{a}{2\Qproxy\Big(1 - \Lambda A\Big) + \Big(1 - \Lambda A\Big)^2} \\
         \label{eq:quadratic-bar:comp-term}
         &\qquad \mspace{20mu} + n\mu \Ep{s}{2\Qproxy\Big(- \indibrac{Q>0} + \mu S\Big)  + \Big(- \indibrac{Q>0} + \mu S\Big)^2}. 
     \end{align}

     Next, we simplify the terms in \eqref{eq:quadratic-bar:arrival-term}, which correspond to the arrival events:
     \begin{align}
        \nonumber
        &\Lambda \Ep{a}{2\Qproxy\Big(1 - \Lambda A\Big) + \Big(1 - \Lambda A\Big)^2} \\
        &\qquad = 2\Lambda \Ep{a}{\Qproxy}\E{1 - \Lambda A}  + \Lambda \E{(1 - \Lambda A)^2} \\
         \label{eq:quadratic-bar:arrival-term-simplify}
         &\qquad = \Lambda \Var{\Lambda A},
     \end{align}
     where the first equality is due to the fact that $A$ is independent of $Q$ and $R_{\sj}$; the second equality is due to $\Lambda \E{A} = 1$. 
     
     We then simplify the terms in \eqref{eq:quadratic-bar:comp-term}, which correspond to the completion events: 
     \begin{align}
        \nonumber
         & n\mu \Ep{s}{2\Qproxy\Big(- \indibrac{Q>0} + \mu S\Big)  + \Big(- \indibrac{Q>0} + \mu S\Big)^2}  \\
         \nonumber
         &\qquad = 2n\mu \Ep{s}{\Qproxy}\,\E{- 1 + \mu S} + 2n\mu \Ep{s}{\Qproxy \indibrac{Q=0}}   \\
         \nonumber
         &\qquad\mspace{23mu} + n\mu \E{(- 1 + \mu S)^2} + 2n\mu \Probp{s}{Q=0}\E{-1+\mu S} + n\mu\Probp{s}{Q=0} \\
         \nonumber
         &\qquad=  2n\mu \Ep{s}{\Qproxy\indibrac{Q=0}} + n\mu \Var{\mu S} + n\mu\Probp{s}{Q=0} \\
         \label{eq:quadratic-bar:comp-term-simplify}
         &\qquad= 2n\mu \Ep{s}{\Qproxy\indibrac{Q=0}} + n\mu \Var{\mu S} + n\mu(1-\rho),
     \end{align}
     where the first equality uses $\indibrac{Q>0}=1-\indibrac{Q=0}$ and the fact that $S$ is independent of $Q$, $R_{\sj}$, and $R_a$; the second equality is due to $\E{\mu S} = 1$; the third equality is due to $\Probp{s}{Q=0} = 1-\rho$ (\Cref{lem:apply-bar-lots-of}). 

    Substituting \eqref{eq:quadratic-bar:arrival-term-simplify} and \eqref{eq:quadratic-bar:comp-term-simplify} back to  \eqref{eq:quadratic-bar:arrival-term} and \eqref{eq:quadratic-bar:comp-term}, we get 
    \[
        2(n\mu  - \Lambda)\Ep{\pi}{\Qproxy}  = 2n\mu \Ep{s}{\Qproxy \indibrac{Q=0}} + \Lambda \Var{\Lambda A} + n\mu \Var{\mu S} + n\mu(1-\rho). 
    \]
    Dividing both sides by $2n\mu$ and using the fact that $\Lambda / (n\mu) = \rho$, the previous equality becomes
    \[
        (1-\rho) \Ep{\pi}{\Qproxy}  =  \Ep{s}{\Qproxy\indibrac{Q=0}} + \frac{\rho\Var{\Lambda A} + \Var{\mu S} + 1 - \rho}{2}. 
    \]
    Because $Q = \Qproxy - \mu \sumj R_{\sj} + \Lambda R_a$, rearranging the terms, we get
    \begin{align*}
        (1-\rho) \Ep{\pi}{Q} 
        &=  \frac{\rho\Var{\Lambda A} + \Var{\mu S} + 1 - \rho}{2} \\
        &\mspace{23mu} + \Epbigg{s}{\bigg(\mu \sumj R_{\sj} - \Lambda R_a\bigg)\indibrac{Q=0}} - (1-\rho)\Epbigg{\pi}{\mu \sumj R_{\sj}  - \Lambda R_a}, 
    \end{align*}
    which implies \eqref{eq:apply-bar-main-quadratic-equality}.
    \qedhere
\end{proof}

\subsection{Step 2: analyzing the covariance terms}
\label{sec:leave-one-out-core-argument}
In this section, we bound the covariance terms, $\Gamma_{\sj}$ and $\Gamma_{a}$, for $j\in[n]$. 

\begin{lemma}
\label{lem:covariance-bounds}
    Consider the modified \ggn system satisfying Assumptions~\ref{assump:bdd-exp-remain} and \ref{assump:modified-ggn-stable}, with the load $\rho < 1$. Let $\Gamma_{\sj}$ and $\Gamma_{a}$ be the terms defined in \eqref{eq:gamma-s-def} and \eqref{eq:gamma-a-def}, respectively. Then 
    \begin{align}
        \label{eq:s-covariance-bound}
       \Gamma_{\sj} &\leq \min\Big(1-\rho, \frac{1}{n}\Big) \Big(\Rsmax - \frac{1}{2} \E{(\mu S)^2}\Big) \\
       \label{eq:a-covariance-bound}
       - \Gamma_{a} &\leq (1-\rho) \bigg(\frac{1}{2}\E{(\Lambda A)^2} - \Ramin\bigg) 
    \end{align}
\end{lemma}

When proving \Cref{lem:covariance-bounds}, we focus on presenting the case of $\rho < 1-1/n$, since the case of $\rho \geq 1-1/n$ is relatively well understood, with the main argument sketched in \Cref{sec:tech-overview-ggn-simple}.

\begin{proof}
Recall that
\begin{align}
    \tag{\ref{eq:gamma-s-def}}
    \Gamma_{\sj} &= \Ep{s}{\indibrac{Q=0} \mu R_{\sj}} - (1-\rho) \Ep{\pi}{\mu R_{\sj}} \quad \forall j\in [n]\\
    \tag{\ref{eq:gamma-a-def}}
    \Gamma_{a} &= \Ep{s}{\indibrac{Q=0} \Lambda R_a} - (1-\rho) \Ep{\pi}{\Lambda R_{a}},
\end{align}
We first bound $\Gamma_{\sj}$ and then bound $-\Gamma_a$. 

\paragraph*{Bounding $\Gamma_{\sj}$}

Observe that when $\rho < 1 - 1/n$, the load of the $j$-th leave-one-out system is less than $1$, so it has a well-defined stationary distribution (\Cref{assump:modified-ggn-stable}). 
We consider the indicator $\indibracbig{\Qdropj=0}$ as a ``baseline'' and subtract it from $\indibrac{Q=0}$. This leads to the following \emph{key decomposition} of $\Ep{s}{\indibrac{Q=0}\mu R_{\sj}}$: 
\begin{equation}
    \label{eq:pf-main:leave-one-out-decomposition}
    \Ep{s}{\indibrac{Q=0}\mu R_{\sj}}
    = \Ep{s}{\Big(\indibrac{Q=0} -  \indibrac{\Qdropj=0}\Big)\mu R_{\sj}} + \Ep{s}{\indibrac{\Qdropj=0} \mu R_{\sj}}. 
\end{equation}
We then bound the two terms on the right-hand side of \eqref{eq:pf-main:leave-one-out-decomposition} separately. 

We bound the first term on the right-hand side of \eqref{eq:pf-main:leave-one-out-decomposition} using similar arguments as in Case 1, i.e., by applying the conditional expectation bound of $R_{\sj}$ proved in \Cref{lem:R-bdd}: 
\begin{align}
    \Ep{s}{\Big(\indibrac{Q=0} -  \indibrac{\Qdropj=0}\Big)\mu R_{\sj}}
    \label{eq:pf-main:apply-zero-prejump-Rs}
    &= \frac{1}{n} \sum_{i\colon i\neq j}\Ep{\si}{\Big(\indibrac{Q=0} -  \indibrac{\Qdropj=0}\Big)\mu R_{\sj}} \\
    \label{eq:pf-main:apply-lower-bound-and-R-bound}
    &\leq  \frac{1}{n} \sum_{i\colon i\neq j}\Ep{\si}{\indibrac{Q=0} -  \indibrac{\Qdropj=0}}\Rsmax \\
    \label{eq:pf-main:relax-to-Ps}
    &\leq  \Big(\Probp{s}{Q=0} -  \frac{1}{n}\sum_{i\colon i\neq j} \Probp{\si}{\Qdropj=0}\Big)\Rsmax \\
    \label{eq:pf-main:apply-q0-prob-2}
    &=  \Big(1-\rho -   \Big(1-\rho - \frac{1}{n}\Big)\Big)\Rsmax \\
    \nonumber
    &= \frac{1}{n}\Rsmax.
\end{align}
Specifically, the equality \eqref{eq:pf-main:apply-zero-prejump-Rs} holds because $R_{\si}=0$ almost surely under $\Pnb_{\si}$ (\Cref{lem:bar}). Here, we have suppressed the superscript $\bef$ for variables inside Palm expectations. The next inequality \eqref{eq:pf-main:apply-lower-bound-and-R-bound} follows from $\Epbig{\si}{\mu R_{\sj}\givenbig Q, \Qdropj} \leq \Rsmax$ (\Cref{lem:R-bdd}) and $\indibrac{Q=0} - \indibracbig{\Qdropj=0} \geq 0$ (\Cref{lem:leave-one-out-upper-bound}). Furthermore, \eqref{eq:pf-main:relax-to-Ps} follows from the definition $\Probp{s}{Q=0} \triangleq \sumi \Probp{\si}{Q=0} / n$, while the equality \eqref{eq:pf-main:apply-q0-prob-2} is obtained by substituting the values $\Probp{s}{Q=0} = 1-\rho$ and $\sum_{i\neq j} \Probpbig{\si}{\Qdropj=0} /n = 1-\rho-1/n$ (\Cref{lem:apply-bar-lots-of}).

To bound the second term on the right-hand side of \eqref{eq:pf-main:leave-one-out-decomposition}, recall that $\indibracbig{\Qdropj=0}$ and $R_{\sj}$ are uncorrelated under the distribution $\Pnb_{\si}$ for $i\neq j$ (\Cref{lem:indep-R-Qtilde}), so we can exactly calculate this term as
\begin{align}
    \Ep{s}{\indibrac{\Qdropj=0} \mu R_{\sj}}
    \nonumber
    &=  \frac{1}{n} \sum_{i\colon i\neq j}\Ep{\si}{\indibrac{\Qdropj=0} \mu R_{\sj}} \\
    \label{eq:pf-main:apply-indep}
    &=  \frac{1}{n}\sum_{i\colon i\neq j} \Probp{\si}{\Qdropj=0} \EpBig{\pi}{\mu R_{\sj}}\\
    \label{eq:pf-main:apply-q0-prob}
    &= \Big(1-\rho-\frac{1}{n}\Big)\Ep{\pi}{\mu R_{\sj}},
\end{align}
where \eqref{eq:pf-main:apply-indep} is due to \Cref{lem:indep-R-Qtilde}, and \eqref{eq:pf-main:apply-q0-prob} is due to \Cref{lem:apply-bar-lots-of}. 

Substituting the above calculations back to \eqref{eq:pf-main:leave-one-out-decomposition}, we get 
\[
    \Ep{s}{\indibrac{Q=0}\mu R_{\sj}} \leq \frac{1}{n}\Rsmax + \Big(1-\rho-\frac{1}{n}\Big)\Ep{\pi}{\mu R_{\sj}}.
\]
Therefore, 
\begin{align*}
    \Gamma_{\sj} 
    &= \Ep{s}{\indibrac{Q=0} \mu R_{\sj}} -  (1-\rho) \Ep{\pi}{\mu R_{\sj}} 
    \leq \frac{1}{n} \Big(\Rsmax - \Ep{\pi}{\mu R_{\sj}}\Big). %
\end{align*}
Because $\Ep{\pi}{\mu R_{\sj}} = \E{(\mu S)^2}/2$, this finishes the proof of \eqref{eq:s-covariance-bound}. %

\noindent \paragraph*{Bounding $-\Gamma_a$}
Finally, we prove \eqref{eq:a-covariance-bound}.
\begin{align}
    \nonumber
     - \Gamma_{a} &=  (1-\rho) \Ep{\pi}{\Lambda R_{a}} - \Ep{s}{\indibrac{Q=0} \Lambda R_a} \\
     \nonumber
     &= (1-\rho) \Ep{\pi}{\Lambda R_{a}} - \frac{1}{n}\sumi\Ep{\si}{\indibrac{Q=0} \Lambda R_a} \\
     \nonumber
     &= (1-\rho) \Ep{\pi}{\Lambda R_{a}} - \frac{1}{n}\sumi\Ep{\si}{\Lambda R_a \given Q=0} \, \Probp{\si}{Q=0} \\
     \label{eq:pf-main:apply-Ramin}
     &\leq  (1-\rho) \Ep{\pi}{\Lambda R_{a}} -  \Ramin \Probp{s}{Q=0} \\
     \label{eq:pf-main:Gamma-bd-substitute}
     &= (1-\rho) \Big(\frac{1}{2}\E{(\Lambda A)^2} - \Ramin\Big).
\end{align}
where \eqref{eq:pf-main:apply-Ramin} is due to the fact that $\Ep{\si}{\Lambda R_a \given Q=0} \geq \Ramin$ (\Cref{lem:R-bdd}), and \eqref{eq:pf-main:Gamma-bd-substitute} follows from substituting the values $\Ep{\pi}{\Lambda R_a} = \E{(\Lambda A)^2} / 2$ and $\Probp{s}{Q=0} = 1-\rho$ (\Cref{lem:apply-bar-lots-of}). 
\qedhere
\end{proof}

\section{Generalization to heterogeneous servers}
\label{sec:hetero}
In the previous sections, we have proved an $O(1/(1-\rho))$ bound for the steady-state mean queue length of the \ggn queue with homogeneous servers.
In this section, we generalize our bound to the \ggn queue with fully heterogeneous servers, where the servers could have mutually distinct service-time distributions. 

The organization of this section roughly mirrors the high-level structure of Sections~\ref{sec:model-homo}, \ref{sec:main-result}, \ref{sec:lemmas-homo}, and \ref{sec:pf-main-upper-bound}. Specifically, we first set up the \ggn queue and modified \ggn queue with heterogeneous servers in \Cref{sec:setup-het}. Then in \Cref{sec:result-het}, we state the assumptions and the main theorem, \Cref{thm:main-upper-bound-het}. 
In \Cref{sec:supporting-lemmas-het}, we establish the technical framework and state the supporting lemmas, which then allows us to prove \Cref{thm:main-upper-bound-het} in \Cref{sec:pf-main-theorem-het}.

\subsection{\ggn queue and modified \ggn queue with heterogeneous servers}
\label{sec:setup-het}

We consider the \ggn queue with heterogeneous servers under the FCFS policy. 
This queueing system consists of a central queue and $n$ servers, where the service-time distributions could be \emph{server-dependent}. 
Jobs arrive according to a renewal process. A newly-arrived job joins the end of the queue if all servers are busy, or start being served by one of the idle servers. 
When a job is in service, it occupies the server for a random amount of time following the service-time distribution of the server; the job leaves the system after completing the service. 
When a server completes a job, it immediately starts serving a job from the head of the queue if the queue is non-empty. 
We allow an \emph{arbitrary routing policy} for assigning jobs to idle servers when a new job arrives or when multiple servers free up simultaneously.

The definitions of initial state and sample paths of the \ggn queue can be straightforwardly generalized to the heterogeneous setting, which we do not repeat here. 
We use $A$ to denote a generic interarrival time, and use $S_i$ to denote a generic service time of the $i$-th server. %
We denote the service rate of the $i$-th server as $\mu_i \triangleq 1/\E{S_i}$, and define the total service rate as $\mutotal \triangleq \sumi \mu_i$. Let the load $\rho = \Lambda / \mutotal$. 
Let $\Ramin \triangleq \inf_{t\geq 0}\E{\Lambda A - t \given \Lambda A\geq t}$ and $\Rsmaxi\triangleq \sup_{t\geq 0} \E{\mu_i S_i - t \given \mu_i S_i\geq t}$ for $i\in[n]$. 

The \ggn queue with heterogeneous servers has the same state representation as the homogeneous setting: at any time $t \geq 0$, its state $X(t)$ is given by
\begin{equation*}
    X(t)\triangleq (R_{\sone}(t), R_{\stwo}(t), \dots R_{\sn}(t), R_a(t), Q(t)), 
\end{equation*}
where $R_{\si}(t)$ is the residual service time of the $i$-th server, $R_a(t)$ is the residual arrival time, and $Q(t)$ is the queue length. 
The state space of $\{X(t)\}_{t\geq 0}$ is $\spacemain \triangleq \R^{n+1}\times \Znn$. 

As in the homogeneous setting, here we also consider the modified \ggn queue, which introduces a virtual job whenever a server frees up and the queue is empty. The modified system provides a queue-length upper bound for the original system, as shown in \Cref{lem:modified-dominates-hetero} below. We prove \Cref{lem:modified-dominates-hetero} in \smartref{app:modified-dominance}{Section}{\valAppDom}. 

\begin{restatable}{lemma}{modifieddominance}
    \label{lem:modified-dominates-hetero}
    With the same initial states, at any time $t$, the queue length of the \ggn queue with heterogeneous servers is stochastically dominated by the queue length of the corresponding modified \ggn queue.
\end{restatable}

With a slight abuse of notation, we also use $X(t)$ to denote the state of the modified \ggn system. Under the dynamics of the modified \ggn system, $R_{\si}(t)$ is the forward recurrence time of a renewal process with the inter-event distribution $S_i$, for each $i\in[n]$; $R_a(t)$ is the forward recurrence time of a renewal process with the inter-event distribution $A$; the queue length $Q(t)$ is determined by these $(n+1)$ independent renewal processes. 
The detailed update rule of $X(t)$ is identical to those summarized in \Cref{sec:modified-ggn}. %

\subsection{Assumptions and results}
\label{sec:result-het}

Our results hold under the following three assumptions, which are analogous to Assumptions~\ref{assump:bdd-exp-remain}, \ref{assump:non-lattice}, and \ref{assump:modified-ggn-stable} in the homogeneous setting. Note that in \Cref{assump:modified-ggn-stable-het}, the load of a modified \ggn queue refers to the ratio of its arrival rate and the total service rate of its servers. 

\begin{assumption}
    \label{assump:bdd-exp-remain-het}
    We assume that $\E{A^2} < \infty$ and $\Rsmaxi < \infty$ for all $i\in[n]$. 
\end{assumption}

\begin{assumption}
    \label{assump:non-lattice-het}
    For each $i\in[n]$, the service-time distribution $S_i$ is non-lattice.  
\end{assumption}

\begin{assumption}
    \label{assump:modified-ggn-stable-het}
    Each modified \ggn queue with heterogeneous servers considered in this section is positive Harris recurrent if its load is less than $1$. 
    In addition, let $\nu$ be its unique stationary distribution, then $\lim_{T\to\infty} \frac{1}{T} \int_0^T \E{Q(t)} \odv t = \Ep{\nu}{Q} < \infty$. 
\end{assumption}

\begin{theorem}
    \label{thm:main-upper-bound-het}
     Consider the \ggn queue with heterogeneous servers under the FCFS policy, satisfying Assumptions~\ref{assump:bdd-exp-remain-het}, \ref{assump:non-lattice-het}, and \ref{assump:modified-ggn-stable-het}. When the load $\rho < 1$, let $\pi$ be the unique stationary distribution of the system. Then the mean queue length under $\pi$ satisfies
    \begin{equation}
        \label{eq:main-upper-bound-het}
        \begin{aligned}
            \Ep{\pi}{Q} &\leq \frac{\rho\Var{\Lambda A} + \sumj (\mu_j / \mutotal) \Var{\mu_j S_j} + 1-\rho}{2(1-\rho)}  \\
            &\mspace{23mu} + \sumj \min\Big(1, \frac{\mu_j}{(1-\rho)\mutotal}\Big) \Big(\Rsmaxj - \frac{1}{2} \E{(\mu_j S_j)^2}\Big) \\
            &\mspace{23mu} + \frac{1}{2}\E{(\Lambda A)^2} - \Ramin.
        \end{aligned}
    \end{equation}
\end{theorem}

The bound in \eqref{eq:main-upper-bound-het} implies the following simpler bound after replacing each minimum with $\mu_j / \big((1-\rho)\mutotal\big)$:
\begin{equation}
    \label{eq:main-upper-bound-simplified-het}
    \Ep{\pi}{Q} \leq \frac{\rho\Var{\Lambda A} - \rho + 2\sumj (\mu_j / \mutotal) \Rsmaxj}{2(1-\rho)} + \frac{1}{2}\E{(\Lambda A)^2} - \Ramin. 
\end{equation}
When the arrival process is Poisson, the bound in \eqref{eq:main-upper-bound-simplified-het} further simplifies into 
\begin{equation}
    \label{eq:main-upper-bound-MGn-het}
     \Ep{\pi}{Q} \leq \frac{\sumj (\mu_j / \mutotal) \Rsmaxj}{1-\rho}. 
\end{equation}
Notably, the bounds in \eqref{eq:main-upper-bound-simplified-het} and \eqref{eq:main-upper-bound-MGn-het} only depend on the service-time distributions through the \emph{weighted average} of $\Rsmaxj$ across the servers. 
This makes these bounds relatively robust, in the sense that having a small fraction of servers with bad tail distributions does not significantly worsen the bounds. 
In particular, when $\Var{\Lambda A}$ and $\sumj (\mu_j / \mutotal) \Rsmaxj$ are bounded by constants, these bounds have the order $O\big(1/(1-\rho)\big)$ for all joint scalings of $n$ and $\rho$.

\subsection{Technical framework and supporting lemmas}
\label{sec:supporting-lemmas-het}
This subsection has a similar structure as \Cref{sec:lemmas-homo}: we define leave-one-out systems, state the Basic Adjoint Relationship (BAR), and then state four additional supporting lemmas. 
We also adopt the following no-simultaneous-event assumption to simplify the presentation. The argument for removing this assumption is identical to that provided in \Cref{remark:simultaneous}, so we do not repeat it here. 

\begin{assumption}
    In the modified \ggn queue with heterogeneous servers, no two events (i.e., an arrival and a completion, two arrivals, or two completions) happen at the same time.
\end{assumption}

We first define the leave-one-out systems in the same way as in \Cref{sec:leave-one-out}: recall that for each $j\in[n]$, the $j$-th leave-one-out system is a modified \ggn queue whose queue length process, $\Qdropj(t)$, is driven by the arrival process $R_a(t)$ and the completion processes $\{R_{\si}(t)\}_{i\in[n]\backslash\{j\}}$ of the main modified \ggn queue under study.

The following lemma states that the queue length of each leave-one-out system pathwise dominates that of the main system. We omit the proof as it is identical to the one provided for the homogeneous case in \Cref{lem:leave-one-out-upper-bound}.

\begin{lemma}\label{lem:leave-one-out-upper-bound-het}
    Consider the modified \ggn queue with heterogeneous servers.
    For each $j=1,2,\dots n$,
    the queue length of the $j$-th leave-one-out system pathwise dominates the queue length of the main system under study, i.e., 
    \[
        \Qdropj(t) \geq Q(t) \quad \forall t\geq 0.
    \]
\end{lemma}

Our analysis focuses on the main modified \ggn system coupled with that subset of leave-one-out systems that have well-defined stationary distributions. Specifically, we denote the load of the $j$-th leave-one-out system as $\rhodropj \triangleq \Lambda / (\mutotal - \mu_j)$. We let $\indexstable \triangleq \{j\in[n]\colon\rhodropj < 1\}$ be the indices of leave-one-out systems with subcritical loads and let $\Qdropstable(t) \triangleq (\Qdropj(t))_{j\in\indexstable}$ denote the vector of their queue lengths at  the time $t$.

We then define the state-augmented Markov process $\{\Xsuper(t)\}_{t\geq 0}$ as
\begin{equation}
    \label{eq:xsuper-def-het}
    \Xsuper(t) \triangleq (R_{\sone}(t), R_{\stwo}(t), \dots R_{\sn}(t), R_a(t), Q(t), (\Qdropj(t))_{j\in\indexstable}). 
\end{equation}
The state space of $\{\Xsuper(t)\}_{t\geq 0}$ is thus given by $\spacesuper \triangleq \Rnn^{n+1} \times \Znn^{1+\abs{\indexstable}}$.

With the differential operator $\opinner$ and the Palm distributions defined in the same way as in \Cref{sec:steady-state-anaysis:bar}, we can prove BAR for the heterogeneous setting: 

\begin{restatable}[Basic Adjoint Relationship for \ggn with heterogeneous servers]{lemma}{barhet}
    \label{lem:bar-het}
    Consider the modified \ggn system with heterogeneous servers. Suppose Assumptions~\ref{assump:bdd-exp-remain-het}, \ref{assump:non-lattice-het} and \ref{assump:modified-ggn-stable-het} hold, and the load $\rho < 1$. 
    For any function $f\colon \spacesuper\to \R$ that is bounded, locally Lipschitz continuous, and partially differentiable from the right with respect to its first $(n+1)$ coordinates, we have 
    \begin{equation}
        \label{eq:bar-het}
        \Ep{\sdsuper}{\opinner f(\Xsuper)} + \Lambda \Ep{a}{f(\Xsuper\aft) -f(\Xsuper\bef)} + \sumi \mu_i \Ep{\si}{f(\Xsuper\aft) - f(\Xsuper\bef)} = 0, 
    \end{equation}
    where $\Xsuper, \Xsuper\aft, \Xsuper\bef \in\spacesuper$ are random elements following the distributions specified in the expectation subscripts. 
    Moreover, the Palm distributions $\Pnb_a$ and $\Pnb_{\si}$ satisfy the following with probability $1$:
   \begin{enumerate}%
        \item[(i)] \label{item:Pa-jumps-het} If $(\Xsuper\aft, \Xsuper\bef)\sim\Pnb_a$, then $R_a\bef=0$ and the coordinates of $\Xsuper\aft$ and $\Xsuper\bef$ satisfy the relations:
        \begin{align}
            \label{eq:Pnba-char-ra-het}
            R_a\aft &= R_a\bef + A \\
            \label{eq:Pnba-char-q-het}
            Q\aft &= Q\bef + 1  \\
            \label{eq:Pnba-char-q-j-het}
            \Qdropj\aft &= \Qdropj\bef + 1 \quad \forall j\in\indexstable,
        \end{align}
        where $A$ represents a generic interarrival time and is independent of $\Xsuper\bef$; $\indexstable \triangleq \{j\in[n]\colon\rhodropj < 1\}$; other coordinates of $\Xsuper\aft$ are identical to $\Xsuper\bef$. 
        \item[(ii)] \label{item:Pi-jumps-het} For each $i\in[n]$, if $(\Xsuper\aft, \Xsuper\bef)\sim\Pnb_{\si}$, then $R_{\si}\bef=0$ and the coordinates of $\Xsuper\aft$ and $\Xsuper\bef$ satisfy the relations:
        \begin{align}
            \label{eq:Pnbsi-char-rsi-het}
            R_{\si}\aft &= R_{\si}\bef + S_i \\
            \label{eq:Pnbsi-char-q-het}
            Q\aft &= Q\bef - \indibrac{Q\bef > 0} \\
            \label{eq:Pnbsi-char-q-j-het}
            \Qdropj\aft &= \Qdropj\bef - \indibrac{\Qdropj\bef > 0} \quad \forall j\in\indexstable \backslash \{i\},
        \end{align}
        where $S_i$ represents a generic service time of the $i$-th server and is independent of $\Xsuper\bef$; other coordinates of $\Xsuper\aft$ are identical to $\Xsuper\bef$. 
    \end{enumerate}
\end{restatable}

\Cref{lem:bar-het} is proved in \smartref{app:pf-BAR}{Section}{\valAppBAR}. 
To simplify the notation, sometimes we would equivalently write the main equation of BAR \eqref{eq:bar-het} as
\begin{equation}
    \label{eq:bar-Es-form-het}
    \Ep{\pi}{\opinner f(\Xsuper)} + \Lambda \Ep{a}{f(\Xsuper\aft) - f(\Xsuper\bef)} + \mutotal \Ep{s}{f(\Xsuper\aft) - f(\Xsuper\bef)} = 0,
\end{equation}
where 
\[
    \Ep{s}{\cdot} \triangleq \sumi\frac{\mu_i}{\mutotal}\Ep{\si}{\cdot}.
\]
We omit the superscript $-$ from $\Xsuper\bef$ when it does not appear together with $\Xsuper\aft$.

Next, we state the four additional supporting lemmas that generalize Lemmas~\ref{lem:R-bdd}, \ref{lem:indep-R-Qtilde}, \ref{lem:apply-bar-lots-of}, and \ref{lem:expfinite} to the heterogeneous case. Their proofs are provided in \smartref{app:technical-lemmas}{Section}{\valAppSupp}.

\begin{restatable}[Bounded mean residual times, heterogeneous]{lemma}{Rbddhet}
    \label{lem:R-bdd-het}
    Consider the modified \ggn queue with heterogeneous servers. Suppose Assumptions~\ref{assump:bdd-exp-remain-het} and \ref{assump:modified-ggn-stable-het} hold and the load $\rho < 1$. Then for any $j\in[n]$,
        \begin{align}
            \tag{\ref{eq:R-sj-bound-stationary}}
            \mu_{j} \Ep{\nu}{\Rsj \given Q, \Qdropstable} &\leq \Rsmaxj \\ 
             \tag{\ref{eq:R-a-bound-stationary}}
            \Lambda\Ep{\nu}{R_a \given Q, \Qdropstable} &\geq \Ramin. 
        \end{align}
        where $\Ep{\nu}{\cdot}$ can denote either $\Ep{\pi}{\cdot}$ or $\Ep{\si}{\cdot}$ for $i\neq j$. 
\end{restatable}

\begin{restatable}[Uncorrelation lemma, heterogeneous]{lemma}{indephet}
\label{lem:indep-R-Qtilde-het}
Consider the modified \ggn queue with heterogeneous servers. Suppose Assumptions~\ref{assump:bdd-exp-remain-het}, \ref{assump:non-lattice-het}, and \ref{assump:modified-ggn-stable-het} hold. Then for any $i, j\in[n]$ such that $i\neq j$ and $\rhodropj < 1$,
\begin{equation}
     \Ep{\si}{\indibrac{\Qdropj=0}\Rsj} = \Probp{\si}{\Qdropj=0}\Ep{\pi}{\Rsj}.
\end{equation}
\end{restatable}

\begin{restatable}{lemma}{applybarpreliminaryhet}
    \label{lem:apply-bar-lots-of-het}
    Consider the modified \ggn queue with heterogeneous servers. Suppose Assumptions~\ref{assump:bdd-exp-remain-het} and \ref{assump:modified-ggn-stable-het} hold, and the load $\rho < 1$. For any $i,j\in[n]$ such that $i\neq j$, we have
    \begin{align}
        \label{eq:E-pi-R-s-het}
        \Ep{\sdsuper}{\Rsi} &= \frac{1}{2}\mu_i\E{S_i^2} \\
        \label{eq:E-pi-R-a-het}
        \Ep{\sdsuper}{R_a} &= \frac{1}{2}\Lambda \E{A^2} \\
        \label{eq:E-s-R-s-het}
        \Ep{\si}{\Rsj} + \Ep{\sj}{\Rsi} &= \frac{1}{2}\mu_i \E{S_i^2} + \frac{1}{2} \mu_j \E{S_j^2} \\
        \label{eq:E-sa-R-sa-het}
        \Ep{\si}{R_a} + \Ep{a}{\Rsi} &= \frac{1}{2}\mu_{i} \E{S_i^2} + \frac{1}{2}\Lambda \E{A^2} \\
        \label{eq:E-s-Q-het}
        \Probp{s}{Q=0} &= 1 - \rho 
    \end{align}
    In addition, for each $j \in [n]$, if $\rhodropj < 1$, the $j$-th leave-one-out system satisfies
    \begin{equation}
        \label{eq:E-s-Qtilde-het}
        \frac{1}{\mutotal} \sum_{i\colon i\neq j} \mu_i \Probp{\si}{\Qdropj=0} = 1 - \rho - \frac{\mu_j}{\mutotal}.
    \end{equation}
\end{restatable}

\begin{restatable}{lemma}{expfinitehet}
\label{lem:expfinite-het}
    Consider the modified \ggn system with heterogeneous servers. Suppose Assumptions~\ref{assump:bdd-exp-remain-het} and \ref{assump:modified-ggn-stable-het} hold, and the load $\rho < 1$. Then we have $\Ep{a}{Q} < \infty$ and $\Ep{\si}{Q} < \infty$ for all $i\in[n]$. 
\end{restatable}

\subsection{Proof of \Cref{thm:main-upper-bound-het}}
\label{sec:pf-main-theorem-het}
We are now ready to prove \Cref{thm:main-upper-bound-het}. The proof follows the same two-step structure used for the homogeneous case (\Cref{sec:pf-main-upper-bound}), with the changes being primarily notational. We first derive an expression of $\Ep{\pi}{Q}$ involving covariance terms (\Cref{lem:apply-bar-main-quadratic-het}), and then bound these terms (\Cref{lem:covariance-bounds-het}) to establish the theorem. 
To avoid repeating the detailed arguments from the previous section, we present only the key steps of the proof and omit explicit citations to some of the corresponding lemmas.

\begin{lemma}
\label{lem:apply-bar-main-quadratic-het}
    Consider the modified \ggn system satisfying Assumptions~\ref{assump:bdd-exp-remain-het} and \ref{assump:modified-ggn-stable-het}, with the load $\rho < 1$. We have
    \begin{equation}
        (1-\rho) \Ep{\pi}{Q}
        =  \frac{\rho\Var{\Lambda A} + \sumi (\mu_i/\mutotal) \Var{\mu_i S_i}  + 1-\rho}{2} +  \sumj \Gamma_{\sj}  - \Gamma_a \;, 
    \end{equation}
    where $\Gamma_{\sj}$ and $\Gamma_{a}$ are defined as
    \begin{align}
        \label{eq:gamma-s-def-het}
        \Gamma_{\sj} &= \Ep{s}{\indibrac{Q=0}\mu_j R_{\sj}} - (1-\rho)\Ep{\pi}{\mu_j R_{\sj}} \quad \forall j\in[n] \\
        \label{eq:gamma-a-def-het}
        \Gamma_a &= \Ep{s}{\indibrac{Q=0}\Lambda R_a} - (1-\rho)\Ep{\pi}{\Lambda R_a}. 
    \end{align}
\end{lemma}

\begin{proof}
    For any integer $\kappa > 0$, we apply \eqref{eq:bar-Es-form-het} to the test function 
    \[
        f(\Xsuper) = \Big(Q\wedge \kappa + \sumj \mu_j R_{\sj}\wedge\kappa - \Lambda R_a\wedge\kappa \Big)^2,
    \]
    rearrange the terms. After taking $\kappa \to\infty$, we get
    \begin{align}
        &2 (\mutotal - \Lambda)\Ep{\pi}{Q + \sumj \mu_j R_{\sj} - \Lambda R_a  } \\
        \label{eq:quadratic-bar-het:arrival-term}
        &= \Lambda \Ep{a}{2\Big(Q + \sumj \mu_j R_{\sj}\Big)\Big(1-\Lambda A\Big) + \Big(1-\Lambda A\Big)^2} \\
        \label{eq:quadratic-bar-het:completion-term}
        &\mspace{20mu} + \sumi \mu_i \Ep{\si}{2\Big(Q + \sumj \mu_j R_{\sj} - \Lambda R_a \Big)\Big(\mu_i S_i  - \indibrac{Q > 0} \Big) + \Big( \mu_i S_i\ -\indibrac{Q > 0}\Big)^2 }.
    \end{align}
    The interchange of the limit $\kappa \to\infty$ and the expectation is justified by an argument analogous to that in the proof of \Cref{lem:apply-bar-main-quadratic} (see \Cref{sec:pf-main:apply-bar-quadratic}). 
    The key non-trivial condition for this step is the finiteness of the moments $\Ep{\pi}{Q}$, $\Ep{a}{Q}$, and $\Ep{\si}{Q}$, which is guaranteed by \Cref{assump:modified-ggn-stable-het} and \Cref{lem:expfinite-het}.  
    
    We now simplify the expressions in \eqref{eq:quadratic-bar-het:arrival-term} and \eqref{eq:quadratic-bar-het:completion-term}. First, due to the independence between $A$ and $(Q, (R_{\sj})_{j\in[n]})$, the terms in \eqref{eq:quadratic-bar-het:arrival-term} simplifies into $\Lambda \Var{\Lambda A}$. 
    Second, due to the independence between $S_i$ and $(Q, R_a, (R_{\sj})_{j\in[n]})$, the summation in \eqref{eq:quadratic-bar-het:completion-term} simplifies into:
    \[
         2\mu_i \Ep{\si}{\Big(\sumj \mu_j R_{\sj} - \Lambda R_a\Big)\indibrac{Q=0}} + \mu_i \Var{\mu_i S_i} + \mu_i (1-\rho).  
    \]
    Substituting these expressions back into  \eqref{eq:quadratic-bar-het:arrival-term}  and \eqref{eq:quadratic-bar-het:completion-term} yields the following:
    \begin{align}
        \nonumber
        &2 (\mutotal - \Lambda)\Ep{\pi}{Q + \sumj \mu_j R_{\sj} - \Lambda R_a} \\
        \nonumber
        &\qquad =  2\mutotal \Ep{s}{\Big(\sumj \mu_j R_{\sj} - \Lambda R_a\Big)\indibrac{Q=0}} +  \Lambda \Var{\Lambda A} + \sumi \mu_i \Var{\mu_i S_i} + \mutotal (1-\rho). 
    \end{align}
    Dividing both sides by $2\mutotal$ and rearranging the terms complete the proof of \Cref{lem:apply-bar-main-quadratic-het}. \qedhere
\end{proof}

\begin{lemma}
\label{lem:covariance-bounds-het}
    Consider the modified \ggn system satisfying Assumptions~\ref{assump:bdd-exp-remain-het}, \ref{assump:non-lattice-het} and \ref{assump:modified-ggn-stable-het}, with the load $\rho < 1$. Let $\Gamma_{\sj}$ and $\Gamma_{a}$ be the terms defined in \eqref{eq:gamma-s-def-het} and \eqref{eq:gamma-a-def-het}, respectively. Then
    \begin{align}
        \label{eq:s-covariance-bound-het}
       \Gamma_{\sj} &\leq \min\Big(1-\rho, \frac{\mu_j}{\mutotal}\Big) \Big(\Rsmaxj - \frac{1}{2} \E{(\mu_j S_j)^2}\Big) \\
       \label{eq:a-covariance-bound-het}
       - \Gamma_{a} &\leq (1-\rho) \bigg(\frac{1}{2}\E{(\Lambda A)^2} - \Ramin\bigg) 
    \end{align}
\end{lemma}

\begin{proof}
    We first bound $\Gamma_{\sj}$ based on two cases:  $\rhodropj \geq 1$ or $\rhodropj < 1$; then we bound $-\Gamma_{a}$. 

    \paragraph*{Bounding $\Gamma_{\sj}$ when $\rhodropj \geq 1$} In this case, $\Lambda / ( \mutotal - \mu_j) \triangleq \rhodropj \geq 1$, so we have $\Lambda  \geq \mutotal - \mu_j$. Consequently, $1 - \rho = 1 - \Lambda / \mutotal \leq \mu_j / \mutotal$. It is thus sufficient to show that
    \[
        \Gamma_{\sj} \leq  (1-\rho)\Big(\Rsmaxj - \frac{1}{2} \E{(\mu_j S_j)^2}\Big).
    \]
    We bound $\Gamma_{\sj}$ as follows: 
    \begin{align}
    \nonumber
    \Gamma_{\sj} 
    &= \Ep{s}{\indibrac{Q=0} \mu_j R_{\sj}} -  (1-\rho) \Ep{\pi}{\mu_j R_{\sj}} \\
    \nonumber
    &= \frac{1}{\mutotal} \sumi \mu_i \Ep{\si}{\indibrac{Q=0} \mu_j R_{\sj}} -  (1-\rho) \Ep{\pi}{\mu_j R_{\sj}} \\
    \label{eq:pf-main:case-1:apply-R-sj-zero}
     &= \frac{1}{\mutotal}  \sum_{i\colon i\neq j}  \mu_i \Ep{\si}{\indibrac{Q=0} \mu_j R_{\sj}} -  (1-\rho) \Ep{\pi}{\mu_j R_{\sj}} \\
     \label{eq:pf-main:case-1:apply-R-bdd-het}
    &\leq \frac{1}{\mutotal} \sum_{i\colon i\neq j} \mu_i \Probp{\si}{Q=0} \Rsmaxj -  (1-\rho) \Ep{\pi}{\mu_j R_{\sj}} \\
    \nonumber
    &\leq \Probp{s}{Q=0}\Rsmaxj -  (1-\rho) \Ep{\pi}{\mu_j R_{\sj}} \\
    \label{eq:pf-main:case-1:apply-q0-prob-het}
    &= (1-\rho)\Rsmaxj - (1-\rho) \Ep{\pi}{\mu_j R_{\sj}}  \\
    \label{eq:covariance-bounds-heavy-traffic:final-het}
    &= (1-\rho)\Big(\Rsmaxj - \frac{1}{2} \E{(\mu_j S_j)^2}\Big),
\end{align}
where \eqref{eq:pf-main:case-1:apply-R-sj-zero} is because $R_{\si}=0$ almost surely under the distribution $\Pnb_{\si}$, as proved in \Cref{lem:bar-het} (we have skipped all superscripts $\bef$ for variables inside the Palm expectations in the above derivations). 
Furthermore, the inequality \eqref{eq:pf-main:case-1:apply-R-bdd-het} uses the fact that $\Ep{\si}{\mu_j R_{\sj}\given Q} \leq \Rsmaxj$ for $i\neq j$ (\Cref{lem:R-bdd-het}). 
Finally, the equalities \eqref{eq:pf-main:case-1:apply-q0-prob-het} and \eqref{eq:covariance-bounds-heavy-traffic:final-het} follow from substituting the values $\Probp{s}{Q=0}=1-\rho$ and $\Ep{\pi}{\mu_j R_{\sj}} = \E{(\mu_j S_j)^2} / 2$.
This implies \eqref{eq:s-covariance-bound-het} when $\rhodropj \geq 1$. 

\paragraph*{Bounding $\Gamma_{\sj}$ when $\rhodropj < 1$} 
In this case, $1 - \rho > \mu_j / \mutotal$, so it suffices to show that 
\[
     \Gamma_{\sj} \leq \frac{\mu_j}{\mutotal} \Big(\Rsmaxj - \frac{1}{2} \E{(\mu_j S_j)^2}\Big).
\]
Because $\rhodropj < 1$, by \Cref{assump:modified-ggn-stable-het}, the $j$-th leave-one-out system has a well-defined stationary distribution. We can thus perform the following decomposition: 
\begin{equation}
    \label{eq:pf-main:leave-one-out-decomposition-het}
    \Ep{s}{\indibrac{Q=0}\mu_j R_{\sj}} 
     = \Ep{s}{\Big(\indibrac{Q=0} -  \indibrac{\Qdropj=0}\Big)\mu_j R_{\sj}} + \Ep{s}{\indibrac{\Qdropj=0} \mu_j R_{\sj}}. 
\end{equation}
We then bound the two terms on the right-hand side of \eqref{eq:pf-main:leave-one-out-decomposition-het} separately. 

We bound the first term on the right-hand side of \eqref{eq:pf-main:leave-one-out-decomposition-het} by applying the conditional expectation bound of $R_{\sj}$ in \Cref{lem:R-bdd-het}, and note that $\indibrac{Q=0} -  \indibracbig{\Qdropj=0}\geq 0$ (\Cref{lem:leave-one-out-upper-bound-het}): 
\begin{align}
    \nonumber
    &\Ep{s}{\Big(\indibrac{Q=0} -  \indibrac{\Qdropj=0}\Big)\mu_j R_{\sj}} \\
    \nonumber
    &\qquad = \sumi \frac{\mu_i}{\mutotal} \Ep{\si}{\Big(\indibrac{Q=0} -  \indibrac{\Qdropj=0}\Big)\mu_j R_{\sj}} \\
    \label{eq:pf-main:apply-lower-bound-and-R-bound-het}
    &\qquad \leq  \sum_{i\colon i\neq j}\frac{\mu_i}{\mutotal}\Ep{\si}{\Big(\indibrac{Q=0} -  \indibrac{\Qdropj=0}\Big)}\Rsmaxj  \\
    \nonumber
    &\qquad \leq   \Big(\Probp{s}{Q=0} -  \frac{1}{\mutotal}\sum_{i\colon i\neq j} \mu_i \Probp{\si}{\Qdropj=0}\Big)\Rsmaxj \\
    \nonumber %
    &\qquad =  \Big(1-\rho -   \Big(1-\rho - \frac{\mu_j}{\mutotal}\Big)\Big)\Rsmaxj \\
    \nonumber
    &\qquad = \frac{\mu_j}{\mutotal}\Rsmaxj. 
\end{align}
where the inequality \eqref{eq:pf-main:apply-lower-bound-and-R-bound-het} follows from three facts: $R_{\si}=0$ almost surely under $\Pnb_{\si}$ (\Cref{lem:bar-het}), $\Ep{\si}{\mu_j R_{\sj}\given Q, \Qdropstable} \leq \Rsmaxj$ for $i\neq j$ (\Cref{lem:R-bdd-het}), and $\indibrac{Q=0} \geq \indibrac{\Qdropj=0}$.

For the second term on the right-hand side of \eqref{eq:pf-main:leave-one-out-decomposition-het}, recall that $\indibracbig{\Qdropj=0}$ and $R_{\sj}$ are uncorrelated under the distribution $\Pnb_{\si}$ for $i\neq j$ (\Cref{lem:indep-R-Qtilde-het}). Therefore,
\begin{align}
    \Ep{s}{\indibrac{\Qdropj=0} \mu_j R_{\sj}}
    \nonumber
    &=  \sum_{i\colon i\neq j} \frac{\mu_i}{\mutotal} \Ep{\si}{\indibrac{\Qdropj=0} \mu_j R_{\sj}} \\
    \nonumber %
    &=  \sum_{i\colon i\neq j} \frac{\mu_i}{\mutotal} \Probp{\si}{\Qdropj=0} \EpBig{\pi}{\mu_j R_{\sj}}\\
    \nonumber %
    &= \Big(1-\rho-\frac{\mu_j}{\mutotal}\Big)\Ep{\pi}{\mu_j R_{\sj}}. 
\end{align}

Substituting the above calculation into \eqref{eq:pf-main:leave-one-out-decomposition-het}, we get 
\[
    \Ep{s}{\indibrac{Q=0}\mu_j R_{\sj}} \leq \frac{\mu_j}{\mutotal}\Rsmaxj + \Big(1-\rho-\frac{\mu_j}{\mutotal}\Big)\Ep{\pi}{\mu_j R_{\sj}}.
\]
Therefore, 
\begin{align*}
    \Gamma_{\sj} 
    &= \Ep{s}{\indibrac{Q=0} \mu_j R_{\sj}} -  (1-\rho) \Ep{\pi}{\mu_j R_{\sj}} 
    \leq \frac{\mu_j}{\mutotal} \Big(\Rsmaxj - \Ep{\pi}{\mu_j R_{\sj}}\Big). 
\end{align*}
Because $\Ep{\pi}{\mu_j R_{\sj}} = \E{(\mu_j S_j)^2}/2$, this finishes the proof of \eqref{eq:s-covariance-bound-het} when $\rhodropj < 1$.

\paragraph*{Bounding $-\Gamma_a$}
\begin{align}
    \nonumber
     - \Gamma_{a} &=  (1-\rho) \Ep{\pi}{\Lambda R_{a}} - \Ep{s}{\indibrac{Q=0} \Lambda R_a} \\
     \nonumber
     &= (1-\rho) \Ep{\pi}{\Lambda R_{a}} - \sumi \frac{\mu_i}{\mutotal}\Ep{\si}{\indibrac{Q=0} \Lambda R_a} \\
     \nonumber
     &\leq  (1-\rho) \Ep{\pi}{\Lambda R_{a}} - (1-\rho) \Ramin \\
     \nonumber
     &= (1-\rho) \Big(\frac{1}{2}\E{(\Lambda A)^2} - \Ramin\Big).
\end{align}
where the inequality is due to the conditional expectation lower bound of $R_a$ in \Cref{lem:R-bdd-het} and the fact that $\Probp{s}{Q=0}=1-\rho$.\qedhere %
\end{proof}

\section{Conclusion}
\label{sec:conclusion}
In this paper, we prove a new universal bound on the steady-state mean queue length for the \ggn queue. Our bound achieves the order $O\big(1/(1-\rho)\big)$ and, while requiring slightly stronger assumptions than those in \citep{LiGol_25}, provides significantly smaller constants. 
Our proof synthesizes three key elements: a modified \ggn queue, the Basic Adjoint Relationship (BAR), and our new leave-one-out technique.
We also extend our result to the \ggn queue with fully heterogeneous servers.

A natural direction for future work is to investigate the tightness of our bounds. While it may be unrealistic to expect any simple closed-form bound to be arbitrarily tight for the \ggn queue, there is potential for improvement. 
An ideal bound might resemble Kingman's bound, depending only on the low-order moments of the service-time distribution. Such a bound would yield smaller constants for heavy-tailed systems. 
Furthermore, in sub-Halfin-Whitt regimes, research in special cases suggests that better orders than $O\big(1/(1-\rho)\big)$ may be possible \citep[e.g.,][]{PraDanMag_24_exp_tail,BraDai_16_Erlang_C}. 

Another promising direction is to apply our proof techniques to other large-scale queueing systems, such as load-balancing models with general service-time distributions. 

Finally, it would be interesting to explore the theoretical connections between the leave-one-out technique and other existing techniques for proving asymptotic independence \citep[e.g.][]{WanHarJia_19}. Clarifying their similarities and respective scopes would enrich the analytical toolbox for queueing systems.

\bibliographystyle{informs2014}
\IfFileExists{bib/refs.bib}{
    \bibliography{bib/refs}       %
}{
    \bibliography{../bib/refs} %
}

\newpage
    \ECSwitch
    \ECHead{E-Companion}

\input{appendix}

\end{document}

%% file: abstract.tex
Bounding the steady-state queue length of a multiserver queue is a central challenge in queueing theory. Even for the classical \ggn queue with homogeneous servers, obtaining a simple, accurate bound that holds across all parameters is highly non-trivial. A recent breakthrough by \citet{LiGol_25} establishes the first \emph{universal} bound of order $O(1/(1-\rho))$, holding for every load $\rho<1$ and server count $n$---an order known to be tight in many regimes, including classical heavy-traffic, Halfin-Whitt, and Non-Degenerate Slowdown. However, their bounds carry astronomically large constants and rely on an intricate proof; they conjecture that a far simpler bound holds.

We introduce a \emph{leave-one-out coupling} technique that yields a new universal $O(1/(1-\rho))$ bound for the \ggn queue, with a simple and transparent proof. Moreover, for light-tailed service times, the leading constant in our bound is orders of magnitude smaller than that in prior work. For instance, we bound the \mgn queue's mean queue length by simply $1/(1-\rho)$ for New-Better-than-Used-in-Expectation service times, with similarly clean bounds for gamma, phase-type, and bounded service times.

Finally, our techniques extend to \ggn queues with \emph{fully heterogeneous} service-time distributions, a setting not addressed by prior universal bounds.

%% file: appendix.tex
\section{Bounded mean residual time implies exponential tail}
\label{app:Rsmax-finite-implies-exponential}
In this section, we show that if a distribution has a bounded mean residual time given any age, it must have an exponentially decaying tail, so all of its moments must be finite. This fact has also been proved in \citep{Dow_91_bdd_sync_overhead}. Given this fact, \Cref{assump:bdd-exp-remain} and \Cref{assump:bdd-exp-remain-het} automatically imply that the service-time distributions have exponentially decaying tails and finite moments.

\begin{restatable}{lemma}{rmaximpliestail}
    \label{eq:rmax-implies-tail}
    For any non-negative random variable $V$, suppose there exists a constant $C > 0$ such that
    \begin{equation}
        \label{eq:bdd-mst:condition}
        \E{V-t \given V\geq t} \leq C \quad \forall t\geq 0,
    \end{equation}
    Then we have
    \begin{equation}
        \label{eq:bdd-mst:exp-tail-bound}
        \Prob{V \geq t} \leq C\exp(-(t-1)/C) \quad \forall t\geq 1,
    \end{equation}
    and consequently, $\E{V^m}< \infty$ for all $m\in\Znn$. 
\end{restatable}

\begin{proof}
    Let $\overline{F}(t)$ be the complementary CDF of $V$. The condition \eqref{eq:bdd-mst:condition} can be rewritten as 
    \begin{equation}
        \label{eq:pf-tail-bound:integral-eq}
        \int_t^\infty \overline{F}(u) \odv u \leq C \overline{F}(t) \quad \forall t\geq 0.
    \end{equation}
    We will mimic the proof of Gronwall's inequality to derive an exponential tail bound for $\overline{F}(t)$. 
    We define the function $G(t)$ as
    \[  
        G(t) = \exp(t/C) \int_t^\infty \overline{F}(u) \odv u.
    \]
    Because $\overline{F}(t)$ is bounded on $[0,\infty)$, it is not hard to see that $\int_t^\infty \overline{F}(u) \odv u$ is locally Lipschitz continuous, and thus almost everywhere differentiable. For any $t$ at which $\int_t^\infty \overline{F}(u) \odv u$ is differentiable, $G(t)$ is also differentiable, and $G'(t)$ is given by
    \[
        G'(t) = \exp(t/C) \left(\frac{1}{C}\int_t^\infty \overline{F}(u) \odv u - \overline{F}(t) \right),
    \]
    which is non-positive by \eqref{eq:pf-tail-bound:integral-eq}. 
    Moreover,  $G(t)$ is also locally Lipschitz continuous. Therefore, on any compact interval, $G(t)$ is absolutely continuous and satisfies the fundamental theorem of calculus. In particular,
    \begin{equation*}
        G(t) = G(0) + \int_0^t G'(u) \odv u 
        \leq G(0) 
        = \int_0^\infty \overline{F}(u) \odv u   \leq C \overline{F}(0)
        = C,
    \end{equation*}
    where the second inequality is due to \eqref{eq:pf-tail-bound:integral-eq}.
    By the definition of $G(t)$, we have
    \begin{equation*}
        \int_t^\infty \overline{F}(u) \odv u \leq C \exp(-t/C).
    \end{equation*}
    Because $\overline{F}(t)$ is non-increasing and non-negative, we can bound $\overline{F}(t+1)$ as
    \[
        \overline{F}(t+1) \leq \int_t^{t+1} \overline{F}(u) \odv u \leq C \exp(-t/C).
    \]
    This proves \eqref{eq:bdd-mst:exp-tail-bound}. The finiteness of $\E{V^m}$ follows straightforwardly from the exponential tail bound. 
   \qedhere  
\end{proof}

\section{Stability of modified \ggn}
\label{app:stability}
In this section, we discuss the conditions under which the modified \ggn queue is positive Harris recurrent and has a finite mean queue length. 
As mentioned in \Cref{remark:stability}, in addition to $\rho < 1$, some regularity conditions are required on the distributions of the interarrival time and service times. 
Our results have been proved by directly assuming the positive Harris recurrence and the finiteness of mean queue length (Assumptions~\ref{assump:modified-ggn-stable} and \ref{assump:bdd-exp-remain-het}). 
In this section, we will provide a more detailed discussion on these conditions, and sketch why they are likely to imply Assumptions~\ref{assump:modified-ggn-stable} and \ref{assump:bdd-exp-remain-het}, without giving formal proofs. 
Throughout this section, we focus on modified \ggn queues with heterogeneous servers and simply refer to them as ``modified \ggn queue''.

We consider adapting the results of  \citet{DaiMey_95_moment_stability}, which considers queueing networks consisting of single-server stations. 
In particular, their Theorem 6.2 shows that if such a queueing network is fluid stable and satisfies certain conditions on the interarrival-time and service-time distributions \citep[][Conditions~A1)--A3)]{DaiMey_95_moment_stability}, it is positive Harris recurrent, with certain finite steady-state moments.
We conjecture that a similar result can be shown for the modified \ggn queue. %
Specifically, in addition to \Cref{assump:bdd-exp-remain-het}, we further assume that the interarrival-time distribution $A$ is \emph{spread out and unbounded}, and the service-time distribution $S_i$ is \emph{spread out} for each server $i\in[n]$ (a distribution $\nu$ is spread out if there exists $n$ such that the distribution of the sum of $n$ independent variables following $\nu$ dominates a non-negative absolutely-continuous measure; see \citealp[][Section VII.1]{Asm_03}).
These assumptions are analogous to Conditions A1)--A3) of \citep{DaiMey_95_moment_stability}. We conjecture that, under these assumptions, a modified \ggn queue is positive Harris recurrent with finite mean queue length whenever $\rho < 1$.

Instead of providing a rigorous full proof, 
we highlight an intermediate step to provide intuition on why our conjectured result is likely to be true. 
In the proof of \citep[Proposition 5.1]{DaiMey_95_moment_stability}, an important fact is that the queue lengths of the queueing network converge to zero in probability under fluid scaling, i.e.,  
\begin{equation}
    \label{eq:daimeyn-q-limit}
    \lim_{\abs{x}\to\infty} \frac{1}{\abs{x}} \abs{Q^x(t_0\abs{x})} = 0 \quad \text{in probability},
\end{equation}
where $x$ is a state of the queueing network, $\abs{x}$ represents the sum of all queue lengths and residual times in state $x$, and $Q^x(t)$ represents the queue length process with the initial state $x$. This fact follows from the assumed stability of the queueing network's fluid model, utilizing the results in \citep[][Section IV]{Dai_95_stability}.

To show \eqref{eq:daimeyn-q-limit} for the modified \ggn queue, we invoke the following formula of queue length: 
\begin{equation}
    \label{eq:Q-explicit-formula}
    Q(t) = \max\left\{Q(0) + N_a(t) - \sumi N_i(t), \; \sup_{0\leq t'\leq t} \Big(N_a(t',t) - \sumi N_i(t',t) \Big) \right\},
\end{equation}
where $N_a(t)$ denotes the number of real-job arrivals by time $t$, $N_i(t)$ denotes the number of completions at server $i$ by time $t$, and $N_a(t', t) \triangleq N_a(t)-N_a(t')$, $N_i(t',t) \triangleq N_i(t)-N_i(t')$.
The special case of \eqref{eq:Q-explicit-formula} is proved in \citep[][Corollary~1]{GamGol_13_GGn_HW}, under the additional assumptions that $Q(0) = 0$, the servers are homogeneous and no events occur simultaneously; dropping these assumptions is straightforward from their proof. 

We intuitively argue that \eqref{eq:Q-explicit-formula} implies \eqref{eq:daimeyn-q-limit}: 
Let $x = (r_1, \dots, r_n, r_a, q)$ be the initial state of the modified \ggn queue. Then \eqref{eq:Q-explicit-formula} implies that
\begin{equation}
\label{eq:pf-Q-fluid-scaling}
\begin{aligned}
    &\frac{1}{\abs{x}} Q^x(t_0\abs{x}) \\
    &=  \max\left\{ \frac{1}{\abs{x}} \Big(q + N_a(t_0\abs{x}) - \sumi N_i(t_0\abs{x}) \Big), \, \frac{1}{\abs{x}}\sup_{0\leq t'\leq t_0\abs{x}} \Big(N_a(t',t_0\abs{x}) - \sumi N_i(t', t_0\abs{x}) \Big) \right\}.
\end{aligned}
\end{equation}
Due to the properties of renewal processes, as $\abs{x}$ gets large, $N_a(t_0\abs{x}) - \sumi N_i(t_0\abs{x})$ is concentrated around its mean
$\Lambda t_0 \abs{x}- \mutotal t_0\abs{x}$, if the residual times are initialized at equilibrium. For other initial residual times, with high probability, $N_a(t_0\abs{x}) - \sumi N_i(t_0\abs{x})$ should not be significantly greater than 
\[
\Lambda t_0 \abs{x} + 1 - \mutotal \Big(t_0\abs{x} - \sumi r_i\Big) + o(\abs{x}) = 
-(\mutotal-\Lambda) t_0\abs{x} + 1 + \sumi \mu_i r_i + o(\abs{x}). 
\]
Therefore, intuitively, 
\begin{align*}
    \frac{1}{\abs{x}} \Big(q + N_a(t_0\abs{x}) - \sumi N_i(t_0\abs{x}) \Big) 
    &\leq  \frac{q -(\mutotal - \Lambda)t_0 \abs{x} + 1 + \sumi \mu_i r_i  + o(\abs{x})}{\abs{x}}  \\
    &\leq \max_i \mu_i - (\mutotal - \Lambda)t_0 + o(1) \quad \text{with high probability,}
\end{align*}
So, taking $t_0$ to be a sufficiently large constant, the first term inside the maximum of \eqref{eq:pf-Q-fluid-scaling} is negative with high probability. 
Moreover, $\sup_{0\leq t'\leq t_0\abs{x}} \big(N_a(t',t_0\abs{x}) - \sumi N_i(t', t_0\abs{x})\big)$ should also scale sublinearly in $\abs{x}$, so the second term inside the maximum of \eqref{eq:pf-Q-fluid-scaling} converges to zero in probability as $\abs{x}\to\infty$. We have thus informally justified \eqref{eq:daimeyn-q-limit} for the modified \ggn queue. 

Intuitively, \eqref{eq:daimeyn-q-limit} implies that no matter how large the initial state $x$ is, the queue length becomes significantly shorter after a time proportional to $\abs{x}$. 
After some technical work that mimics those in Proposition 5.1, 5.3, and 5.4 of \citep{DaiMey_95_moment_stability}, one can show that the state of the modified \ggn queue would return to a compact subset of the state space reasonably fast and thus satisfy a certain drift condition (Proposition 6.1 of \citep{DaiMey_95_moment_stability}). 
Furthermore, the spread out and unbounded conditions of $A$ and $S_i$'s can be used to show that compact subsets under the system's dynamics are petite. 
Theorem~5.3 of \citep{MeyTwe_93_stability} then implies that the modified \ggn queue is positive Harris recurrent, with a finite mean queue length.

\section{Proof of modified \ggn queue's dominance}
\label{app:modified-dominance}

In this section, we prove that the queue length distribution of the modified \ggn queue stochastically dominates that of the \ggn queue. We focus on the general case of heterogeneous servers. 
The proof for homogeneous \ggn can be also found in \citep[][Proposition~1]{GamGol_13_GGn_HW}; our proof follows essentially the same logic, but we include it for completeness.

\modifieddominance*

\begin{proof}
We prove the lemma by constructing a coupling for the sample paths of the \ggn queue and the modified \ggn queue. First, by the assumption of the lemma, two systems have the same initial queue length $Q(0)$, the same set of initially busy servers that are serving real jobs $\mathcal{I}(0)$, and the same initial residual service times for the real jobs in service $(R_{\si}(0))_{i\in \mathcal{I}(0)}$. 
For each $k\in\Zp$, consider the $k$-th real job that begins service after time $0$. We let this job arrive at time $\taua^k$ in both systems. 
This job may begin service at different times in the original system and in the modified system, denoted as $\tau_s^k$ and $\hat{\tau}_s^k$, respectively. 
However, when this job begins service, we route it to the same server and let it have the same service time $S^k$ in both systems. 
Note that here, since we allow the service time distributions to be heterogeneous across servers, the sequence $\{S^k\}_{k=1}^\infty$ is not necessarily i.i.d., but this does not complicate our proof.

Let $Q(t)$ and $\hat{Q}(t)$ denote the queue lengths in the original and modified systems, respectively. They can be represented as 
\begin{align}
    \label{eq:pf-dominance:Q}
    Q(t) &= \abs{\{k\in\Zp \colon \taua^k \leq t, \tau_s^k > t\}} \\
    \label{eq:pf-dominance:Qhat}
    \hat{Q}(t) &= \abs{\{k\in\Zp \colon \taua^k \leq t, \hat{\tau}_s^k > t\}}.
\end{align}
Our goal is to show that $Q(t) \leq \hat{Q}(t)$ for all $t\geq 0$. By \eqref{eq:pf-dominance:Q} and \eqref{eq:pf-dominance:Qhat}, it suffices to show that $\tau_s^k \leq \hat{\tau}_s^k$ for all $k\in\Zp$. 

We define some additional notation. 
For $k\in\Zp$, we let $Z^k(t)$ be the number of jobs that begin service no later than the $k$-th job, but have not left the system by time $t$ in the original system. We define $\hat{Z}^k(t)$ similarly for the real jobs in the modified system (excluding the virtual jobs). 
We let $Z^0(t)$ and $\hat{Z}^0(t)$ be the number of real jobs initially in the two systems and have not left by time $t$. 
Formally, for $k\in\Znn$,
\begin{align*}
    Z^k(t) &\triangleq \abs{\{k'\in\Zp\colon k'\leq k, \tau_s^{k'} + S^{k'} > t\}} + \abs{\{i\in\mathcal{I}\colon R_{\si}(0) > t\}} \\
    \hat{Z}^k(t) &\triangleq \abs{\{k'\in\Zp\colon k'\leq k, \hat{\tau}_s^{k'} + S^{k'} > t\}} + \abs{\{i\in\mathcal{I}\colon R_{\si}(0) > t\}}.
\end{align*}
We let $\hat{V}(t)$ be the number of virtual jobs in service at time $t$. %

Now we prove by induction that $\tau_s^k \leq \hat{\tau}_s^k$ for all $k\in\Zp$, i.e., each job in the original \ggn queue begins service no later than in the modified \ggn queue.
\begin{itemize}
    \item \textbf{Base step}: $k=1$. Because job $1$ is the first job that begins service after time $0$, we have %
    \begin{align*}
        \tau_s^1 &= \inf\{t > \taua^1 \colon Z^0(t) \leq n-1\}  \\
        \hat{\tau}_s^1 &= \inf\{t > \taua^1 \colon \hat{Z}^0(t) + \hat{V}(t) \leq n-1\}.
    \end{align*}
   By definitions, $Z^0(t) = \hat{Z}^0(t)$ for all $t\geq 0$, and $\hat{V}(t) \geq 0$, so we have $\tau_s^1 \leq \hat{\tau}_s^1$. 
   \item \textbf{Induction step}: Suppose we have proved $\tau_s^{k'} \leq \hat{\tau}_s^{k'}$ for $1\leq k'\leq k$, and we want to show that $\tau_s^{k+1} \leq \hat{\tau}_s^{k+1}$. Observe that %
   \begin{align*}
        \tau_s^{k+1} &= \inf\{t > \taua^{k+1} \colon Z^k(t) \leq n-1\}  \\
        \hat{\tau}_s^{k+1} &= \inf\{t > \taua^{k+1} \colon \hat{Z}^k(t) + \hat{V}(t) \leq n-1\}.
    \end{align*}
    Because $\tau_s^{k'} + S^{k'} \leq \hat{\tau}_s^{k'} + S^{k'}$ for all $1\leq k'\leq k$, we have $Z^k(t) \leq \hat{Z}^k(t)$ for all $t\geq 0$. Moreover, $\hat{V}(t) \geq 0$, so we get $\tau_s^{k+1} \leq \hat{\tau}_s^{k+1}$. 
\end{itemize}
By induction, we have shown that $\tau_s^k \leq \hat{\tau}_s^k$ for all $k\in\Zp$. Combining the induction results and the formula of $Q(t)$ and $\hat{Q}(t)$ in \eqref{eq:pf-dominance:Q} and \eqref{eq:pf-dominance:Qhat}, we have finished the proof. \qedhere 
\end{proof}

\section{Proofs of supporting lemmas}
\label{app:technical-lemmas}
We prove the supporting lemmas for homogeneous \ggn (\Cref{sec:lemmas-homo}) and heterogeneous \ggn (\Cref{sec:hetero}) together. 
This section is organized as follows. We first introduce some additional notation and definitions in \Cref{app:notation-support-lemma}. Then we prove BAR (Lemmas~\ref{lem:bar} and \ref{lem:bar-het}) in \Cref{app:pf-BAR}. Next, in \Cref{app:pf-Rbdd}, we prove the lemmas of bounded mean residual times (Lemmas~\ref{lem:R-bdd} and \ref{lem:R-bdd-het}). In \Cref{app:pf-uncorrelation}, we prove the uncorrelation lemmas (Lemmas~\ref{lem:indep-R-Qtilde} and \ref{lem:indep-R-Qtilde-het}). 
Finally, in \Cref{app:pf-equalities-and-finiteness}, we prove the preliminary steady-state equalities (Lemmas~\ref{lem:apply-bar-lots-of} and \ref{lem:apply-bar-lots-of-het}) and finiteness results (Lemmas~\ref{lem:expfinite} and \ref{lem:expfinite-het}). 

Throughout \Cref{app:technical-lemmas}, we focus on the modified \ggn, whose service-time distributions are allowed to be heterogeneous across servers. %
We focus on the case where $\rho < 1$, and use the convention that the system is in steady state at time $0$, consistent with \Cref{sec:steady-state-anaysis:bar}. 
Without loss of generality, we assume that no events (arrivals or completions) happen at the same time point (\Cref{assump:simultaneous}); see \Cref{remark:simultaneous} for the necessary changes to proof without this assumption.

\subsection{Additional notation and definitions}
\label{app:notation-support-lemma}
In this subsection, we introduce some additional notation and definitions regarding the modified \ggn queue and the leave-one-out systems. Since we do not consider the original \ggn queue throughout this section, we will reuse some notation in \Cref{sec:model-homo} and \Cref{app:modified-dominance} that was reserved for the original \ggn queue.

Recall that the modified \ggn queue is driven by a set of independent renewal processes corresponding to the arrivals and completions. In the main body, we describe these renewal processes mainly using the residual time processes (i.e., forward recurrence times), $R_a(t)$ and $\{R_{\si}(t)\}_{i\in[n]}$; here, we define the corresponding ages processes (i.e., backward recurrence times): at each time $t$, 
\begin{itemize}
    \item Let $\agea(t)$ be the time elapsed since the last arrival.  
    \item Let $\agej(t)$ be the time elapsed since the last completion of the $j$-th server, for each $j\in[n]$.  
\end{itemize}

Next, we define the alternative state representation based on the ages: we let $\Ysuper(t) = (Y(t), \Qdropstable(t))$, where $Y(t) \triangleq (\age_1(t), \dots, \age_n(t), \agea(t), Q(t))$, and recall that the vector $\Qdropstable$ represents the queue lengths of the set of leave-one-out systems whose loads are below $1$. 

We let $\Ysuper(0)$ follow the stationary distribution. 
Each sample path of $\{\Ysuper(t)\}_{t\geq 0}$ is determined by the initial state $\Ysuper(0)$, and the sequences $\{A^{k}\}_{k\in\Znn}$ and $\{S_i^{k}\}_{k\in\Znn}$ for $i\in[n]$, defined as follows:
\begin{itemize}
    \item $A^0 - \agea(0)$ is the time of the first arrival, and $A^k$ is the interarrival time between the $k$-th and $(k+1)$-th arrival events for $k\in\Zp$. Then $A^k\sim\Prob{A\in\,\cdot\,}$ for $k\in\Zp$, and $A^0\sim\Prob{A\in\,\cdot\,\given A\geq \agea(0)}$, where $A$ denotes a generic interarrival time independent of anything else. 
    \item $S_i^0-\agei(0)$ is the completion time of the initial job on the $i$-th server, and $S_i^k$ is the time between the $k$-th and $(k+1)$-th completion events on the $i$-th server, for each $i\in[n]$ and $k\in\Zp$. Then $S_i^k\sim\Prob{S_i\in\,\cdot\,}$ for $k\in\Zp$, and $S_i^0\sim\Prob{S_i\in\,\cdot\,\given S_i\geq \agei(0)}$, where $S_i$ denotes a generic service time of the $i$-th server and is independent of anything else.
\end{itemize}
The above-defined random variables and $\Ysuper(0)$ are ``mutually independent'' in the following sense:
$A^0$ is independent of $\Ysuper(0)$ conditioned on $\agea(0)$; $S_i^0$ is independent of $\Ysuper(0)$ conditioned on $\agei(0)$; for each $k\in\Zp$, $A^k$ is independent of everything else, i.e., $\Ysuper(0), \{S_i^\ell\}_{i\in[n], \ell\in\Znn}$ and $\{A^\ell\}_{\ell\neq k}$; for each $k\in\Zp$ and $i\in[n]$, $S_i^k$ is independent of everything else, i.e., $\Ysuper(0), \{S_j^\ell\}_{(j,\ell)\neq (i,k)}$, and $\{A^\ell\}_{\ell\in\Znn}$. 

Now we can write some basic quantities using the stochastic primitives defined above:

For each $k\in\Zp$, the $k$-th arrival time is $\taua^k = \sum_{\ell=0}^{k-1} A^{\ell}-\agea(0)$, and we let $\taua^0 = -\agea(0)$. 
For each time $t\geq 0$, the number of arrivals by time $t$ is given by $\Na(t) \triangleq \min \{k \in\Znn\colon \taua^k + A^k > t\}$; the age at time $t$ is $\agea(t) = t - \taua^{\Na(t)}$; the residual arrival time at time $t$ is $R_{a}(t) = A^{\Na(t)} - \agea(t)$. 

Similarly, on the $i$-th server, for each $k\in\Zp$, the $k$-th completion time is $\taui^k = \sum_{\ell=0}^{k-1} S_i^{\ell}-\agei(0)$, and we let $\taui^0=-\agei(0)$. 
For each time $t\geq 0$, the number of completions by time $t$ is $\Ni(t) = \min\{k\in\Znn \colon \taui^k + S_i^k > t\}$; the age at time $t$ is $\agei(t) = t - \taui^{\Ni(t)}$; the residual completion time at time $t$ is $R_{\si}(t) = S_i^{\Ni(t)} - \agei(t)$. 

We then define a set of filtrations, corresponding to the arrival process, completion processes, and the whole system: 
\begin{align*}
    \cF_{a,t} &= \sigma\Big(\agea(0), A^0, A^1, \dots, A^{\Na(t)-1}, \indibracbig{A^{\Na(t)} > t-\taua^{\Na(t)}}\Big) \quad \forall t\geq 0\\
    \cF_{i,t} &= \sigma\Big(\agei(0), S_i^0, S_i^1, \dots, S_i^{\Ni(t)-1}, \indibracbig{S_i^{\Ni(t)} > t - \taui^{\Ni(t)}}\Big) \quad \forall i\in[n], t\geq 0 \\
    \cF_t &= \sigma\Big(\Ysuper(0), \cF_{a,t}, \cup_{i\in[n]} \cF_{i,t}\Big) \quad \forall t\geq 0.
\end{align*}
We denote $\cF_{a,\infty} = \cup_{t\geq 0} \cF_{a,t}$, $\cF_{i,\infty} = \cup_{t\geq 0} \cF_{i,t}$ for $i\in[n]$, and $\cF_{\infty} = \cup_{t\geq 0} \cF_{t}$. 

It is not hard to see that $\Na(t), \agea(t)\in \cF_{a,t}$ for $t\geq 0$, but $R_{a}(t)\notin \cF_{a,t}$. Similarly, $\Ni(t), \agei(t)\in\cF_{i,t}$ for each $i\in[n]$ and $t\geq 0$, but $R_{\si}(t)\notin \cF_{i,t}$.
Furthermore, since $Q(t)$ is determined by the initial state and the arrival and completion events by time $t$, we have $Q(t) \in \cF_t$. 
Similarly, we have $\Qdropj(t) \in \sigma\big(\Qdropj(0), \cF_{a,t}, \cup_{i\neq j} \cF_{i,t} \big)$ for $j\in[n]$ and $t\geq 0$.

\subsection{Proof of BAR}
\label{app:pf-BAR}

In this subsection, we prove BAR for the heterogeneous case \Cref{lem:bar-het}, which will imply \Cref{lem:bar} for the homogeneous case as a special case. 
The arguments we use are adapted from Section~6 of \citep[][]{BraDaiMiy_23_bar_ss}.

\barhet*

\begin{proof}
    We let $\Xsuper(0)$ follow the stationary distribution $\pi$ at time $0$. 
    Let $\tau_1, \tau_2, \dots$ be the times of all arrivals or completion events after time $0$, indexed in temporal order. As assumed throughout our proofs, no events occur simultaneously, and we have sketched what should be changed for the general case in~\Cref{remark:simultaneous}. 

    \paragraph*{Proving BAR}
    In the first part of the proof, we prove the main BAR equation \eqref{eq:bar-het} utilizing the fundamental theorem of calculus.
    We start by examining the differentiability of $f(\Xsuper(t))$ between two consecutive events, $\tau^k$ and $\tau^{k+1}$, for each $k\in\Zp$. 
    Observe that when $t\in(\tau^k, \tau^{k+1})$, $\odv \Xsuper_\ell(t) / \odv t = - 1$ for $\ell=1,2,\dots, n+1$, and the other coordinates of $\Xsuper(t)$ remain constant. This implies that $\Xsuper(t)$ is Lipschitz continuous for $t\in[\tau^k, \tau^{k+1}]$. 
    Given the local Lipschitz continuity of $f(\cdot)$ and the compactness of $[\tau^k, \tau^{k+1}]$, we deduce that $f(\Xsuper(t))$ is Lipschitz continuous for $t\in[\tau^k, \tau^{k+1}]$. 
    Since Lipschitz continuity implies absolute continuity, $f(\Xsuper(t))$ is almost everywhere differentiable on $[\tau^k, \tau^{k+1}]$. Then for any $t$ where $f(\Xsuper(t))$ is differentiable, the derivative of $f(\Xsuper(t))$ agrees with the right derivative:
    \[
        \frac{\odv f(\Xsuper(t))}{\odv t} = - \sum_{\ell=1}^{n+1} \partial^+_\ell f(\Xsuper(t)) \triangleq \opinner f(\Xsuper(t)).
    \]
    Moreover, because the fundamental theorem of calculus holds for absolutely continuous functions, we have
    \begin{align}
        \nonumber
        f(\Xsuper(\tau^{k+1}\wedge t -)) - f(\Xsuper(\tau^{k}))  
        &= \int_{\tau^k}^{\tau^{k+1}\wedge t} \frac{\odv f(\Xsuper(u))}{\odv u} \odv u \\
        \label{eq:}
        &= \int_{\tau^k}^{\tau^{k+1}\wedge t} \opinner f(\Xsuper(u)) \odv u,
    \end{align}
    for $1\leq k < N(t)\triangleq\min\{k'\colon \tau^{k'} > t\}$. 
    Similarly, we can apply the fundamental theorem of calculus to $f(\Xsuper(\tau^{1}\wedge t)) - f(\Xsuper(0))$. Summing up all the resulting equations, we get
    \begin{equation}
        f(\Xsuper(t)) - f(\Xsuper(0)) = \int_0^t \opinner f(\Xsuper(u)) \odv u + \sum_{k=1}^{N(t)-1} \big(f(\Xsuper(\tau^k)) - f(\Xsuper(\tau^k-))\big).
    \end{equation}
    Rearranging the terms based on different types of events, we get
    \begin{equation}
        \label{eq:pf-bar:result-of-fct-1}
        \begin{aligned}
        f(\Xsuper(t)) - f(\Xsuper(0)) 
        &= \int_0^t \opinner f(\Xsuper(u)) \odv u + \sum_{k=1}^\infty \big(f(\Xsuper(\taua^k)) - f(\Xsuper(\taua^k-))\big) \indibrac{\taua^k \leq t}  \\
        &\mspace{23mu} +  \sumi \sum_{k=1}^\infty \big(f(\Xsuper(\taui^k)) - f(\Xsuper(\taui^k-))\big) \indibrac{\taui^k \leq t}.
        \end{aligned}
    \end{equation}
    
    Now we take the expectations on both sides of \eqref{eq:pf-bar:result-of-fct-1} to get 
    \begin{equation}
        \label{eq:pf-bar:result-of-fct-2}
        \begin{aligned}
        \E{f(\Xsuper(t)) - f(\Xsuper(0))}
        &= \E{\int_0^t \opinner f(\Xsuper(u)) \odv u} \\
        &\mspace{23mu} + \E{\sum_{k=1}^\infty \big(f(\Xsuper(\taua^k)) - f(\Xsuper(\taua^k-))\big) \indibrac{\taua^k \leq t}}  \\
        &\mspace{23mu} +  \sumi \E{\sum_{k=1}^\infty \big(f(\Xsuper(\taui^k)) - f(\Xsuper(\taui^k-))\big) \indibrac{\taui^k \leq t}}.
        \end{aligned}
    \end{equation}
    Because $\Xsuper(0)$ follows the stationary distribution $\pi$, $\Xsuper(0) \overset{d}{=} \Xsuper(u) \sim \pi$ for all $u\geq 0$. This implies that
    \begin{align*}
        \E{f(\Xsuper(t)) - f(\Xsuper(0))} &= 0 \\
        \E{\int_0^t \opinner f(\Xsuper(u)) \odv u} 
         &= \int_0^t \E{\opinner f(\Xsuper(u))}  \odv u \\
         &= t \, \Ep{\pi}{\opinner f(\Xsuper)},
    \end{align*}
    where the interchange of $\E{\cdot}$ with $\int_0^t\cdot$ is valid because $\opinner f(\Xsuper(u)) = - \sum_{\ell=1}^{n+1} \partial^+_\ell f(\Xsuper(u))$ is bounded for $u\in[0,t]$, due to the local Lipschitz continuity of $f(\cdot)$. Therefore, substituting the above two equations back to \eqref{eq:pf-bar:result-of-fct-2} and taking $t=1$, we get
    \begin{equation}
        \label{eq:pf-bar:simplify-continuous-terms}
        \begin{aligned}
        0&= \Ep{\pi}{\opinner f(\Xsuper)} + \E{\sum_{k=1}^\infty \big(f(\Xsuper(\taua^k)) - f(\Xsuper(\taua^k-))\big) \indibrac{\taua^k \leq 1}}  \\
        &\mspace{23mu} +  \sumi \E{\sum_{k=1}^\infty \big(f(\Xsuper(\taui^k)) - f(\Xsuper(\taui^k-))\big) \indibrac{\taui^k \leq 1}}.
        \end{aligned}
    \end{equation}
    Note that the second and third terms of \eqref{eq:pf-bar:simplify-continuous-terms} have similar forms as the Palm expectations: as defined in \Cref{sec:steady-state-anaysis:bar}, we have
    \begin{align*}
        \Ep{a}{f(\Xsuper\aft) - f(\Xsuper\bef)} &= \frac{1}{\Lambda} \E{\sum_{k=1}^\infty \big(f(\Xsuper(\taua^k)) - f(\Xsuper(\taua^k-))\big) \indibrac{\taua^k \leq 1}}  \\
        \Ep{\si}{f(\Xsuper\aft) - f(\Xsuper\bef)} &=  \frac{1}{\mu_i} \E{\sum_{k=1}^\infty \big(f(\Xsuper(\taui^k)) - f(\Xsuper(\taui^k-))\big) \indibrac{\taui^k \leq 1}} \quad \forall i\in[n].
    \end{align*}
    Substituting the definitions into \eqref{eq:pf-bar:simplify-continuous-terms}, we get the main equation of BAR:
    \begin{equation}
        \tag{\ref{eq:bar-het}}
        \Ep{\sdsuper}{\opinner f(\Xsuper)} + \Lambda \Ep{a}{f(\Xsuper\aft) -f(\Xsuper\bef)} + \sumi \mu_i \Ep{\si}{f(\Xsuper\aft) - f(\Xsuper\bef)} = 0.
    \end{equation} 

    \paragraph*{Characterizing Palm expectations} In the remainder of the proof, we prove \Cref{lem:bar-het}~(i) and \Cref{lem:bar-het}~(ii), which characterize the joint distribution of $\Xsuper\aft$ and $\Xsuper\bef$ under the Palm expectations $\Pnb_a$ and $\Pnb_{\si}$ for $i\in[n]$. 

    We start by proving \Cref{lem:bar-het}~(i). First, because $R_a(\taua^k-)=0$ for all $k\in\Zp$, by the definition of $\Pnb_a$, we have
    \[
        \Probp{a}{R_a=0} = \frac{1}{\Lambda}\E{\sum_{k=1}^\infty \indibrac{R_a(\taua^k-)=0, \taua^k\leq 1}} = 1.
    \]
    Next, to understand the relations between the coordinates of $\Xsuper\aft$ and $\Xsuper\bef$ under $\Pnb_a$, we examine the expectation $\Ep{a}{f(\Xsuper\bef)g(\Xsuper\aft-\Xsuper\bef)}$ for any bounded measurable functions $f, g\colon \R^{n+1} \times \Z^{1+\abs{\indexstable}} \to \R$:
    \begin{align}
        \nonumber
        \Ep{a}{f(\Xsuper\bef)g(\Xsuper\aft-\Xsuper\bef)} 
        &= \frac{1}{\Lambda} \E{\sum_{k=1}^\infty f\big(\Xsuper(\taua^k-)\big)g\big(\Xsuper(\taua^k) - \Xsuper(\taua^k-)\big)  \indibrac{\taua^k \leq 1}} \\
        \nonumber
        &= \frac{1}{\Lambda} \sum_{k=1}^\infty \E{f\big(\Xsuper(\taua^k-)\big)g\big(\Xsuper(\taua^k) - \Xsuper(\taua^k-)\big)  \indibrac{\taua^k \leq 1}}.
    \end{align}
    By the definitions of the sample paths, for each $k\in\Zp$, at the time of arrival $\taua^k$, 
    \begin{equation}
        \Xsuper(\taua^k) = \Xsuper(\taua^k-) + A^k \bm{e}_a + \bm{e}_q + \sum_{j\in\indexstable} \bm{e}_{q,-j}, 
    \end{equation}
    where $\bm{e}_a, \bm{e}_q, \bm{e}_{q,-j} \in\spacesuper$ are one-hot vectors with the non-zero entries at the coordinates of $R_a$, $Q$, and $\Qdropj$, respectively. %
    Note that $A^k$ is independent of $(\taua^k, \Xsuper(\taua^k-))$ and has the same distribution as $A$, so
    \begin{align}
        \nonumber
        \Ep{a}{f(\Xsuper\bef)g(\Xsuper\aft-\Xsuper\bef)} 
        &=  \frac{1}{\Lambda} \sum_{k=1}^\infty \E{f\big(\Xsuper(\taua^k-)\big)\indibrac{\taua^k \leq 1}}\, \Ebigg{g\Big( A \bm{e}_a + \bm{e}_q + \sum_{j\in\indexstable} \bm{e}_{q,-j}\Big)} \\
        \label{eq:pf-bar:characterize-pa-interm}
        &=  \Ep{a}{f\big(\Xsuper\bef\big)} \, \Ebigg{g\Big( A \bm{e}_a + \bm{e}_q + \sum_{j\in\indexstable} \bm{e}_{q,-j}\Big)}.
    \end{align}
    Because $f$ and $g$ can be arbitrary bounded measurable functions, \eqref{eq:pf-bar:characterize-pa-interm} implies that under $\Pnb_a$, $\Xsuper\aft - \Xsuper\bef$ is independent of $\Xsuper\bef$, and has the same distribution as $A \bm{e}_a + \bm{e}_q + \sum_{j\in\indexstable} \bm{e}_{q,-j}$. This proves \Cref{lem:bar-het}~(i). 

    To show \Cref{lem:bar-het}~(ii), note that for each $i\in[n]$ and $k\in\Zp$, we have $R_{\si}(\taui^k-) = 0$, which implies that $R_{\si}\bef = 0$ with probability $1$ under $\Pnb_{\si}$. 
    Moreover, by definitions,
    \begin{equation}
        \label{eq:pf-bar:completion-event-update}
        \Xsuper(\taui^k) = \Xsuper(\taui^k-) + S_i^k \bm{e}_{i} - \indibrac{Q(\taui^k-) > 0} \bm{e}_q - \sum_{j\in\indexstable \backslash \{i\}} \indibrac{\Qdropj(\taui^k-) > 0} \bm{e}_{q,-j},
    \end{equation}
    where $\bm{e}_i\in\spacesuper$ denotes the one-hot vector with the non-zero entry at the coordinate of $R_{\si}$. Because $S_i^k$ is independent of $(\taui^k, \Xsuper(\taui^k-))$ and has the same distribution as $S_i$, for any bounded measurable functions $f\colon\R^{n+1} \times \Z^{1+\abs{\indexstable}} \to \R$ and $g\colon \R\to\R$, we have
    \begin{align}
        \nonumber
        &\Ep{\si}{f(\Xsuper\bef)g(R_{\si}\aft-R_{\si}\bef)}  \\
        \nonumber
        &\qquad = \frac{1}{\Lambda} \sum_{k=1}^\infty \E{f\big(\Xsuper(\taui^k-)\big)g\big(R_{\si}(\taui^k) - R_{\si}(\taui^k-)\big)  \indibrac{\taui^k \leq 1}} \\
        \nonumber
        &\qquad =  \frac{1}{\Lambda} \sum_{k=1}^\infty \E{f\big(\Xsuper(\taui^k-)\big)\indibrac{\taui^k \leq 1}}\, \Ebig{g(S_i)} \\
        \label{eq:pf-bar:characterize-pi-interm}
        &\qquad =  \Ep{\si}{f\big(\Xsuper\bef\big)} \, \Ebig{g(S_i)}.
    \end{align}
    This implies that under $\Pnb_{\si}$, $R_{\si}\aft - R_{\si}\bef$ is independent of $\Xsuper\bef$ and has the same distribution as $S_i$. 
    In addition, for any bounded measurable function $f\colon \R^{n+1} \times \Z^{1+\abs{\indexstable}} \times \Z\to\R$, 
    \begin{align}
        \nonumber
        \Ep{\si}{f(\Xsuper\bef, Q\aft)}
        &= \frac{1}{\Lambda} \sum_{k=1}^\infty \E{f\big(\Xsuper(\taui^k-), Q(\taui^k)\big) \indibrac{\taui^k \leq 1}} \\
        \nonumber
        &= \frac{1}{\Lambda} \sum_{k=1}^\infty \E{f\Big(\Xsuper(\taui^k-),\; Q(\taui^k-) - \indibrac{ Q(\taui^k-) > 0} \Big) \indibrac{\taui^k \leq 1}}  \\
        \nonumber
        &= \Ep{\si}{f\big(\Xsuper\bef,\; Q\bef - \indibrac{Q\bef>0}\big)},
    \end{align}
    which implies that $Q\aft = Q\bef - \indibrac{Q\bef > 0}$ under the distribution $\Pnb_{\si}$ with probability $1$. Similarly, we can show that $\Qdropj\aft = \Qdropj\bef - \indibracbig{\Qdropj\bef > 0}$ for $j\in\indexstable\backslash \{i\}$, $R_a\aft = R_a\bef$, and $R_{\sj}\aft = R_{\sj}\bef$ for $j\in[n]\backslash \{i\}$ under $\Pnb_{\si}$ with probability $1$. 
    This proves \Cref{lem:bar-het}~(ii). 
    \qedhere
\end{proof}

\subsection{Proof of bounded mean residual times (Lemmas~\ref{lem:R-bdd} and \ref{lem:R-bdd-het})}
\label{app:pf-Rbdd}

In this section, we prove \Cref{lem:R-bdd} and \Cref{lem:R-bdd-het}, restated as follows: 

\Rbdd*

\Rbddhet*

Below, we focus on proving \Cref{lem:R-bdd-het} since it is a strict generalization of \Cref{lem:R-bdd}.

\begin{proof}[Proof of \Cref{lem:R-bdd-het}]
    We first show \eqref{eq:R-sj-bound-stationary} for $\nu=\pi$. By the construction of the sample paths in \Cref{app:notation-support-lemma} and the tower property of conditional expectations,  
    \begin{align}
        \nonumber
        \Ep{\pi}{R_{\sj} \given Q, \Qdropstable}
        &= \E{R_{\sj}(0) \given Q(0), \Qdropstable(0)} \\
        \nonumber
        &= \E{\E{R_{\sj}(0) \given Q(0), \Qdropstable(0), \agej(0)}\given Q(0), \Qdropstable(0)} \\
        \label{eq:pf-R-bdd:Epi-R-intermediate}
        &= \E{\E{S_j^0 - \agej(0) \given Q(0), \Qdropstable(0), \agej(0)}\given Q(0), \Qdropstable(0)}.
    \end{align}
    The inner conditional expectation in the last expression can be bounded as follows: 
    \begin{align}
        \label{eq:pf-R-bdd:apply-cond-indep}
        \E{S_j^0 - \agej(0) \given Q(0), \Qdropstable(0), \agej(0)} &=\E{S_j^0 - \agej(0) \given \agej(0)} \\
        \label{eq:pf-R-bdd:apply-S0-law}
        &= \E{S_j - \agej(0) \given S_j\geq \agej(0)} \\
        \label{eq:pf-R-bdd:apply-assump-bdd}
        &\leq \Rsmaxj / \mu_j.
    \end{align}
    where \eqref{eq:pf-R-bdd:apply-cond-indep} is because $S_j^0$ is by definition independent of $Q(0)$ and $\Qdropstable(0)$ given $\agej(0)$; \eqref{eq:pf-R-bdd:apply-S0-law} is by the distribution of $S_j^0$; \eqref{eq:pf-R-bdd:apply-assump-bdd} is due to the definition $\Rsmaxj \triangleq \sup_{t'\geq 0} \E{\mu_j S_j - t' \given \mu_j S_j\geq t'}$. Substituting the bound in \eqref{eq:pf-R-bdd:apply-assump-bdd} back to \eqref{eq:pf-R-bdd:Epi-R-intermediate} yields $\mu_j \Ep{\pi}{R_{\sj}\given Q, \Qdropstable} \leq \Rsmaxj$.

    Next, we prove \eqref{eq:R-sj-bound-stationary} for $\nu=\Pnb_{\si}$ with $i\neq j$. Let $(q,\qdropstable)$ be any realization of $(Q, \Qdropstable)$. 
    By the definition of $\Ep{\si}{\cdot}$, we have
    \begin{align}
        \nonumber
        &\Ep{\si}{R_{\sj} \indibracbig{(Q, \Qdropstable)=(q,\qdropstable)}} \\
        \nonumber
        &\qquad = \frac{1}{\mu_i}\E{\sum_{k=1}^\infty \indibracBig{\taui^k \leq 1, \big(Q(\taui^k-), \Qdropstable(\taui^k-)\big)=(q,\qdropstable)} R_{\sj}(\taui^k-)} \\
        \nonumber
        &\qquad = \frac{1}{\mu_i} \sum_{k=1}^\infty \mathbb{E}\Bigg[\mathbb{E}\Bigg[\indibracBig{\taui^k \leq 1, \big(Q(\taui^k-), \Qdropstable(\taui^k-)\big)=(q,\qdropstable)} \\
        \nonumber
        &\qquad \mspace{110mu} \cdot R_{\sj}(\taui^k-) \;\Bigg|\; \Ysuper(0), \cF_{a,\infty}, \{\cF_{\ell,\infty}\}_{\ell\neq j}, \cF_{j,\taui^k}\Bigg]\Bigg] \\
        \label{eq:pf-R-bdd:E-si-intermediate}
        &\qquad = \frac{1}{\mu_i} \sum_{k=1}^\infty \mathbb{E}\Bigg[\indibracbig{\taui^k \leq 1} \indibracbig{\big(Q(\taui^k-), \Qdropstable(\taui^k-)\big)=(q,\qdropstable)} \\
        \nonumber
        &\qquad \mspace{80mu} \cdot\E{R_{\sj}(\taui^k-)  \given \Ysuper(0), \cF_{a,\infty}, \{\cF_{\ell,\infty}\}_{\ell\neq j}, \cF_{j,\taui^k}}\Bigg],
    \end{align}
    where \eqref{eq:pf-R-bdd:E-si-intermediate} is because $\taui^k$, $Q(\taui^k-)$, and $\Qdropstable(\taui^k-)$ are measurable with respect to the sigma algebra \[\sigma\Big(\Ysuper(0), \cF_{a,\infty}, \{\cF_{\ell,\infty}\}_{\ell\neq j}, \cF_{j,\taui^k}\Big).\]
    To bound the conditional expectation of $R_{\sj}(\taui^k-)$ in \eqref{eq:pf-R-bdd:E-si-intermediate}, we rewrite it as:
    \begin{align}
        \nonumber
        &\E{R_{\sj}(\taui^k-)  \given \Ysuper(0), \cF_{a,\infty}, \{\cF_{\ell,\infty}\}_{\ell\neq j}, \cF_{j,\taui^k}} \\
        \nonumber
        &\qquad = \E{S_{j}^{N_j(\taui^k-)} - \agej(\taui^k-)  \given \Ysuper(0), \cF_{a,\infty}, \{\cF_{\ell,\infty}\}_{\ell\neq j}, \cF_{j,\taui^k}} \\
        \label{eq:pf-R-bdd:E-si-intermediate-2}
        &\qquad = \sum_{m=0}^\infty \E{\Big(S_{j}^{m}- \agej(\taui^k-)\Big) \indibrac{N_j(\taui^k-)=m} \given \Ysuper(0), \cF_{a,\infty}, \{\cF_{\ell,\infty}\}_{\ell\neq j}, \cF_{j,\taui^k}}.
    \end{align}
    Because $\taui^k$ is measurable with respect $\sigma\big(\Ysuper(0), \cF_{i,\infty}\big)$, we can rewrite \eqref{eq:pf-R-bdd:E-si-intermediate-2} as $\sum_{m=0}^\infty U_m(\taui^k-)$, where for each $t \geq 0$, 
    $U_m(t)$ is the random variable define as
    \begin{align*}
        U_m(t) 
        &\triangleq \E{\Big(S_{j}^{m} - \agej(t)\Big) \indibrac{N_j(t)=m} \given \Ysuper(0), \cF_{a,\infty}, \{\cF_{\ell,\infty}\}_{\ell\neq j}, \cF_{j,t}}. 
    \end{align*}
    For each $t\geq 0$ and $m\in\Znn$, $U_m(t)$ can be bounded as
    \begin{align}
        \label{eq:pf-R-bdd:apply-cF-independence}
         U_m(t) &= \E{\Big(S_{j}^{m} - \agej(t)\Big) \indibracbig{N_j(t)=m} \given \cF_{j,t}} \\
         \label{eq:pf-R-bdd:apply-defs-N-and-cF}
        &= \E{\Big(S_{j}^{m} - (t-\tauj^m)\Big) \indibracbig{\tauj^m \leq t < \tauj^m + S_{j}^{m}} \given \agej(0),  \{S_j^{\ell}\}_{\ell=0}^{m-1}, \indibracbig{S_j^{m} > (t - \tauj^m)}} \\
        \nonumber
        &= \left(\E{S_{j}^{m} \given \agej(0),  \{S_j^{\ell}\}_{\ell=0}^{m-1}, \indibracbig{S_j^{m} > (t - \tauj^m)}} - (t-\tauj^m) \right) \cdot \indibracbig{\tauj^m \leq t < \tauj^m + S_{j}^{m}} \\ 
        \label{eq:pf-R-bdd:apply-Sm-def}
        &= \left(\E{S_{j} \given S_j > t-\tauj^m} - (t-\tauj^m) \right) \cdot \indibracbig{\tauj^m \leq t < \tauj^m + S_{j}^{m}}  \\
        \label{eq:pf-R-bdd:apply-assump-bdd-2}
        &\leq (\Rsmaxj / \mu_j) \cdot \indibracbig{\tauj^m \leq t < \tauj^m + S_{j}^{m}}. 
     \end{align}
     Here, \eqref{eq:pf-R-bdd:apply-cF-independence} follows from the conditional independence between $(S_j^m, \agej(t), \Nj(t))$ and $\sigma\Big(\Ysuper(0), \cF_{a,\infty}, \{\cF_{\ell,\infty}\}_{\ell\neq j}\Big)$ given $\agej(0)$, and the fact that $\agej(0)\in\cF_{j,t}$. Furthermore, \eqref{eq:pf-R-bdd:apply-defs-N-and-cF} is due to the definitions of $\Nj(t)$ and $\cF_{j,t}$;  
    \eqref{eq:pf-R-bdd:apply-Sm-def} is by the definition of $S_j^m$ and its independence with $\big(\agej(0), \{S_j^{\ell}\}_{\ell=0}^{m-1}, \tauj^m)$ if $m\geq 1$; 
    \eqref{eq:pf-R-bdd:apply-assump-bdd-2} follows from the definition  $\Rsmaxj \triangleq \sup_{t'\geq 0} \E{\mu_j S_j - t' \given \mu_j S_j\geq t'}$.
    Substituting \eqref{eq:pf-R-bdd:apply-assump-bdd-2} back to \eqref{eq:pf-R-bdd:E-si-intermediate-2} implies that for $k\in\Zp$,
    \begin{align*}
        \E{R_{\sj}(\taui^k-)  \given \Ysuper(0), \cF_{a,\infty}, \{\cF_{\ell,\infty}\}_{\ell\neq j}, \cF_{j,\taui^k}}
        & \leq \sum_{m=0}^\infty U_m(\taui^k-) \\
        & \leq (\Rsmaxj / \mu_j) \cdot \sum_{m=0}^\infty \indibracbig{\tauj^m < \taui^k \leq \tauj^m + S_{j}^{m}} \\
        & = (\Rsmaxj / \mu_j) \cdot \indibrac{\tauj^0 < \taui^k < \infty} \\
        & \leq \Rsmaxj / \mu_j \quad a.s.
    \end{align*}
    Further substituting the above inequality back to  \eqref{eq:pf-R-bdd:E-si-intermediate} yields
    \begin{align*}
        &\Ep{\si}{R_{\sj} \indibracbig{(Q, \Qdropstable)=(q,\qdropstable)}} \\
        &\qquad \leq (\Rsmaxj / \mu_j)\cdot \frac{1}{\mu_i} \sum_{k=1}^\infty \E{\indibracbig{\taui^k \leq 1} \indibracbig{\big(Q(\taui^k-), \Qdropstable(\taui^k-)\big)=(q,\qdropstable)}} \\
        &\qquad = (\Rsmaxj / \mu_j) \cdot \Probp{\si}{\big(Q, \Qdropstable\big)=(q,\qdropstable)},
    \end{align*}
    where $(q,\qdropstable)$ is an arbitrary realization of $(Q, \Qdropstable)$. This proves $\Ep{\si}{R_{\sj}\given Q, \Qdropstable} \leq \Rsmaxj/\mu_j$. 

    The proof for \eqref{eq:R-a-bound-stationary} mirrors the proof for \eqref{eq:R-sj-bound-stationary} with the following two modifications: 
    \begin{enumerate}
        \item All quantities associated with the $j$-th server's completion process are swapped with the corresponding quantities associated with the arrival process. 
        \item The definition of $\Ramin$ is used in place of $\Rsmaxj$. 
    \end{enumerate}
    Given the similarity, the detailed proof for \eqref{eq:R-a-bound-stationary} is omitted. 
    \qedhere
\end{proof}

\subsection{Proof of the uncorrelation lemmas (Lemmas~\ref{lem:indep-R-Qtilde} and \ref{lem:indep-R-Qtilde-het})}
\label{app:pf-uncorrelation}

In this subsection, we prove the uncorrelation lemmas, \Cref{lem:indep-R-Qtilde} and \Cref{lem:indep-R-Qtilde-het}, restated as follows: 

\indep*

\indephet*

We focus on the uncorrelation lemma for the heterogeneous setting (\Cref{lem:indep-R-Qtilde-het}), since it is a strict generalization of \Cref{lem:indep-R-Qtilde}. 
The proof of \Cref{lem:indep-R-Qtilde-het} is based on the following lemma:

\begin{lemma}
    \label{lem:E-R-t-conv}
    Consider the modified \ggn queue with heterogeneous servers. For any $j\in[n]$, consider the residual service time process of the $j$-th server, $\{R_{\sj}(t)\}_{t\geq 0}$, and assume that it has non-lattice service times. 
    Then for any $\epsilon > 0$, there exists a deterministic $T_0 > 0$ such that for any $t \geq T_0 + R_{\sj}(0)$, we have 
    \[
        \abs{\E{R_{\sj}(t) \given R_{\sj}(0)} - \Ep{\pi}{R_{\sj}}} \leq \epsilon \quad \text{w.p. } 1.
    \]
\end{lemma}

\begin{proof}[Proof of \Cref{lem:E-R-t-conv}]
    We first consider the case when $R_{\sj}(0) = 0$, so $R_{\sj}(t)$ is a pure renewal process. Because the inter-event time  $S_j$ of this renewal process is non-lattice (\Cref{assump:non-lattice-het}), 
    we can invoke the Key Renewal Theorem \citep[][Theorem V.4.3]{Asm_03}. Restated in our setting, this theorem implies that 
    \begin{equation*}
        \lim_{t\to\infty} \E{R_{\sj}(t) \given R_{\sj}(0) = 0}  = \Ep{\pi}{R_{\sj}}, 
    \end{equation*}
    if we can show that $z(t) \triangleq  \E{\max(S_j-t, 0)}$ is \emph{directly Riemann integrable} and satisfies
    \begin{equation}
        \label{eq:pf-E-R-t-conv:z-equation}
        \E{R_{\sj}(t) \given R_{\sj}(0) = 0} = z(t) + \int_0^t \E{R_{\sj}(t- u) \given R_{\sj}(0) = 0}  F_j(\odv u), 
    \end{equation}
    where $\int \,\cdot\, F_j(\odv u)$ denotes the integral with respect to $u$ when $u$ follows the distribution of $S_j$. 

    We first show that $z(t)$ is directly Riemann integrable. By Proposition V.4.1 of \citep{Asm_03}, $z(t)$ is directly Riemann integrable if it is non-increasing and Lebesgue integrable on $[0,\infty)$. Obviously, $z(t)$ is non-increasing. Moreover, $z(t)$ is non-negative and its Lebesgue integral can be bounded as
    \begin{align*}
        \int_0^\infty z(t) 
        &= \int_0^\infty \E{\max(S_j-t, 0)} \odv t \\
        &= \E{\int_0^\infty \max(S_j-t, 0) \odv t} \\
        &= \frac{1}{2}\E{S_j^2} \\
        &\leq \frac{1}{2}(\Rsmaxj)^2,
    \end{align*}
    which is finite by \Cref{assump:bdd-exp-remain-het} (or \Cref{assump:bdd-exp-remain} in the homogeneous case). 

    Next, we show that $z(t)$ satisfies \eqref{eq:pf-E-R-t-conv:z-equation}. Recall that $S_j^1$ is the service time of the job that starts service at time $R_{\sj}(0)$. We decompose $\E{R_{\sj}(t) \given R_{\sj}(0) = 0}$ based on whether $S_j^1 > t$ to get the following: 
    \begin{align*}
        \E{R_{\sj}(t) \given R_{\sj}(0) = 0}
        &= \E{R_{\sj}(t) \indibracbig{S_j^1  > t} \given R_{\sj}(0) = 0} \\
        &\mspace{23mu} + \E{R_{\sj}(t) \indibracbig{S_j^1  \leq t} \given R_{\sj}(0) = 0} \\
        &= \E{\max(S_j^1-t, 0)} +  \int_0^t \E{R_{\sj}(t)\given R_{\sj}(0) = 0, S_j^1 =u} F_j(\odv u) \\
        &= z(t) + \int_0^t \E{R_{\sj}(t- u) \given R_{\sj}(0) = 0}  F_j(\odv u). 
    \end{align*}

    The above arguments verify the assumptions of the Key Renewal Theorem. Therefore, given $R_{\sj}(0)= 0$, for any $\epsilon > 0$, there exists $T_0 > 0$ such that for any $t \geq T_0$, 
    \begin{equation}
        \label{eq:pf-E-R-t-conv:pure-renewal-result}
        \abs{\E{R_{\sj}(t) \given R_{\sj}(0) = 0}  - \Ep{\pi}{R_{\sj}}} \leq \epsilon. 
    \end{equation}
    In general, when $R_{\sj}(0)\neq 0$, $R_{\sj}(t)$ is a delayed renewal process and we have 
    \[
        \E{R_{\sj}\big(t + R_{\sj}(0)\big) \given R_{\sj}(0)} = \E{R_{\sj}(t) \given R_{\sj}(0) = 0} \quad \forall t> 0. 
    \]
    Substituting the last equation into \eqref{eq:pf-E-R-t-conv:pure-renewal-result} finishes the proof. 
\end{proof}

Now we are ready to prove \Cref{lem:indep-R-Qtilde-het}. 

\begin{proof}[Proof of \Cref{lem:indep-R-Qtilde-het}]
    By the definition of the Palm expectation $\Ep{\si}{\cdot}$, for arbitrary $T>0$, we have
    \begin{align}
        \nonumber
        &\Ep{\si}{\indibrac{\Qdropj=0}R_{\sj}} \\
        \nonumber
        &\qquad = \frac{1}{\mu_i T} \E{\sum_{k=1}^\infty \indibrac{\taui^k\leq T,\Qdropj(\taui^k-)=0} R_{\sj}(\taui^k-)} \\
        \nonumber
        &\qquad = \frac{1}{\mu_i T} \sum_{k=1}^\infty \E{\E{\indibrac{\taui^k\leq T,\Qdropj(\taui^k-)=0} R_{\sj}(\taui^k-) \given \Ysuper(0), \cF_{a,\infty}, \{\cF_{\ell,\infty}\}_{\ell\neq j}}}  \\
        \label{eq:pf-indep:take-out-Q}
        &\qquad = \frac{1}{\mu_i T} \sum_{k=1}^\infty \E{\indibrac{\taui^k\leq T,\Qdropj(\taui^k-)=0} \cdot \E{R_{\sj}(\taui^k-) \given \Ysuper(0), \cF_{a,\infty}, \{\cF_{\ell,\infty}\}_{\ell\neq j}}} \\
       \label{eq:pf-indep:rewrite-as-V}
        &\qquad = \frac{1}{\mu_i T} \sum_{k=1}^\infty \E{\indibrac{\taui^k\leq T,\Qdropj(\taui^k-)=0} \cdot \Rmean(\taui^k-)},
    \end{align}
    where \eqref{eq:pf-indep:take-out-Q} is because $\big(\taui^k, \Qdropj(\taui^k-)\big) \in \sigma\Big(\Ysuper(0), \cF_{a,\infty}, \{\cF_{\ell,\infty}\}_{\ell\neq j}\Big)$; $\Rmean(\cdot)$ is defined as
    \[
        \Rmean(t) \triangleq \E{R_{\sj}(t) \given \Ysuper(0), \cF_{a,\infty}, \{\cF_{\ell,\infty}\}_{\ell\neq j}} \quad \forall t\geq 0,
    \]
    and \eqref{eq:pf-indep:rewrite-as-V} is due to the fact that $\taui^k\in\cF_{i,\infty}$. Note that because the process $R_{\sj}(t)$ is conditionally independent of $\Ysuper(0)$ given $\agej(0)$, we can simplify $\Rmean(t)$ into
    \begin{equation}
        \Rmean(t) = \E{R_{\sj}(t) \given \agej(0)}.
    \end{equation}
    In the subsequent argument, we will use the fact that
    \begin{equation}
        \label{eq:pf-indep:Vt-conditional-rule}
        \Rmean(t) = \E{\E{R_{\sj}(t) \given R_{\sj}(0)} \given \agej(0)},
    \end{equation}
    which is true because 
    $R_{\sj}(t) = R_{\sj}(0) + \sum_{\ell=1}^{\Nj(t)} S_j^\ell - t$ and $\agej(0) = S_j^0 - R_{\sj}(0)$ are conditionally independent given $R_{\sj}(0)$.

    Next, we bounds the difference between $\Rmean(\taui^k-)$ and $\Ep{\pi}{R_{\sj}}$ for each $k\in\Zp$.
    Recall that by \Cref{lem:E-R-t-conv}, for any $\epsilon > 0$, there exists a deterministic $T_0 > 0$ such that for any $t \geq T_0 + R_{\sj}(0)$, we have 
    \begin{equation}
        \label{eq:E-R-t-conv}
        \abs{\E{R_{\sj}(t) \given R_{\sj}(0)} - \Ep{\pi}{R_{\sj}}} \leq \epsilon \quad \text{w.p. } 1.
    \end{equation}
    Let the random variable $K^\epsilon \triangleq \min\{k' \in\Zp \colon \taui^{k'+1} > T_0 + R_{\sj}(0)\}$. Then by \eqref{eq:pf-indep:Vt-conditional-rule} and \eqref{eq:E-R-t-conv}, we have 
    \begin{equation}
        \label{eq:E-R-at-tau-k-conv}
         \abs{\Rmean(\taui^k-) - \Ep{\pi}{R_{\sj}}} \leq \epsilon \quad  \text{ if } k\geq K^\epsilon + 1. 
    \end{equation}
    Moreover, we argue that 
    \begin{equation}
        \label{eq:K-eps-bdd-exp}
        \E{K^\epsilon} \leq \mu_i T_0 + \mu_i \Ep{\pi}{R_{\sj}} + \E{(\mu_i S_i)^2} < \infty. 
    \end{equation}
    To see this, recall that $\taui^{k'+1} = \sum_{\ell=0}^{k'} S_i^\ell - \agei(0) = \sum_{\ell=1}^{k'} S_i^\ell + R_{\si}(0)$, where $S_i^\ell$ are i.i.d. for $\ell\geq 1$. 
    We can thus rewrite $K^\epsilon = \inf\{k'\in\Zp \colon \sum_{\ell=1}^{k'} S_i^\ell > T_0 + R_{\sj}(0) - R_{\si}(0)\}$, which shows that $K^\epsilon$ is a stopping time with respect to the discrete-time random walk $\big(\sum_{\ell=1}^{k'} S_i^\ell \big)_{k'=1}^\infty$. By Wald's equation and Lorden's inequality \citep{Lor_70_lorden_excess}, we have
    \begin{align*}
        \E{K^\epsilon \given R_{\sj}(0), R_{\si}(0)} \E{S_i}
        &= \E{\sum_{\ell=1}^{K^\epsilon} S_i^\ell \given R_{s,j}(0), R_{s,i}(0)} \\
        &\leq \max(T_0 + R_{\sj}(0) - R_{\si}(0), 0) + \mu_i \E{S_i^2},
    \end{align*}
    which implies \eqref{eq:K-eps-bdd-exp} after rearranging the terms. 
    
    Now we can finish the proof utilizing \eqref{eq:pf-indep:rewrite-as-V},  \eqref{eq:E-R-at-tau-k-conv}, and \eqref{eq:K-eps-bdd-exp}.
    First, by the definition of $\Probp{\si}{\Qdropj=0}$, we have that for any $T>0$,
    \begin{equation*}
         \Probp{\si}{\Qdropj=0}
         = \frac{1}{\mu_i T} \E{\sum_{k=1}^\infty \indibrac{\taui^k\leq T,\Qdropj(\taui^k-)=0 }}
    \end{equation*}   
    Combining the last equality with the equality in \eqref{eq:pf-indep:rewrite-as-V}, we get
    \begin{align*}
        &\abs{\Ep{\si}{\indibrac{\Qdropj=0}R_{\sj}} - \Probp{\si}{\Qdropj=0} \Ep{\pi}{R_{\sj}}} \\
        &\qquad = \frac{1}{\mu_i T}\Bigg\lvert \E{\sum_{k=1}^\infty \indibrac{\taui^k\leq T,\Qdropj(\taui^k-)=0} \cdot \Rmean(\taui^k-)} \\
        &\quad \mspace{70mu} -   \E{\sum_{k=1}^\infty \indibrac{\taui^k\leq T,\Qdropj(\taui^k-)=0 }}\Ep{\pi}{R_{\sj}} \Bigg\rvert \\
        &\qquad =  \frac{1}{\mu_i T} \abs{\E{ \sum_{k=1}^\infty\indibrac{\taui^k\leq T, \Qdropj(\taui^k-)=0} \left(\Rmean(\taui^k-)  - \Ep{\pi}{R_{\sj}}\right)}} \\
        &\qquad \leq \frac{1}{\mu_i T} \E{\sum_{k=1}^\infty \indibrac{\taui^k\leq T} \cdot \absBig{\Rmean(\taui^k-)  - \Ep{\pi}{R_{\sj}}}}.
    \end{align*}
    Fixing any $\epsilon > 0$, we break the infinite sum above into two parts based on $K^\epsilon$: 
    \begin{align}
        \nonumber
        &\abs{\Ep{\si}{\indibrac{\Qdropj=0}R_{\sj}} - \Probp{\si}{\Qdropj=0} \Ep{\pi}{R_{\sj}}} \\
        \nonumber
        &\qquad \leq \frac{1}{\mu_i T} \E{\sum_{k=1}^\infty \indibrac{\taui^k\leq T} \cdot \absBig{\Rmean(\taui^k-)  - \Ep{\pi}{R_{\sj}}}} \\
        \label{eq:pf-indep:before-split-use-Keps}
        &\qquad =  \frac{1}{\mu_i T} \E{\sum_{k=1}^{K^\epsilon} \indibrac{\taui^k\leq T} \abs{\Rmean(\taui^k-)  - \Ep{\pi}{R_{\sj}}}}  \\
        \nonumber
        &\qquad \mspace{23mu} +  \frac{1}{\mu_i T} \E{\sum_{k=K^\epsilon+1}^\infty \indibrac{\taui^k\leq T} \abs{\Rmean(\taui^k-)  - \Ep{\pi}{R_{\sj}}}} \\
        \label{eq:pf-indep:split-use-Keps}
        &\qquad \leq \frac{2\Rsmaxj\E{K^\epsilon}}{\mu_i T} +  \frac{\epsilon}{\mu_i T} \E{\sum_{k=K^\epsilon+1}^\infty \indibrac{\taui^k\leq T}} \\
        \label{eq:pf-indep:sum-indic-exp}
        &\qquad \leq  \frac{2\Rsmaxj \E{K^\epsilon}}{\mu_i T} + \epsilon.
    \end{align}
    Here, to get \eqref{eq:pf-indep:split-use-Keps}, we bound the expectation in \eqref{eq:pf-indep:before-split-use-Keps} using the argument that
    \[
        \abs{\Rmean(\taui^k-) - \Ep{\pi}{R_{\sj}}} \leq \Rmean(\taui^k-) + \Ep{\pi}{R_{\sj}} \leq 2\Rsmaxj \quad a.s.
    \]
    To get \eqref{eq:pf-indep:sum-indic-exp}, we use the fact that $\E{\sum_{k=1}^\infty \indibrac{\taui^k\leq T}} = \mu_i T$. 
    Now, because $\E{K^\epsilon}<\infty$, and $T$ and $\epsilon$ are arbitrary, first taking $T\to\infty$ and then taking $\epsilon \to 0$ in \eqref{eq:pf-indep:sum-indic-exp} finish the proof. \qedhere
\end{proof}

\subsection{Proofs of some steady-state equalities and finiteness results}
\label{app:pf-equalities-and-finiteness}

In this subsection, we will prove several basic equalities about the residual times of the completion processes and the idle probability of the queue (Lemmas~\ref{lem:apply-bar-lots-of} and \ref{lem:apply-bar-lots-of-het}). We will also prove that modified \ggn queues have finite mean queue lengths under the Palm distributions $\Pnb_{a}$ and $\Pnb_{\si}$ (Lemmas~\ref{lem:expfinite} and \ref{lem:expfinite-het}). 
All the results in this subsection are derived from BAR (\Cref{lem:bar-het}), which we restate below for ease of reference:

\barhet*

Now, we restate \Cref{lem:apply-bar-lots-of} and its generalization in the heterogeneous setting, \Cref{lem:apply-bar-lots-of-het}. We will directly prove \Cref{lem:apply-bar-lots-of-het}. 

\applybarpreliminary*

\applybarpreliminaryhet*

\begin{proof}[Proof of \Cref{lem:apply-bar-lots-of-het}]    
    \noindent \textbf{Proof of \eqref{eq:E-pi-R-s-het}.} For any truncation threshold $\kappa > 0$, consider the test function $f(\Xsuper) = (R_{\si} \wedge \kappa)^2$, which is bounded, locally Lipschitz continuous, and has right partial derivatives. Consequently, Basic Adjoint Relationship \eqref{eq:bar-het} implies that
    \[
        -2\Ep{\pi}{R_{\si} \indibrac{R_{\si} < \kappa}} + \Lambda \Ep{\si}{\big((R_{\si}+S_i) \wedge \kappa\big)^2 - \big(R_{\si} \wedge \kappa\big)^2} = 0.
    \]
    Because $R_{\si} = 0$ with probability $1$ under the measure $\Pnb_{\si}$, rearranging the terms, we get
    \begin{equation*}
        \Ep{\pi}{R_{\si} \indibrac{R_{\si} < \kappa}} = \frac{1}{2} \Lambda \E{(S_i\wedge\kappa)^2}. 
    \end{equation*}
    Observe that $R_{\si} \indibrac{R_{\si} < \kappa}$ and $(S_i\wedge\kappa)^2$ are both non-negative and non-decreasing in $\kappa$, by the monotone convergence theorem, 
    \begin{equation}
        \Ep{\pi}{R_{\si}} = \lim_{\kappa\to\infty} \Ep{\pi}{R_{\si} \indibrac{R_{\si} < \kappa}} = \lim_{\kappa\to\infty} \frac{1}{2} \Lambda \E{(S_i\wedge\kappa)^2} = \frac{1}{2} \Lambda \E{S_i^2},
    \end{equation}
    which implies \eqref{eq:E-pi-R-s-het}. 

    \noindent \textbf{Proof of \eqref{eq:E-pi-R-a-het}.} 
    This proof is identical to the proof of \eqref{eq:E-pi-R-s-het} except that $R_{\si}$ is replaced by $R_a$. 

    \noindent \textbf{Proof of \eqref{eq:E-s-R-s-het}.}
    For any $\kappa > 0$, we take the test function $f(\Xsuper) = (R_{\si} \wedge \kappa)(R_{\sj} \wedge \kappa)$ in \eqref{eq:bar-het}:
    \begin{align*}
        & - \Ep{\pi}{\indibrac{R_{\si}<\kappa} (R_{\sj}\wedge\kappa) + \indibrac{R_{\sj}<\kappa} (R_{\si}\wedge\kappa)} \\
        &+ \mu_i \Ep{\si}{(S_i\wedge\kappa)(R_{\sj}\wedge\kappa) - 0} + \mu_j \Ep{\sj}{(S_j\wedge\kappa)(R_{\si}\wedge\kappa) - 0}
        =0,
    \end{align*}
    where in the second and third expectations we have used the facts that $R_{\sell} = 0$ with probability $1$ under $\Pnb_{\sell}$ for any $\ell\in[n]$. 
    Because $S_i$ is independent of $R_{\sj}$ under the measures $\Pnb_{\si}$, and $S_j$ is independent of $R_{\si}$ under the measure $\Pnb_{\sj}$, we get
    \begin{equation*}
        \begin{aligned}
        &\Ep{\pi}{\indibrac{R_{\si}<\kappa} (R_{\sj}\wedge\kappa) + \indibrac{R_{\sj}<\kappa} (R_{\si}\wedge\kappa)} \\
        &\qquad =\mu_i \E{S_i\wedge\kappa} \Ep{\si}{R_{\sj}\wedge\kappa} +  \mu_j \E{S_j\wedge\kappa}\Ep{\sj}{R_{\si}\wedge\kappa}. 
        \end{aligned}
    \end{equation*}
    Because $\indibrac{R_{\si}<\kappa} (R_{\sj}\wedge\kappa)$, $\indibrac{R_{\sj}<\kappa} (R_{\si}\wedge\kappa)$, $S_i\wedge\kappa$, $S_j\wedge\kappa$, $R_{\sj}\wedge\kappa$, and $R_{\si}\wedge\kappa$ are all non-negative and non-decreasing in $\kappa$, by the monotone convergence theorem, we have the following after letting $\kappa\to\infty$: 
    \begin{equation*}
        \Ep{\pi}{R_{\sj} + R_{\si}} = \mu_i \E{S_i} \Ep{\si}{R_{\sj}} + \mu_j \E{S_j} \Ep{\sj}{R_{\si}}. 
    \end{equation*}
    Note that $\E{S_i} = 1/\mu_i$, $\E{S_j}=1/\mu_j$, so the last equality simplifies into
    \begin{equation}
         \label{eq:pf-Rsj-Rsi}
         \Ep{\pi}{R_{\sj} + R_{\si}} =  \Ep{\si}{R_{\sj}} + \Ep{\sj}{R_{\si}}.
    \end{equation}
    By \eqref{eq:E-pi-R-s-het}, the left-hand side of \eqref{eq:pf-Rsj-Rsi} equals $\mu_i\Eplain{S_i^2}/2 +  \mu_j\Eplain{S_j^2}/2$. This proves \eqref{eq:E-s-R-s-het}. 

    \noindent \textbf{Proof of \eqref{eq:E-sa-R-sa-het}.} For any $\kappa > 0$, we take the test function $f(\Xsuper) = (R_{\si} \wedge \kappa)(R_a \wedge \kappa)$ in \eqref{eq:bar-het}. The rest of the proof is identical to the proof of \eqref{eq:E-s-R-s-het} except that $R_{\sj}$ is replaced by $R_a$. 

    \noindent \textbf{Proof of \eqref{eq:E-s-Q-het}.} For any integer $\kappa > 0$, taking the test function $f(\Xsuper) = Q\wedge\kappa$ in \eqref{eq:bar-het} yields 
    \begin{align*}
        \Lambda \Ep{a}{(Q+1)\wedge\kappa - Q\wedge\kappa} + \mutotal\Ep{s}{(Q-1)^+\wedge\kappa - Q\wedge\kappa} &= 0, \\
        \implies \quad \Lambda \Ep{a}{\indibrac{Q\leq \kappa-1}} - \mutotal \Ep{s}{\indibrac{0<Q\leq \kappa}} &= 0.
    \end{align*}
    Because each term in the expectations is bounded, letting $\kappa\to\infty$ and invoking the bounded convergence theorem, we get
    \begin{equation}
        \Lambda - \mutotal \Probp{s}{Q > 0} = 0.
    \end{equation}
    Rearranging the terms yields \eqref{eq:E-s-Q}.

    \noindent \textbf{Proof of \eqref{eq:E-s-Qtilde-het}.} This proof is identical to the proof of \eqref{eq:E-s-Q-het} except that $Q$ is replaced by $\Qdropj$.  
    \qedhere
\end{proof}

Next, we restate \Cref{lem:expfinite} and its generalization in the heterogeneous setting \Cref{lem:expfinite-het}. We will directly prove \Cref{lem:expfinite-het}.

\expfinite*

\expfinitehet*

\begin{proof}[Proof of \Cref{lem:expfinite-het}]
    For any positive integer $\kappa$, we consider the test function $f(\Xsuper) = (Q\wedge\kappa - 2\Lambda R_a\wedge\kappa)^2$, which is bounded, locally Lipschitz continuous, and has right partial derivatives. Consequently, Basic Adjoint Relationship \eqref{eq:bar-het} implies that 
    \begin{align}
        0 
        \label{eq:pf-exp-finite-cont-term}
        &= \Ep{\pi}{2(Q\wedge\kappa - 2\Lambda R_a\wedge\kappa) \cdot (2\Lambda\indibrac{R_a \leq \kappa})} \\
        \label{eq:pf-exp-finite-arr-term}
        &\mspace{23mu} + \Lambda \Ep{a}{ \big((Q+1)\wedge\kappa - 2\Lambda A\wedge\kappa\big)^2 - (Q\wedge\kappa)^2} \\
        \label{eq:pf-exp-finite-comp-term}
        &\mspace{23mu} + \mutotal \Ep{s}{\big((Q-\indibrac{Q>0})\wedge\kappa - 2\Lambda R_a\wedge\kappa\big)^2 - (Q\wedge\kappa - 2\Lambda R_a\wedge\kappa)^2},
    \end{align}
    where to get the term in \eqref{eq:pf-exp-finite-arr-term}, we have used the fact that $R_a = 0$ with probability $1$ under $\Pnb_a$. 
    
    We first simplify the term in \eqref{eq:pf-exp-finite-arr-term}: substituting $(Q+1)\wedge\kappa = Q\wedge\kappa + \indibrac{Q \leq \kappa - 1}$ and rearranging the terms, we get
    \begin{align*}
        &\big((Q+1)\wedge\kappa - 2\Lambda A\wedge\kappa\big)^2 - (Q\wedge\kappa)^2 \\
        &\qquad = 2 (Q\wedge\kappa) \cdot (\indibrac{Q \leq \kappa-1} - 2\Lambda A\wedge\kappa) + (\indibrac{Q\leq\kappa-1} - 2\Lambda A\wedge\kappa)^2.
    \end{align*}
    Because $A$ is independent of $Q$ under the measure $\Pnb_a$, we have
    \begin{align*}
        &\Ep{a}{ \big((Q+1)\wedge\kappa - 2\Lambda A\wedge\kappa\big)^2 - (Q\wedge\kappa)^2} \\
         &\qquad = 2 \Ep{a}{(Q\wedge\kappa) \indibrac{Q \leq \kappa-1}} - 4\Ep{a}{Q\wedge\kappa}\E{ \Lambda A\wedge\kappa} \\
         &\qquad \mspace{23mu} + \Ep{a}{(\indibrac{Q\leq\kappa-1} - 2\Lambda A\wedge\kappa)^2}. 
    \end{align*}

    We then simplify the term in \eqref{eq:pf-exp-finite-comp-term}. Substituting $(Q-\indibrac{Q>0})\wedge\kappa = Q\wedge\kappa - \indibrac{0 < Q \leq \kappa}$ and rearranging the terms, we get 
    \begin{align*}
        &\Ep{s}{\big((Q-\indibrac{Q>0})\wedge\kappa - 2\Lambda R_a\wedge\kappa\big)^2 - (Q\wedge\kappa - 2\Lambda R_a\wedge\kappa)^2} \\
        &\qquad = \Ep{s}{-2(Q\wedge\kappa - 2\Lambda R_a\wedge\kappa) \indibrac{0<Q\leq\kappa} +  \indibrac{0<Q\leq\kappa}} \\
        &\qquad = - 2\Ep{s}{(Q\wedge\kappa)\indibrac{0<Q\leq\kappa}}
        + 4\Ep{s}{(\Lambda R_a\wedge\kappa) \indibrac{0<Q\leq\kappa}} 
        +  \Probp{s}{0<Q\leq\kappa}.
    \end{align*}

    Substituting the above calculations back into \eqref{eq:pf-exp-finite-arr-term} and \eqref{eq:pf-exp-finite-comp-term}, we get 
    \begin{align*}
        0
        &= \Ep{\pi}{2(Q\wedge\kappa - 2\Lambda R_a\wedge\kappa) \cdot (2\Lambda\indibrac{R_a \leq \kappa})} + 2 \Lambda \Ep{a}{(Q\wedge\kappa) \indibrac{Q \leq \kappa-1}} \\
        &\mspace{23mu}  - 4\Lambda\Ep{a}{Q\wedge\kappa}\E{ \Lambda A\wedge\kappa}  + \Lambda\Ep{a}{(\indibrac{Q\leq\kappa-1} - 2\Lambda A\wedge\kappa)^2} \\
        &\mspace{23mu}
        - 2\mutotal\Ep{s}{(Q\wedge\kappa)\indibrac{0<Q\leq\kappa}}
        + 4\mutotal\Ep{s}{(\Lambda R_a\wedge\kappa) \indibrac{0<Q\leq\kappa}} 
        +  \mutotal\Probp{s}{0<Q\leq\kappa}. 
    \end{align*}
    Moving the terms involving $\Ep{a}{Q\wedge\kappa}$ and $\Ep{s}{Q\wedge\kappa}$ to the left-hand side, we get 
    \begin{align}
        \label{eq:pf-exp-finite-intermediate-1}
        & - 2 \Lambda \Ep{a}{(Q\wedge\kappa) \indibrac{Q \leq \kappa-1}} + 4\Lambda\Ep{a}{Q\wedge\kappa}\E{ \Lambda A\wedge\kappa}  + 2\mutotal\Ep{s}{(Q\wedge\kappa)\indibrac{0<Q\leq\kappa}} \\
        \nonumber
        &\qquad = \Ep{\pi}{2(Q\wedge\kappa - 2\Lambda R_a\wedge\kappa) \cdot (2\Lambda\indibrac{R_a \leq \kappa})}
        + \Lambda\Ep{a}{(\indibrac{Q\leq\kappa-1} - 2\Lambda A\wedge\kappa)^2}  \\
        \nonumber
        &\qquad \mspace{23mu}
        +  4\mutotal\Ep{s}{(\Lambda R_a\wedge\kappa) \indibrac{0<Q\leq\kappa}}
        +  \mutotal\Probp{s}{0<Q\leq\kappa} \\
        \nonumber
        &\qquad \leq 4\Lambda \Ep{\pi}{Q} + \Lambda\E{(2\Lambda A)^2} + \Lambda + 4\mutotal\Ep{s}{\Lambda R_a} + \mutotal. 
    \end{align}
    Note that each term in the last expression is finite. In particular, $\Ep{\pi}{Q} < \infty$ due to \Cref{assump:modified-ggn-stable-het}, 
    and $\Ep{s}{\Lambda R_a} < \infty$ due to the following argument:
    \begin{align*}
        \Ep{s}{\Lambda R_a} &= \frac{1}{\mutotal}\sumi \mu_i \Ep{\si}{\Lambda R_a} \\
        &\leq \frac{\Lambda}{\mutotal} \sumi \mu_i \big(\Ep{\si}{R_a} + \Epplain{a}{R_{\si}} \big) \\
        &= \frac{\Lambda}{\mutotal} \sumi \frac{\mu_i}{2} \Big(\mu_i\E{S_i^2} + \Lambda \E{A^2}\Big)\\
        &< \infty,
    \end{align*}
    where the last equality follows from \eqref{eq:E-sa-R-sa-het} of \Cref{lem:apply-bar-lots-of-het}. 
    Rearranging the terms in \eqref{eq:pf-exp-finite-intermediate-1} and using the fact that $\indibrac{Q \leq \kappa-1} \leq 1$, we get 
    \begin{equation}
        \label{eq:pf-exp-finite-intermediate-2} 
        \begin{aligned}
        &\Lambda \Ep{a}{Q\wedge\kappa} \Big(4\E{ \Lambda A\wedge\kappa} - 2\Big)  + 2\mutotal\Ep{s}{(Q\wedge\kappa)\indibrac{0<Q\leq\kappa}} \\
        &\qquad \leq 4\Lambda \Ep{\pi}{Q} + \Lambda\E{(2\Lambda A)^2} + \Lambda + 4\mutotal\Ep{s}{\Lambda R_a} + \mutotal \\
        &\qquad < \infty.
        \end{aligned}
    \end{equation}

    Now we try to take $\kappa \to\infty$ in \eqref{eq:pf-exp-finite-intermediate-2}. 
    Because $\lim_{\kappa\to\infty}\E{\Lambda A\wedge \kappa} = \Lambda \E{A} = 1$, there exists $\kappa_0$ such that for all $\kappa > \kappa_0$, we have $\E{\Lambda A\wedge \kappa} > 1/2$ and $4\E{ \Lambda A\wedge\kappa} - 2 > 0$. Because $Q\wedge\kappa$ and $(Q\wedge\kappa)\indibrac{0<Q\leq\kappa}$ are both non-negative and non-decreasing in $\kappa$, by the monotone convergence theorem and \eqref{eq:pf-exp-finite-intermediate-2}, we obtain
    \begin{align}
        \nonumber
       & 2\Lambda\Ep{a}{Q} + 2\mutotal\Ep{s}{Q} \\
       \nonumber
       &\qquad = \lim_{\kappa\to\infty} \left(\Lambda\Ep{a}{Q\wedge\kappa} \Big(4\E{ \Lambda A\wedge\kappa} - 2\Big)  + 2\mutotal\Ep{s}{(Q\wedge\kappa)\indibrac{0<Q\leq\kappa}}\right) \\
       \label{eq:pf-exp-finite-intermediate-3}
       &\qquad \leq 4\Lambda \Ep{\pi}{Q} + \Lambda\E{(2\Lambda A)^2} + \Lambda + 4\mutotal\Ep{s}{\Lambda R_a} + \mutotal \\
       &\qquad < \infty.
    \end{align}
    Recall that $\Ep{s}{\cdot} = \sumi \mu_i\Ep{\si}{\cdot} / \mutotal$, so \eqref{eq:pf-exp-finite-intermediate-3} implies that $\Ep{a}{Q} < \infty$ and $\Ep{\si}{Q} < \infty$ for all $i\in[n]$. This completes the proof. \qedhere

\end{proof}